\documentclass[12pt]{article}

\usepackage{amssymb}
\usepackage{amsmath}
\usepackage{amsthm}
\usepackage{txfonts}
\usepackage{enumerate}
\usepackage[dvipdfmx]{color}%ishitani_add
\usepackage[top=30truemm, bottom=30truemm, left=20truemm, right=20truemm]{geometry}
\usepackage{comment}

\DeclareMathOperator*{\esssup}{ess\,sup}

\newtheorem{thm}{Theorem}[section]
\newtheorem{lem}{Lemma}[section]
\newtheorem{cor}{Corollary}[section]
\newtheorem{rem}{\rm REMARK}[section]
\newtheorem{prop}{Proposition}[section]

\title{{\huge Construction and sample path properties of diffusion house-moving between two curves}}
\author{Kensuke Ishitani
\thanks{Department of Mathematical Sciences, Tokyo Metropolitan University, Hachioji, Tokyo 192-0397, Japan} 
\thanks{E-mail: k-ishitani@tmu.ac.jp
\qquad https://orcid.org/0000-0001-8400-0654}
\and 
Soma Nishino
\footnotemark[1]}
\date{}
\begin{document}
\maketitle

\begin{abstract}
This study aims to construct a stochastic process called 
``diffusion house-moving,'' which is a diffusion bridge 
conditioned to stay between its starting point and its ending point. 
To construct this process, statements are prepared 
on the weak convergence of conditioned diffusion processes, and conditioned diffusion bridges.
Moreover, the sample path properties of diffusion house-moving are studied as well.   

\end{abstract}

%%%%%%%%%%%%%%%%%%%%%%%%%%%%%%%%%%%
\section{Introduction}
%%%%%%%%%%%%%%%%%%%%%%%%%%%%%%%%%%%

Recently, \cite{bib_Ishitani} developed 
a chain rule for Wiener path integrals between two curves 
that arise in the computation of 
Greeks for barrier options under the Black-Scholes model, 
and demonstrated the effectiveness of this chain rule through numerical examples. 
In this chain rule, Brownian meander, BES($3$)-bridge and Brownian house-moving (\cite{bib_ishitani_hatakenaka_suzuki})
expected to play an important role. 
These stochastic processes also played an important role in 
\cite{bib_FunakiIshitani, bib_ishitani_nishino}.  
Furthermore, we are currently investigating chain rules 
for computing Greeks for barrier options under diffusion market models, 
and we expect a stochastic process called 
``diffusion meander'', 
``conditioned diffusion-bridge''  and  
``diffusion house-moving''  
to play an important role in their computation. 
The purpose of this study is to construct these stochastic processes.

The remainder of this paper is organized as follows. 
In Section~\ref{sec_notation}, we present the notation used in this study. 
Section~\ref{section_Mainresults} states the main results of this study. 
In Subsection~\ref{subsection_Moving_between_2curves}, 
we construct diffusion house-moving (Theorem~\ref{Thm_Def_and_Decomp_curved_Moving}).
In addition, the sample path properties of 
diffusion house-moving (Corollaries~\ref{Cor_30_curve_max_0t}, \ref{Cor_30_curve_min_t1} and Theorem~\ref{Thm_abs_conti}) 
are provided in this subsection. 
In Subsection~\ref{subsection_Meander_between_2curves}, 
we construct diffusion meander between two curves. 
In Subsection~\ref{subsection_diffbridge_between_2curves}, 
we construct conditioned diffusion-bridge between two curves. 
Sections~\ref{section_proof_main_Lemma}, 
\ref{section_proof_Preparation_for_Proofs}, 
\ref{section_proof_Moving_between_2curves}, 
\ref{section_proof_Meander_between_2curves}, and
\ref{section_proof_conddiffbridge_between_2curves}
are devoted to proving the main results in Section~\ref{section_Mainresults}.
%%%Section~\ref{Sec_conclusion_future_work} presents an application of Brownian house-moving and concludes.

%%%%%%%%%%%%%%%%%%%%%%%%%%%%%%%%%%%
\section{Notation}\label{sec_notation}
%%%%%%%%%%%%%%%%%%%%%%%%%%%%%%%%%%%

For $0\leq s < t\leq 1$, let $C([s,t], \mathbb{R})$ 
be the class of $\mathbb{R}$-valued continuous functions 
defined on $[s,t]$, and let 
$$d_{\infty}(w,w') = \sup_{u\in [s,t]} \left| w(u) - w'(u) \right|
\quad (w, w' \in C([s,t], \mathbb{R})). $$ 
$\mathcal{B}(C([s,t], \mathbb{R}))$ denotes the Borel $\sigma$-algebra 
with respect to the topology generated by the metric $d_{\infty}$. 
In addition, for $0 \leq s < t \leq 1$, 
$\pi_{[s, t]} : C([0,1], \mathbb{R}) \to C([s,t], \mathbb{R})$ 
denotes the restriction map.

Assume that 
$V : (\Omega, \mathcal{F}, P) \to 
(C([0,1], \mathbb{R}), \mathcal{B}(C([0,1], \mathbb{R})))$ 
is a random variable and 
$\Lambda \in \mathcal{B}(C([0,1], \mathbb{R}))$ 
satisfies $P(V \in \Lambda) > 0$. 
Then, we define the probability measure 
$P_{V^{-1}(\Lambda)}$ on 
$(V^{-1}(\Lambda), V^{-1}(\Lambda) \cap \mathcal{F})$ as
\begin{align*}
P_{V^{-1}(\Lambda)}(A):= \frac{P(A)}{P(V \in \Lambda)}, 
\qquad A \in V^{-1}(\Lambda) \cap \mathcal{F}
:= \left\{ V^{-1}(\Lambda) \cap F~|~F \in \mathcal{F} \right\}.
\end{align*}
Let $V|_{\Lambda}$ denote the restriction of $V$ to 
$(V^{-1}(\Lambda), V^{-1}(\Lambda) \cap \mathcal{F}, P_{V^{-1}(\Lambda)})$. 
Then, 
\begin{align}\label{Def_conditional_StochProc}
V|_{\Lambda} : (V^{-1}(\Lambda), V^{-1}(\Lambda) \cap \mathcal{F}, P_{V^{-1}(\Lambda)}) \to 
(\Lambda, \mathcal{B}(\Lambda))
\end{align}
is a random variable. Throughout this study, 
$P_{V^{-1}(\Lambda)}(V|_{\Lambda} \in \Gamma)$ is often written as 
$P(V|_{\Lambda} \in \Gamma)$, and
$E^{P_{V^{-1}(\Lambda)}}[f(V|_{\Lambda})]$ is often written as $E[f(V|_{\Lambda})]$.

For $s>0$, we define
\begin{align*}
n_s(x) := \frac{1}{\sqrt{2 \pi s}} \exp \left( -\frac{x^2}{2s} \right) \qquad (x\in \mathbb{R}).
\end{align*}

$V_n \overset{\mathcal{D}}{\to} V$ denotes the convergence in distribution 
of the sequence of random variables $\{ V_n \}_{n=1}^{\infty}$ to the random variable $V$.  
In addition, we write $V \overset{\mathcal{D}}{=} U$ 
for random variables $V$, $U$ that follow the same distribution.

Let $0 \leq t_1 < t_2 \leq 1$. 
Throughout this study, we use the following notation. 

For $f, g \in C([0, 1], \mathbb{R})$, we define
\begin{align*}
&K_{[t_1,t_2]}(f,g) := 
\{ w = \{ w(t) \}_{t \in [t_1,t_2]} \in C([t_1,t_2], \mathbb{R} )
~|~f(t) \leq w(t) \leq g(t),\ t_1 \leq t \leq t_2 \},\\
&K_{[t_1,t_2]}^+(f) := \bigcup_{n=1}^{\infty} K_{[t_1,t_2]}(f, n),
\qquad K_{[t_1,t_2]}^-(g) := \bigcup_{n=1}^{\infty} K_{[t_1,t_2]}(-n, g). 
\end{align*}

For an $\mathbb{R}$-valued process $V=\{ V(t) \}_{t\in [0,1]}$, 
we write its maximal and minimal values as
\begin{align*}
&M_{[t_1,t_2]}(V) = \max_{t_1 \leq u \leq t_2}V(u), \qquad
m_{[t_1,t_2]}(V) = \min_{t_1 \leq u \leq t_2}V(u).
\end{align*}
Moreover, the natural filtration 
$\sigma(V(s) \mid 0 \leq s \leq t)$ of $V$ is denoted by $\mathcal{F}_t^V$.

$W = \{ W(t) \}_{t \geq 0}$, 
$B^{a \to b} = \{ B^{a \to b}(t) \}_{t \in [0,1]}$ $(a, b \in \mathbb{R})$, 
$W^+ = \{ W^+(t) \}_{t \in [0,1]}$, 
and $r^{c \to d} = \{ r^{c \to d}(t) \}_{t \in [0,1]}$ $(c, d \geq 0)$ denote 
standard one-dimensional Brownian motion, 
one-dimensional Brownian bridge from $a$ to $b$ on the time interval $[0,1]$, 
Brownian meander on the time interval $[0,1]$, and 
BES($3$)-bridge from $c$ to $d$ on the time interval $[0,1]$ 
defined on some probability space, respectively.  
For $a, b \in \mathbb{R}$ and $c, d \geq 0$, 
$W_{[t_1,t_2]}$, $W_{[t_1,t_2]}^a$, $B_{[t_1,t_2]}^{a \to b}$, $W_{[t_1,t_2]}^+$ and $r_{[t_1,t_2]}^{c \to d}$ denote 
one-dimensional Brownian motion, 
one-dimensional Brownian motion from $a$, 
one-dimensional Brownian bridge from $a$ to $b$, 
Brownian meander, and 
BES($3$)-bridge from $c$ to $d$ defined on $[t_1,t_2]$, respectively. 
Laws of $W_{[t_1,t_2]}$, $W_{[t_1,t_2]}^a$, $B_{[t_1,t_2]}^{a \to b}$, $W_{[t_1,t_2]}^+$, and $r_{[t_1,t_2]}^{c \to d}$ are given by
\begin{align*}
&\{ W_{[t_1,t_2]}(u) \}_{u \in [t_1, t_2]}
\overset{\mathcal{D}}{=} 
\{ W(u-t_1) \}_{u \in [t_1, t_2]}, \quad
\{ W_{[t_1,t_2]}^a(u) \}_{u \in [t_1, t_2]}
\overset{\mathcal{D}}{=} 
\{ a+W(u-t_1) \}_{u \in [t_1, t_2]}, \\
&\left\{ B_{[t_1,t_2]}^{a \to b}(u) \right\}_{u \in [t_1, t_2]}
\overset{\mathcal{D}}{=}
\left\{ \sqrt{t_2-t_1}B^{\frac{a}{\sqrt{t_2-t_1}} \to \frac{b}{\sqrt{t_2-t_1}}}\left(\frac{u - t_1}{t_2-t_1}\right) \right\}_{u \in [t_1, t_2]}, \\
&\left\{ W_{[t_1,t_2]}^+(u) \right\}_{u \in [t_1, t_2]}
\overset{\mathcal{D}}{=}
\left\{ \sqrt{t_2-t_1}W^+\left(\frac{u - t_1}{t_2-t_1}\right) \right\}_{u \in [t_1, t_2]}, \\
&\left\{ r_{[t_1,t_2]}^{c \to d}(u) \right\}_{u \in [t_1, t_2]}
\overset{\mathcal{D}}{=}
\left\{ \sqrt{t_2-t_1}r^{\frac{c}{\sqrt{t_2-t_1}} \to \frac{d}{\sqrt{t_2-t_1}}}\left(\frac{u - t_1}{t_2-t_1}\right) \right\}_{u \in [t_1, t_2]}.
\end{align*}

For an $\mathbb{R}$-valued continuous process $V$ on $[t_1,t_2]$ and 
$\mathbb{R}$-valued $C^2$-function $g$ defined on $[t_1,t_2]$, we define
\begin{align*}
&Z_{[t_1,t_2]}^g (V) := 
\exp \left\{ g'(t_2)V(t_2) - g'(t_1)V(t_1) 
-\int_{t_1}^{t_2} V(u)g^{\prime \prime}(u) du 
- \frac{1}{2} \int_{t_1}^{t_2} g'(u)^2 du \right\}, \\
&\widetilde{Z}_{[t_1,t_2]}^{g}(V) := Z_{[t_1,t_2]}^{g}(V + g) \\
&=\exp \left\{ g'(t_2)V(t_2) - g'(t_1)V(t_1) - \int_{t_1}^{t_2} V(u) g''(u) \ du + \frac{1}{2} \int_{t_1}^{t_2} g'(u) ^2 \ du \right\} .
\end{align*}
Then, we obtain
\begin{align*}
&Z_{[t_1,t_2]}^g (V)
=Z_{[t_1,t_2]}^{-g} (-V)
=Z_{[t_1,t_2]}^{g+a} (V+b), \\
&\widetilde{Z}_{[t_1,t_2]}^{g}(V)
=\widetilde{Z}_{[t_1,t_2]}^{-g}(-V)
=\widetilde{Z}_{[t_1,t_2]}^{g+a}(V+b)
\qquad (a, b\in \mathbb{R}).
\end{align*}
If $V$ is $W_{[t_1,t_2]}$, then it follows from It\^{o}'s formula that
\begin{align*}
Z_{[t_1,t_2]}^g (W_{[t_1,t_2]}) = 
\exp \left\{ \int_{t_1}^{t_2} g'(u) dW_{[t_1,t_2]}(u) 
- \frac{1}{2} \int_{t_1}^{t_2} g'(u)^2 du \right\}.
\end{align*}

For $f \in C([t_1,t_2], \mathbb{R})$, 
we define $\overleftarrow{f}\in C([t_1,t_2], \mathbb{R})$ as
\begin{align*}
\overleftarrow{f}(t):=f(t_1+t_2-t), \quad t_1 \leq t \leq t_2.
\end{align*}

Let $0 \leq t_0 < t_1 < t_2 < t_3 < \infty$.
For $w_{[t_0,t_1]}^{(1)} \in C([t_0,t_1], \mathbb{R})$ and $w_{[t_1,t_2]}^{(2)} \in C([t_1,t_2], \mathbb{R})$, 
we define $w_{[t_0,t_1]}^{(1)} \oplus_{t_1} w_{[t_1,t_2]}^{(2)}$ as
\begin{equation}\nonumber
(w_{[t_0,t_1]}^{(1)} \oplus_{t_1} w_{[t_1,t_2]}^{(2)} )(t) :=
\begin{cases}
w_{[t_0,t_1]}^{(1)}(t), & t_0 \leq t < t_1 , \\
w_{[t_1,t_2]}^{(2)}(t), & t_1 \leq t \leq t_2 .
\end{cases}
\end{equation}
Similarly, for $w_{[t_{i-1},t_{i}]}^{(i)} \in C([t_{i-1},t_{i}], \mathbb{R})$ ($i = 1, 2, 3$), 
we define $w_{[t_0,t_1]}^{(1)} \oplus_{t_1} w_{[t_1,t_2]}^{(2)} \oplus_{t_2} w_{[t_2,t_3]}^{(3)}$ as
\begin{equation}\nonumber
(w_{[t_0,t_1]}^{(1)} \oplus_{t_1} w_{[t_1,t_2]}^{(2)} \oplus_{t_2} w_{[t_2,t_3]}^{(3)} )(t) :=
\begin{cases}
w_{[t_0,t_1]}^{(1)}(t), & t_0 \leq t < t_1 , \\
w_{[t_1,t_2]}^{(2)}(t), & t_1 \leq t < t_2 , \\
w_{[t_2,t_3]}^{(3)}(t), & t_2 \leq t \leq t_3 .
\end{cases}
\end{equation}

%%%%%%%%%%%%%%%%%%%%%%%%%%%%%%
\section{Main results}\label{section_Mainresults}
%%%%%%%%%%%%%%%%%%%%%%%%%%%%%%

Let $\nu$ be an $\mathbb{R}$-valued $C^1$-function defined on $\mathbb{R}$. 
Let $\sigma$ be an $\mathbb{R}_{>0}$-valued $C^1$-function defined on $\mathbb{R}$. 
For every $c \in \mathbb{R}$, we assume that there exists a unique strong solution of the following SDE:
\begin{align}\label{X_SDE}
dX(t) = \nu(X(t)) \, dt + \sigma(X(t)) dW(t), \quad X(0) = c.
\end{align}
Let $X^c = \{ X^c(t) \}_{t \geq 0}$ denotes the solution of \eqref{X_SDE}. 
We define
\begin{align*}
&\Phi(x) := \int_0^x \frac{1}{\sigma(u)} du
\qquad (x \in \mathbb{R}), \\
&\mu(y) := \left( \frac{\nu}{\sigma} - \frac{1}{2} \sigma' \right) \circ \Phi^{-1}(y) 
= \frac{ \nu( \Phi^{-1}(y) ) }{ \sigma( \Phi^{-1}(y) ) } - \frac{1}{2} \sigma'( \Phi^{-1}(y) )
\qquad (y \in \Phi(\mathbb{R})).
\end{align*}
We define $Y^{\Phi(c)} = \{ Y^{\Phi(c)}(t) \}_{t \geq 0}$ as $Y^{\Phi(c)}(t) = \Phi(X^{c}(t))$ $(t\geq 0)$. 
Then, from It\^o's formula, $Y^{\Phi(c)}$ is the strong solution of the following SDE:
\begin{align*}
dY(t) = \mu(Y(t)) dt + dW(t), \quad Y(0) = \Phi(c).
\end{align*}
For $t > 0$, we define the transition densities
\begin{align*}
&q_X(t,x,y):=\frac{P(X^x(t)\in dy)}{dy}\qquad (x, y\in \mathbb{R}), \\
&q_Y(t,x,y):=\frac{P(Y^x(t)\in dy)}{dy}\qquad (x\in \Phi(\mathbb{R}), \ y\in \mathbb{R}), \\
&q_W(t,x,y):=\frac{P(x+W(t)\in dy)}{dy}=n_t(y-x)\qquad (x, y\in \mathbb{R}).
\end{align*}
Let $0 \leq t_1 < t_2<\infty$. 
For $f \in C([t_1,t_2], \mathbb{R} )$, we define $\phi(f)\in C([t_1,t_2], \mathbb{R} )$ as
\begin{align*}
(\phi (f))(t):=\Phi(f(t))=\int_0^{f(t)} \frac{1}{\sigma(u)} du, \qquad t_1\leq t \leq t_2.
\end{align*}
Note that 
\begin{align*}
\phi(C^1([t_1,t_2], \mathbb{R} ))\subset C^1([t_1,t_2], \mathbb{R} )
\quad \mbox{and}\quad
\phi(C^2([t_1,t_2], \mathbb{R} ))\subset C^2([t_1,t_2], \mathbb{R} ).
\end{align*}
For $x\in \mathbb{R}$ and $y\in \Phi(\mathbb{R})$,  
we define 
$X_{[t_1,t_2]}^x=\{X_{[t_1,t_2]}^x(u)\}_{u\in [t_1,t_2]}$ and
$Y_{[t_1,t_2]}^y=\{Y_{[t_1,t_2]}^y(u)\}_{u\in [t_1,t_2]}$ as
\begin{align*}
X_{[t_1,t_2]}^x(u)=X^x(u-t_1), \quad 
Y_{[t_1,t_2]}^y(u)=Y^y(u-t_1),\qquad u\in [t_1,t_2].
\end{align*}
For $x_1, x_2\in \mathbb{R}$ and $y_1, y_2\in \Phi(\mathbb{R})$, 
$X_{[t_1,t_2]}^{x_1 \to x_2}=\{X_{[t_1,t_2]}^{x_1 \to x_2}(u)\}_{u\in [t_1,t_2]}$ and 
$Y_{[t_1,t_2]}^{y_1 \to y_2}=\{Y_{[t_1,t_2]}^{y_1 \to y_2}(u)\}_{u\in [t_1,t_2]}$
denote $X$-bridge from $x_1$ to $x_2$ defined on the time interval $[t_1,t_2]$, 
and $Y$-bridge from $y_1$ to $y_2$ defined on the time interval $[t_1,t_2]$, respectively.  
Note that $X^{c \to d}_{[t_1,t_2]}\overset{\mathcal{D}}{=} \phi^{-1}( Y^{\Phi(c) \to \Phi(d)}_{[t_1,t_2]} )$ holds for $c, d\in \mathbb{R}$. 
Further, we define
\begin{align*}
&N_{[t_1,t_2]}(w) := \int_{t_1}^{t_2} \left\{ \mu'(w(u)) + \mu^2(w(u)) \right\} \, du
\qquad (w\in C([t_1,t_2], \mathbb{R})), \\
&\Psi(y) := \int_0^y \mu(z) \, dz 
\qquad (y\in \Phi(\mathbb{R})).
\end{align*}

Let $g^-$ and $g^+$ be 
$\mathbb{R}$-valued $C^2$-functions defined on $[0, 1]$ that satisfy
\begin{align*}
\min_{0 \leq t \leq 1}(g^+(t) - g^-(t)) > 0.
\end{align*} 

Let $0 \leq t_1 < t_2 \leq 1$. 
According to the values $g^-(t_1) \leq a \leq g^+(t_1)$ 
and $g^-(t_2) \leq b \leq g^+(t_2)$, 
the continuous process $X_{[t_1,t_2]}^{a, b,(g^-, g^+)}$ on $[t_1,t_2]$ 
is defined as follows 
(see also \eqref{Def_conditional_StochProc}, Lemma~\ref{Lem_for_H_gm_to_gp} below):
\begin{itemize}
\item in the case $a=g^-(t_1)$, $g^-(t_2) \leq b < g^+(t_2)$, the weak limit of 
$X_{[t_1,t_2]}^{a \to b} |_{K_{[t_1,t_2]}(g^--\varepsilon, g^+)}$ 
as $\varepsilon \downarrow 0$;
\item in the case $g^-(t_1)<a \leq g^+(t_1)$, $b = g^+(t_2)$, the weak limit of 
$X_{[t_1,t_2]}^{a \to b} |_{K_{[t_1,t_2]}(g^-, g^++\varepsilon)}$ 
as $\varepsilon \downarrow 0$;
\item in the case $a=g^-(t_1)$, $b = g^+(t_2)$, the weak limit of 
$X_{[t_1,t_2]}^{a \to b} |_{K_{[t_1,t_2]}(g^--\varepsilon, g^++\varepsilon)}$ 
as $\varepsilon \downarrow 0$;
\item in the case $g^-(t_1) < a < g^+(t_1), g^-(t_2) < b < g^+(t_2)$, 
the conditioned process $X_{[t_1,t_2]}^{a \to b} |_{K_{[t_1,t_2]}(g^-, g^+)}$.
\end{itemize}
In addition, according to the value $g^-(t_1) \leq a \leq g^+(t_1)$, the continuous process 
$X_{[t_1,t_2]}^{a, (g^-, g^+)}$ on $[t_1,t_2]$ is defined as follows 
(see also \eqref{Def_conditional_StochProc} and Lemma~\ref{Lem_WeakConv_Meander_btwn_2crvs} below):
\begin{itemize}
\item in the case $a = g^-(t_1)$, the weak limit of 
$X_{[t_1,t_2]}^a |_{K_{[t_1,t_2]}(g^--\varepsilon, g^+)}$ 
as $\varepsilon \downarrow 0$; 
\item in the case $a = g^+(t_1)$, the weak limit of 
$X_{[t_1,t_2]}^a |_{K_{[t_1,t_2]}(g^-, g^++\varepsilon)}$ 
as $\varepsilon \downarrow 0$; 
\item in the case $g^-(t_1) < a < g^+(t_1)$, the conditioned process 
$X_{[t_1,t_2]}^a |_{K_{[t_1,t_2]}(g^-, g^+)}$. 
\end{itemize}
In the case $\sigma(x) = 1$ $(x\in \mathbb{R})$, 
$X_{[t_1,t_2]}^{a, b,(g^-, g^+)}$ is denoted by 
$Y_{[t_1,t_2]}^{a, b,(g^-, g^+)}$, and
$X_{[t_1,t_2]}^{a,(g^-, g^+)}$ is denoted by 
$Y_{[t_1,t_2]}^{a,(g^-, g^+)}$, respectively. 
In the case $\nu(x)=0$, $\sigma(x) = 1$ $(x\in \mathbb{R})$, 
$X_{[t_1,t_2]}^{a, b,(g^-, g^+)}$ is denoted by 
$B_{[t_1,t_2]}^{a, b,(g^-, g^+)}$, and
$X_{[t_1,t_2]}^{a,(g^-, g^+)}$ is denoted by 
$W_{[t_1,t_2]}^{a,(g^-, g^+)}$, respectively. 
Note that $B_{[t_1,t_2]}^{a, b,(g^-, g^+)}$ and $W_{[t_1,t_2]}^{a,(g^-, g^+)}$ 
exist and their distributions are given by 
Lemma~3.1, Lemma~3.2 and Theorem~3.1 in \cite{bib_ishitani_hatakenaka_suzuki}.

The following two lemmas are prepared to state our main results 
(Theorems \ref{Thm_Def_and_Decomp_curved_Moving}, 
\ref{Thm_Def_and_Decomp_curved_Meander_two_curve} and 
\ref{Thm_Def_and_Decomp_curved_BESbridge_two_curve}). 
The processes $X_{[t_1,t_2]}^{a, b,(g^-, g^+)}$ and $X_{[t_1,t_2]}^{a, (g^-, g^+)}$ 
in Lemma~\ref{Lem_for_H_gm_to_gp} $(A)$, $(B)$ and Lemma~\ref{Lem_WeakConv_Meander_btwn_2crvs} 
are not our main target of this study, but just auxiliary ones.
The process $X_{[t_1,t_2]}^{a, b,(g^-, g^+)}$ in Lemma~\ref{Lem_for_H_gm_to_gp} $(C)$ and
the process $X^{a, b, (g^-, g^+)}$ in Theorem~\ref{Thm_Def_and_Decomp_curved_Moving}
are our main target of this study.

For $\varepsilon>0$, let $g^-_{\varepsilon}$ and $g^+_{\varepsilon}$ be 
$\mathbb{R}$-valued $C^2$-functions defined on $[0, 1]$ that satisfy
\begin{align*}
&\min_{0 \leq t \leq 1}(g^+_{\varepsilon}(t) - g^-_{\varepsilon}(t)) > 0 \qquad (\varepsilon>0), \\
&\lim_{\varepsilon \downarrow 0}\sup_{0 \leq t\leq 1}\left\{ \vert g^-_{\varepsilon}(t)-g^-(t)\vert 
+ \left\vert \frac{d}{dt}g^-_{\varepsilon}(t) - \frac{d}{dt}g^-(t)\right\vert 
+ \left\vert \frac{d^2}{dt^2}g^-_{\varepsilon}(t) - \frac{d^2}{dt^2}g^-(t)\right\vert \right\}
=0, \\
&\lim_{\varepsilon \downarrow 0}\sup_{0 \leq t\leq 1}\left\{ \vert g^+_{\varepsilon}(t)-g^+(t)\vert 
+ \left\vert \frac{d}{dt}g^+_{\varepsilon}(t) - \frac{d}{dt}g^+(t)\right\vert 
+ \left\vert \frac{d^2}{dt^2}g^+_{\varepsilon}(t) - \frac{d^2}{dt^2}g^+(t)\right\vert \right\}
=0.
\end{align*} 

\begin{lem}\label{Lem_for_H_gm_to_gp}
Let $0 \leq t_1 < t_2 \leq 1$. 
$X_{[t_1,t_2]}^{a, b,(g^-, g^+)}$ exists and its distribution is given as follows. 
For every $\mathbb{R}$-valued bounded continuous function $F$ on $C([t_1,t_2], \mathbb{R})$,
\begin{itemize}
\item[{\rm ({\mbox{\rm A}})}] 
if $a=g^-(t_1) > g^-_{\varepsilon}(t_1)$, $g^-(t_2)\vee g^-_{\varepsilon}(t_2) \leq b < g^+(t_2)$\ \ $(\varepsilon > 0)$, then
\begin{align}
& E\left[ F(X_{[t_1,t_2]}^{a, b,(g^-, g^+)}) \right] 
= \lim_{\varepsilon \downarrow 0}
E\left[ F\left( 
X_{[t_1,t_2]}^{a \to b} |_{K_{[t_1,t_2]}(g^-_{\varepsilon}, g^+_{\varepsilon})}
\right)\right] \nonumber \\
&\quad = \dfrac{E \left[ e^{-\frac{1}{2}N_{[t_1,t_2]}\big(B^{\Phi(a), \Phi(b), (\phi(g^-), \phi(g^+))}_{[t_1,t_2]}\big)} (F \circ \phi^{-1} )\big(B^{\Phi(a), \Phi(b), (\phi(g^-), \phi(g^+))}_{[t_1,t_2]}\big)\right]  }
{E \left[ e^{-\frac{1}{2}N_{[t_1,t_2]}\big(B^{\Phi(a), \Phi(b), (\phi(g^-), \phi(g^+))}_{[t_1,t_2]}\big)} \right]  },
\label{G_bridge_on_t1t2_touch_at_t1}
\end{align}
\item[{\rm ({B})}] 
if $g^-(t_1)<a \leq g^+(t_1)\wedge g^+_{\varepsilon}(t_1)$, $b = g^+(t_2) < g^+_{\varepsilon}(t_2)$\ \ $(\varepsilon > 0)$, then
\begin{align}
& E\left[F(X_{[t_1,t_2]}^{a, b,(g^-, g^+)}) \right] 
= \lim_{\varepsilon \downarrow 0}
E\left[ F\left( 
X_{[t_1,t_2]}^{a \to b} |_{K_{[t_1,t_2]}(g^-_{\varepsilon}, g^+_{\varepsilon})}
\right)\right] \nonumber \\
&\quad = \dfrac{E \left[ e^{-\frac{1}{2}N_{[t_1,t_2]}\big(B^{\Phi(a), \Phi(b), (\phi(g^-), \phi(g^+))}_{[t_1,t_2]}\big)} (F \circ \phi^{-1} )\big(B^{\Phi(a), \Phi(b), (\phi(g^-), \phi(g^+))}_{[t_1,t_2]}\big)\right]  }
{E \left[ e^{-\frac{1}{2}N_{[t_1,t_2]}\big(B^{\Phi(a), \Phi(b), (\phi(g^-), \phi(g^+))}_{[t_1,t_2]}\big)} \right]  }, 
\label{G_bridge_on_t1t2_touch_at_t2}
\end{align}
\item[{\rm ({C})}] 
if $a = g^-(t_1) > g^-_{\varepsilon}(t_1)$ and $b = g^+(t_2) < g^+_{\varepsilon}(t_2)$\ \ $(\varepsilon > 0)$, then
\begin{align}
& E\left[F(X_{[t_1,t_2]}^{a, b,(g^-, g^+)}) \right] 
= \lim_{\varepsilon \downarrow 0}
E\left[ F\left( 
X_{[t_1,t_2]}^{a \to b} |_{K_{[t_1,t_2]}(g^-_{\varepsilon}, g^+_{\varepsilon})}
\right)\right] \nonumber \\
&\quad = \dfrac{E \left[ e^{-\frac{1}{2}N_{[t_1,t_2]}\big(B^{\Phi(a), \Phi(b), (\phi(g^-), \phi(g^+))}_{[t_1,t_2]}\big)} (F \circ \phi^{-1} )\big(B^{\Phi(a), \Phi(b), (\phi(g^-), \phi(g^+))}_{[t_1,t_2]}\big)\right]  }
{E \left[ e^{-\frac{1}{2}N_{[t_1,t_2]}\big(B^{\Phi(a), \Phi(b), (\phi(g^-), \phi(g^+))}_{[t_1,t_2]}\big)} \right]  }.
\label{G_bridge_on_t1t2_touch_at_t1t2}
\end{align}
\end{itemize}
\end{lem}

Here, note that Lemma~\ref{Lem_for_H_gm_to_gp} is necessary 
to state Theorem~\ref{Thm_Def_and_Decomp_curved_Moving}.

\begin{lem}\label{Lem_WeakConv_Meander_btwn_2crvs}
Let $0 \leq t_1 < t_2 \leq 1$. 
$X_{[t_1,t_2]}^{a, (g^-, g^+)}$ 
exists and its distribution is given as follows. 
For every $\mathbb{R}$-valued bounded continuous function $F$ on $C([t_1,t_2], \mathbb{R})$,
\begin{itemize}
\item[{\rm ({A})}] if $a=g^-(t_1) > g^-_{\varepsilon}(t_1)$\ \ $(\varepsilon > 0)$, then
\begin{align}
& E\left[ F(X_{[t_1,t_2]}^{a, (g^-, g^+)}) \right] 
= \lim_{\varepsilon \downarrow 0}
E\left[ F\left( 
X_{[t_1,t_2]}^a |_{K_{[t_1,t_2]}(g^-_{\varepsilon}, g^+_{\varepsilon})}
\right)\right] \nonumber \\
& \quad = 
\dfrac{E \left[ e^{\Psi\big(W^{\Phi(a), (\phi(g^-), \phi(g^+))}_{[t_1,t_2]}(t_2)\big)}
e^{-\frac{1}{2}N_{[t_1,t_2]}\big(W^{\Phi(a), (\phi(g^-), \phi(g^+))}_{[t_1,t_2]}\big)} (F \circ \phi^{-1} )\big(W^{\Phi(a), (\phi(g^-), \phi(g^+))}_{[t_1,t_2]}\big)\right]  }
{E \left[ e^{\Psi\big(W^{\Phi(a), (\phi(g^-), \phi(g^+))}_{[t_1,t_2]}(t_2)\big)} e^{-\frac{1}{2}N_{[t_1,t_2]}\big(W^{\Phi(a), (\phi(g^-), \phi(g^+))}_{[t_1,t_2]}\big)} \right]  },
\label{G_meander_on_t1t2}
\end{align}
\item[{\rm ({B})}] if $a=g^+(t_1) < g^+_{\varepsilon}(t_1)$\ \ $(\varepsilon > 0)$, then
\begin{align}
& E\left[ F(X_{[t_1,t_2]}^{a, (g^-, g^+)}) \right] 
= \lim_{\varepsilon \downarrow 0}
E\left[ F\left( 
X_{[t_1,t_2]}^a |_{K_{[t_1,t_2]}(g^-_{\varepsilon}, g^+_{\varepsilon})}
\right)\right] \nonumber \\
& \quad = 
\dfrac{E \left[ e^{\Psi\big(W^{\Phi(a), (\phi(g^-), \phi(g^+))}_{[t_1,t_2]}(t_2)\big)}
e^{-\frac{1}{2}N_{[t_1,t_2]}\big(W^{\Phi(a), (\phi(g^-), \phi(g^+))}_{[t_1,t_2]}\big)} (F \circ \phi^{-1} )\big(W^{\Phi(a), (\phi(g^-), \phi(g^+))}_{[t_1,t_2]}\big)\right]  }
{E \left[ e^{\Psi\big(W^{\Phi(a), (\phi(g^-), \phi(g^+))}_{[t_1,t_2]}(t_2)\big)} e^{-\frac{1}{2}N_{[t_1,t_2]}\big(W^{\Phi(a), (\phi(g^-), \phi(g^+))}_{[t_1,t_2]}\big)} \right]  } .
\label{G_minus_meander_on_t1t2}
\end{align}
\end{itemize}
\end{lem}

\begin{rem}\label{Rem_F_canbetakento_Indicator_MeanderBESbridge}
Let $A$ be a closed subset of $C([t_1, t_2], \mathbb{R})$, and let
\begin{align*}
&d_{\infty}(w, A):=\inf\{d_{\infty}(w,v) \mid v\in A\} \quad (w\in C([t_1, t_2], \mathbb{R})), \\
&\varphi(x):=1-\int_0^1 1_{(-\infty, x]}(u)du \quad (x\in \mathbb{R}), \\
&F_n(w):=\varphi(nd_{\infty}(w, A))\quad (w\in C([t_1, t_2], \mathbb{R})).
\end{align*}
Then, $F_n$ is a bounded continuous function on $C([t_1, t_2], \mathbb{R})$  
and satisfies 
\[F_n(w)\downarrow 1_A(w), \qquad n \to \infty \]
for $w\in C([t_1, t_2], \mathbb{R})$. 
Thus, the dominated convergence theorem 
implies that Lemmas~\ref{Lem_for_H_gm_to_gp} and \ref{Lem_WeakConv_Meander_btwn_2crvs} hold true for $F=1_A$. 
Let $B\in \mathcal{B}(C([t_1,t_2], \mathbb{R}))$. 
Then, it follows from Dynkin's $\pi$-$\lambda$ theorem that 
Lemmas~\ref{Lem_for_H_gm_to_gp} and \ref{Lem_WeakConv_Meander_btwn_2crvs} hold true for $F=1_B$. 
\end{rem}

Further, we present the notation used in 
Subsections~\ref{subsection_Moving_between_2curves}, 
\ref{subsection_Meander_between_2curves} and
\ref{subsection_diffbridge_between_2curves}.

For $0<t<1$, $0 \leq t_1 < t_2 \leq 1$ and 
$y \in (g^-(t), g^+(t))$, $y_i \in (g^-(t_i), g^+(t_i))$ $(i=1,2)$, we define
\begin{align*}
p^{(g^-,g^+), (\uparrow)}_{[0,t]}(y)
&=\lim_{\varepsilon \downarrow 0}\frac{P\Big(X^{g^-(0)\to y}_{[0, t]} \in K_{[0, t]}(g^--\varepsilon, g^+)\Big)}{\varepsilon}
\cdot \frac{P\Big(X^{g^-(0)}_{[0, t]}(t) \in dy\Big)}{dy}, \\
p^{(g^-,g^+), (\downarrow)}_{[t,1]}(y)
&=\lim_{\varepsilon \downarrow 0}\frac{P\Big(X^{y\to g^+(1)}_{[t, 1]} \in K_{[t, 1]}(g^-, g^+ + \varepsilon)\Big)}{\varepsilon}
\cdot \frac{P(X^y_{[t, 1]}(1) \in dc)}{dc}\bigg\vert_{c=g^+(1)}, \\
p^{(g^-,g^+)}_{[t_1,t_2]}(y_1)&=P(X^{y_1}_{[t_1,t_2]} \in K_{[t_1,t_2]}(g^-, g^+)), \\
p^{(g^-,g^+)}_{[t_1,t_2]}(y_1, y_2)&=\frac{P(X^{y_1}_{[t_1,t_2]} \in K_{[t_1,t_2]}(g^-, g^+), X^{y_1}_{[t_1,t_2]}(t_2) \in dy_2)}{dy_2}\\
&=P(X^{y_1\to y_2}_{[t_1,t_2]} \in K_{[t_1,t_2]}(g^-, g^+))
\cdot \frac{P(X^{y_1}_{[t_1,t_2]}(t_2) \in dy_2)}{dy_2}.
\end{align*}

%%%%%%%%%%%%%%%%%%%%%%%%%%%%%%%%%%%%%%
\subsection{Sample path properties of diffusion house-moving}
\label{subsection_Moving_between_2curves}
%%%%%%%%%%%%%%%%%%%%%%%%%%%%%%%%%%%%%%

In this subsection, we define $a:=g^-(0)$, $b:=g^+(1)$, and assume that
\begin{align*}
a = g^-(0) > g^-_{\varepsilon}(0), \quad b = g^+(1) < g^+_{\varepsilon}(1)
\qquad (\varepsilon > 0).
\end{align*} 

For $0<t<1$, $0<t_1 < t_2<1$ and 
$y \in (g^-(t), g^+(t))$, $y_i \in (g^-(t_i), g^+(t_i))$ $(i=1,2)$, we define
\begin{align*}
&h(t,y)=(C_{g^-,g^+})^{-1}p^{(g^-,g^+), (\uparrow)}_{[0,t]}(y)p^{(g^-,g^+), (\downarrow)}_{[t,1]}(y),\\
&h(t_1, y_1, t_2, y_2) 
=\dfrac{p^{(g^-,g^+)}_{[t_1,t_2]}(y_1, y_2) p^{(g^-,g^+), (\downarrow)}_{[t_2,1]}(y_2)}
{p^{(g^-,g^+), (\downarrow)}_{[t_1,1]}(y_1) } ,
\end{align*}
where
\begin{align}\label{Def_Cgmgp}
C_{g^-,g^+}:=
\lim_{\varepsilon \downarrow 0} 
\frac{P(X_{[0, 1]}^{a \to b} \in K_{[0, 1]}(g^-_{\varepsilon}, g^+_{\varepsilon}))}
{ ( g^-(0) - g^-_{\varepsilon}(0) ) ( g^+_{\varepsilon}(1) - g^+(1) ) }
\cdot
\frac{P(X^a_{[0,1]}(1) \in d\widetilde{b})}{d\widetilde{b}}\bigg\vert_{\widetilde{b}=b} . 
\end{align}
Here, note that $h(t, \cdot)$ becomes a density function on the interval $(g^-(t), g^+(t))$. 

Our aim in this subsection is to prove the existence of the weak limit of 
$X_{[0,1]}^{a \to b} |_{K_{[0, 1]}(g^-_{\varepsilon}, g^+_{\varepsilon})}$ 
as $\varepsilon \downarrow 0$. 
$X^{a, b, (g^-, g^+)}$ denotes this weak limit. 
In the case $\nu(x)=0$, $\sigma(x)=1$ $(x\in \mathbb{R})$, 
$X^{a, b, (g^-, g^+)}$ is the Brownian house-moving $H^{g^-\to g^+}$ introduced in \cite{bib_ishitani_hatakenaka_suzuki}. 
In the case $b>0$, sample paths of $X^{0, b, (0, b)}$  
appear as if they are randomly moving from old house $0$ to new house $b$, see \cite{bib_Ishitani_moving}. 
Therefore, in this study, we call $X^{a, b, (g^-, g^+)}$ ``diffusion house-moving.''

\begin{thm}\label{Thm_Def_and_Decomp_curved_Moving}
There exists an $\mathbb{R}$-valued continuous Markov process 
$X^{a, b, (g^-, g^+)}=\{X^{a, b, (g^-, g^+)}(t)\}_{t\in [0,1]}$ that satisfies
\begin{align}
& E\left[ F(X^{a, b, (g^-, g^+)}) \right]
=\lim_{\varepsilon \downarrow 0} 
E[F(X_{[0,1]}^{a \to b} |_{K_{[0, 1]}(g^-_{\varepsilon}, g^+_{\varepsilon})})]
\label{curvedMoving_def_eq1}\\
&\quad = \int_{g^-(t)}^{g^+(t)} 
E\left[ F(X_{[0,t]}^{a, y, (g^-, g^+)} \oplus_{t} X_{[t,1]}^{y, b, (g^-, g^+)}) \right] h(t,y) dy \label{6.3} \\
&\quad =\int_{g^-(t_1)}^{g^+(t_1)} \int_{g^-(t_2)}^{g^+(t_2)} 
E[F(X_{[0,t_1]}^{a, y_1, (g^-, g^+)} \oplus_{t_1} 
X_{[t_1,t_2]}^{y_1, y_2, (g^-, g^+)}\oplus_{t_2} 
X_{[t_2, 1]}^{y_2, b, (g^-, g^+)})] \label{6.4} \\
&\quad \qquad \qquad \qquad \times 
h(t_1, y_1)h(t_1, y_1, t_2, y_2) dy_1 dy_2 \nonumber
\end{align}
for every $\mathbb{R}$-valued bounded continuous function $F$ on 
$C([0, 1], \mathbb{R})$, $0<t<1$ and $0<t_1<t_2<1$,  
where the respective processes 
that appear in {\rm (\ref{6.3})} and {\rm (\ref{6.4})} are independent of each other. 
Moreover, for $0 < t < 1$, $0<t_1<t_2<1$ and $y\in (g^-(t), g^+(t))$, $y_i \in (g^-(t_i), g^+(t_i))$ $(i=1,2)$, 
the transition densities for $X^{a, b, (g^-, g^+)}$ are given by
\begin{align*}
&P(X^{a, b, (g^-, g^+)}(t) \in dy) = h(t, y)dy, \\
&P(X^{a, b, (g^-, g^+)}(t_2) \in dy_2 \ \vert \ X^{a, b, (g^-, g^+)}(t_1)=y_1) 
=h(t_1, y_1, t_2, y_2) dy_2.
\end{align*}
\end{thm}

\begin{rem}\label{Rem_F_canbetakento_Indicator_Moving}
Let $B\in \mathcal{B}(C([0,1], \mathbb{R}))$ be a measurable subset of $C([0,1], \mathbb{R})$. 
Then, it follows 
from the same argument in Remark~\ref{Rem_F_canbetakento_Indicator_MeanderBESbridge} 
that Theorem~\ref{Thm_Def_and_Decomp_curved_Moving} holds true for $F=1_B$.
\end{rem}

\begin{cor}\label{Cor_30_curve_max_0t}
Let $g$ be an $\mathbb{R}$-valued $C^1$-function defined on $[0, 1]$ that satisfies 
\begin{align*}
g^-(t)<g(t)\leq g^+(t), \quad 0\leq t \leq 1.
\end{align*} 
Then, for $t \in (0,1)$ and $g^-(t)\leq z\leq g(t)$, we have
\begin{align}
&P\left( \min_{u\in [0,t]} \left\{ g(u)-X^{a, b, (g^-, g^+)}(u) \right\} = 0 \right) = 0, 
\label{eq_curve_max_0t_bdry} \\
&P\left( \min_{u\in [0,t]}\left\{ g(u)-X^{a, b, (g^-, g^+)}(u) \right\} \geq 0, X^{a, b, (g^-, g^+)}(t) \leq z \right)
\nonumber \\
&\quad 
=\int_{g^-(t)}^z(C_{g^-,g^+})^{-1}p^{(g^-,g), (\uparrow)}_{[0,t]}(y)p^{(g^-,g^+), (\downarrow)}_{[t,1]}(y)dy.
\label{eq_curve_max_0t_jointdist}
\end{align}
\end{cor}

\begin{cor}\label{Cor_30_curve_min_t1}
Let $g$ be an $\mathbb{R}$-valued $C^1$-function defined on $[0, 1]$ that satisfies 
\begin{align*}
g^-(t)\leq g(t)< g^+(t), \quad 0\leq t \leq 1.
\end{align*} 
Then, for $t \in (0,1)$ and $g(t)\leq z \leq g^+(t)$, we have
\begin{align}
&P\left( \min_{u\in [t,1]} \left\{ X^{a, b, (g^-, g^+)}(u)-g(u) \right\} = 0 \right) = 0, 
\label{eq_curve_min_t1_bdry}\\
&P\left( \min_{u\in [t,1]}\left\{ X^{a, b, (g^-, g^+)}(u)-g(u) \right\} \geq 0, X^{a, b, (g^-, g^+)}(t) \leq z \right)
\nonumber \\
&\quad =\int_{g(t)}^z(C_{g^-,g^+})^{-1}
p^{(g^-,g^+), (\uparrow)}_{[0,t]}(y)p^{(g,g^+), (\downarrow)}_{[t,1]}(y)dy.
\label{eq_curve_min_t1_jointdist}
\end{align}
\end{cor}

\begin{rem}
Let $t\in (0,1)$. Applying 
Corollary~\ref{Cor_30_curve_max_0t} 
(resp., Corollary~\ref{Cor_30_curve_min_t1}) 
for $g=g^+$ (resp., $g=g^-$), we have
\begin{align*}
&P\left( \min_{u\in [0,t]} \left\{ g^+(u)-X^{a, b, (g^-, g^+)}(u)\right\} = 0 \right) = 0, 
\quad P\left( \min_{u\in [t,1]} \left\{ X^{a, b, (g^-, g^+)}(u)-g^-(u)\right\} = 0 \right) = 0, \\
&P\left( \min_{u\in [0,t]} \left\{ g^+(u)-X^{a, b, (g^-, g^+)}(u)\right\} \geq 0 \right) \\
&\quad = P\left( \min_{u\in [0,t]} \left\{ g^+(u)-X^{a, b, (g^-, g^+)}(u)\right\} \geq 0, X^{a, b, (g^-, g^+)}(t) \leq g^+(t) \right) 
= \int_{g^-(t)}^{g^+(t)} h(t,y) dy = 1, \\
&P\left( \min_{u\in [t,1]} \left\{ X^{a, b, (g^-, g^+)}(u)-g^-(u)\right\} \geq 0 \right) \\
&\quad = P\left( \min_{u\in [t,1]} \left\{ X^{a, b, (g^-, g^+)}(u)-g^-(u)\right\} \geq 0, X^{a, b, (g^-, g^+)}(t) \leq g^+(t) \right) 
= \int_{g^-(t)}^{g^+(t)} h(t,y) dy = 1, \\
&P\left( \min_{u\in [0,t]} \left\{ g^+(u)-X^{a, b, (g^-, g^+)}(u)\right\} > 0 \right) \\
&\quad = P\left( \min_{u\in [0,t]} \left\{ g^+(u)-X^{a, b, (g^-, g^+)}(u)\right\} \geq 0 \right) 
- P\left( \min_{u\in [0,t]} \left\{ g^+(u)-X^{a, b, (g^-, g^+)}(u)\right\} = 0 \right) =1, \\
&P\left( \min_{u\in [t,1]} \left\{ X^{a, b, (g^-, g^+)}(u)-g^-(u)\right\} > 0 \right)  \\
&\quad = P\left( \min_{u\in [t,1]} \left\{ X^{a, b, (g^-, g^+)}(u)-g^-(u)\right\} \geq 0 \right) 
- P\left( \min_{u\in [t,1]} \left\{ X^{a, b, (g^-, g^+)}(u)-g^-(u)\right\} = 0 \right) =1.
\end{align*}
Therefore, diffusion house-moving $X^{a, b, (g^-, g^+)}$ satisfies
\begin{align*}
P\left( \bigcap_{n\geq 2} \left\{ \min_{0\leq u \leq 1-1/n}(g^+(u)-X^{a, b, (g^-, g^+)}(u)) > 0, 
\ \min_{1/n\leq u\leq 1}(X^{a, b, (g^-, g^+)}(u)-g^-(u)) > 0\right\} \right)=1.
\end{align*}
\end{rem}

%%%%%%%%%%%%%%%%%%%%%%%%%%%%%%%%%%%%%%%%%%%%%%%%%%
\subsection{Sample path properties of diffusion meander between two curves}
\label{subsection_Meander_between_2curves}
%%%%%%%%%%%%%%%%%%%%%%%%%%%%%%%%%%%%%%%%%%%%%%%%%%

In this subsection, we define $a:=g^-(0)$, and assume that
\begin{align*}
a = g^-(0) > g^-_{\varepsilon}(0)
\qquad (\varepsilon > 0).
\end{align*} 

For $0<t<1$, $0<t_1 < t_2<1$ and 
$y \in (g^-(t), g^+(t))$, $y_i \in (g^-(t_i), g^+(t_i))$ $(i=1,2)$, we define
\begin{align*}
&k(t,y)=(\widetilde{C}_{g^-,g^+})^{-1}p^{(g^-,g^+), (\uparrow)}_{[0,t]}(y)p^{(g^-,g^+)}_{[t,1]}(y), \\
&k(t_1, y_1, t_2, y_2) 
=\dfrac{p^{(g^-,g^+)}_{[t_1,t_2]}(y_1, y_2) p^{(g^-,g^+)}_{[t_2,1]}(y_2)}{p^{(g^-,g^+)}_{[t_1,1]}(y_1) } ,
\end{align*}
where
\begin{align*}
\widetilde{C}_{g^-,g^+}:=
\lim_{\varepsilon \downarrow 0} 
\frac{P\Big(X^a_{[0, 1]} \in K_{[0, 1]}(g^-_{\varepsilon}, g^+_{\varepsilon})\Big)}{ g^-(0) - g^-_{\varepsilon}(0) }.
\end{align*}

$X^{+, (g^-,g^+)}$ denotes $X_{[0,1]}^{a,(g^-, g^+)}$, 
which is the weak limit of 
$X_{[0,1]}^a |_{K_{[0, 1]}(g^-_{\varepsilon}, g^+_{\varepsilon})}$ 
as $\varepsilon \downarrow 0$. 
In the case $\nu(x)=0$, $\sigma(x)=1$ $(x\in \mathbb{R})$, 
the well-known Brownian meander $W^+$ can be formally denoted by $X^{+, (0,\infty)}$ in our notations. 
Therefore, in this study, we call $X^{+, (g^-,g^+)}$ ``diffusion meander between two curves.'' 
In the case $\nu(x)=0$, $\sigma(x)=1$ $(x\in \mathbb{R})$, $X^{+, (g^-,g^+)}$ played an important role in \cite{bib_Ishitani}. 
Our aim in this subsection is to prove that 
$X^{+, (g^-,g^+)}$ is an $\mathbb{R}$-valued continuous Markov process on $[0,1]$.

\begin{thm}\label{Thm_Def_and_Decomp_curved_Meander_two_curve}
There exists an $\mathbb{R}$-valued continuous Markov process 
$X^{+, (g^-,g^+)}=\{X^{+, (g^-,g^+)}(t)\}_{t\in [0,1]}$ that satisfies
\begin{align}
&E\left[ F(X^{+, (g^-,g^+)}) \right]
=\lim_{\varepsilon \downarrow 0} 
E[F(X^a_{[0,1]} |_{K_{[0, 1]}(g^-_{\varepsilon}, g^+_{\varepsilon})})]
\label{curvedMeander_twocurve_def_eq1}\\
&\quad = \int_{g^-(t)}^{g^+(t)} 
E\left[ F(X_{[0,t]}^{a,y,(g^-, g^+)} \oplus_{t} X_{[t,1]}^{y,(g^-, g^+)}) \right] k(t,y) dy 
\label{twocurve_meander_decomp_1} \\
&\quad =\int_{g^-(t_1)}^{g^+(t_1)} \int_{g^-(t_2)}^{g^+(t_2)} 
E[F(X_{[0,t_1]}^{a,y_1,(g^-, g^+)} \oplus_{t_1} 
X_{[t_1,t_2]}^{y_1, y_2,(g^-, g^+)}\oplus_{t_2} 
X_{[t_2, 1]}^{y_2,(g^-, g^+)})] 
\label{twocurve_meander_decomp_2} \\
&\quad \qquad \qquad \qquad \times
k(t_1, y_1)k(t_1, y_1, t_2, y_2) dy_1 dy_2
\nonumber
\end{align}
for every $\mathbb{R}$-valued  bounded continuous function $F$ on 
$C([0, 1], \mathbb{R})$, $0<t<1$ and $0<t_1<t_2<1$,  
where the respective processes 
that appear in \eqref{twocurve_meander_decomp_1} and \eqref{twocurve_meander_decomp_2} are independent of each other. 
Moreover, for $0 < t < 1$, $0<t_1<t_2<1$ and $y\in (g^-(t), g^+(t))$, $y_i \in (g^-(t_i), g^+(t_i))$ $(i=1,2)$, 
the transition densities for $X^{+, (g^-,g^+)}$ are given by
\begin{align*}
&P(X^{+, (g^-,g^+)}(t) \in dy) = k(t, y)dy, \quad 
P(X^{+, (g^-,g^+)}(t_2) \in dy_2 \ \vert \ X^{+, (g^-,g^+)}(t_1)=y_1) =k(t_1, y_1, t_2, y_2) dy_2.
\end{align*}
\end{thm}

\begin{rem}\label{Rem_F_canbetakento_Indicator_Meander}
Let $B\in \mathcal{B}(C([0,1], \mathbb{R}))$ be a measurable subset of $C([0,1], \mathbb{R})$. 
Then, it follows from the same argument in Remark~\ref{Rem_F_canbetakento_Indicator_MeanderBESbridge} 
that Theorem~\ref{Thm_Def_and_Decomp_curved_Meander_two_curve} holds true for $F=1_B$.
\end{rem}

Let $b:=g^+(1)$ and $t\in (0,1)$. 
Applying Theorem~\ref{Thm_Def_and_Decomp_curved_Moving} \eqref{curvedMoving_def_eq1} 
and a change of measure formula between diffusion meander and BES$(3)$-process (\cite{bib_Imhof}), 
we obtain the Radon-Nikodym derivative of $\pi_{[0, t]}\circ X^{a, b, (g^-, g^+)}$ with respect to $\phi^{-1}(R_{[0,t]} + \phi(g^-))$.

\begin{thm}\label{Thm_abs_conti} 
Let $b:=g^+(1)$ and $t\in (0,1)$.  
Let $R_{[0,t]} = \{ R_{[0,t]}(u)\}_{u\in [0,t]}$ 
be the BES($3$)-process starting at $0$ on $[0,t]$. 
Then, it holds that
\begin{align*}
&\frac{d\left( P\circ (\pi_{[0, t]}\circ X^{a,b,(g^-, g^+)})^{-1} \right)}{d\left( P\circ ( \phi^{-1}(R_{[0,t]} + \phi(g^-)) )^{-1}\right)}(w) \\
&\quad = \frac{ e^{-\Psi(\Phi(a))} }{ C_{g^-,g^+} \sigma(a) }\cdot 
\dfrac{ e^{\Psi(\Phi(w(t)))} e^{-\frac{1}{2}N_{[0,t]}(\phi(w))} p^{(g^-,g^+), (\downarrow)}_{[t,1]}\left(w(t)\right) }
{(\Phi(w(t)) - \Phi(g^-(t)))\cdot Z_{[0,t]}^{\phi(g^-)}(\phi(w))}
\cdot 1_{K_{[0,t]}^-(g^+)}(w), 
\quad w\in C([0, t],\mathbb{R}) .
\end{align*}
\end{thm}

\begin{rem}\label{Remark_Approx_ExpectedVal_HhouseMoving}
Combining Theorem~\ref{Thm_abs_conti} 
and a Monte Carlo sampling technique for the BES$(3)$-process, 
we can approximate the expected values of the functional of diffusion house-moving $X^{a,b,(g^-,g^+)}$.
\end{rem}

%%%%%%%%%%%%%%%%%%%%%%%%%%%%%%%%%%%%%%%%%%%%%%%%%%
\subsection{Sample path properties of conditioned diffusion-bridge between two curves}
\label{subsection_diffbridge_between_2curves}
%%%%%%%%%%%%%%%%%%%%%%%%%%%%%%%%%%%%%%%%%%%%%%%%%%

In this subsection, 
we define $a:=g^-(0)$ and assume that
\begin{align*}
g^-(1)<c<g^+(1), \qquad
a = g^-(0) > g^-_{\varepsilon}(0)
\quad (\varepsilon > 0). 
\end{align*}

For $0<t<1$, $0<t_1 < t_2<1$ and 
$y \in (g^-(t), g^+(t))$, $y_i \in (g^-(t_i), g^+(t_i))$ $(i=1,2)$, we define
\begin{align*}
& l(t,y)=(\widehat{C}_{g^-,g^+})^{-1}p^{(g^-,g^+), (\uparrow)}_{[0,t]}(y)p^{(g^-,g^+)}_{[t,1]}(y,c), \\
& l(t_1, y_1, t_2, y_2) 
=\dfrac{p^{(g^-,g^+)}_{[t_1,t_2]}(y_1, y_2) p^{(g^-,g^+)}_{[t_2,1]}(y_2, c)}{p^{(g^-,g^+)}_{[t_1,1]}(y_1, c) } ,
\end{align*}
where
\begin{align*}
\widehat{C}_{g^-,g^+}:=
\lim_{\varepsilon \downarrow 0} 
\frac{P(X_{[0, 1]}^{a\to c} \in K_{[0, 1]}(g^-_{\varepsilon}, g^+_{\varepsilon}))}
{ g^-(0) - g^-_{\varepsilon}(0) }
\cdot \frac{P(X^a_{[0,1]}(1) \in d\widetilde{c})}{d\widetilde{c}}\bigg\vert_{\widetilde{c}=c} .
\end{align*}

$r_X^{a \to c, (g^-,g^+)}$ denotes $X_{[0,1]}^{a,c,(g^-, g^+)}$, 
which is the weak limit of 
$X_{[0,1]}^{a\to c} |_{K_{[0, 1]}(g^-_{\varepsilon}, g^+_{\varepsilon})}$ 
as $\varepsilon \downarrow 0$. 
In the case $\nu(x)=0$, $\sigma(x)=1$ $(x\in \mathbb{R})$, 
the well-known BES($3$)-bridge $r^{0\to d}$ ($d>0$) can be formally denoted by $r_X^{0 \to d, (0,\infty)}$ in our notations. 
Therefore, in this study, we call $r_X^{a \to c, (g^-,g^+)}$ ``conditioned diffusion-bridge between two curves.'' 
In the case $\nu(x)=0$, $\sigma(x)=1$ $(x\in \mathbb{R})$, 
$r_X^{a \to c, (g^-,g^+)}$ played an important role in \cite{bib_Ishitani}. 
Our aim in this subsection is to prove that 
$r_X^{a \to c, (g^-,g^+)}$ is an $\mathbb{R}$-valued continuous Markov process on $[0,1]$.

\begin{thm}\label{Thm_Def_and_Decomp_curved_BESbridge_two_curve}
There exists an $\mathbb{R}$-valued continuous Markov process 
$r_X^{a \to c, (g^-,g^+)}=\{r_X^{a \to c, (g^-,g^+)}(t)\}_{t\in [0,1]}$ that satisfies
\begin{align}
&E\left[ F(r_X^{a \to c, (g^-,g^+)}) \right]
=\lim_{\varepsilon \downarrow 0} 
E[F(X_{[0,1]}^{a\to c} |_{K_{[0, 1]}(g^-_{\varepsilon}, g^+_{\varepsilon})})]
\label{curvedBESbridge_twocurve_def_eq1}\\
&\quad = \int_{g^-(t)}^{g^+(t)} 
E\left[ F(X_{[0,t]}^{a,y,(g^-, g^+)} \oplus_{t} X_{[t,1]}^{y,c,(g^-, g^+)}) \right] l(t,y) dy 
\label{twocurve_BESbridge_decomp_1} \\
&\quad =\int_{g^-(t_1)}^{g^+(t_1)} \int_{g^-(t_2)}^{g^+(t_2)} 
E[F(X_{[0,t_1]}^{a,y_1,(g^-, g^+)} \oplus_{t_1} 
X_{[t_1,t_2]}^{y_1, y_2,(g^-, g^+)}\oplus_{t_2} 
X_{[t_2, 1]}^{y_2,c,(g^-, g^+)})] 
\label{twocurve_BESbridge_decomp_2} \\
&\quad \qquad \qquad \qquad \times 
l(t_1, y_1)l(t_1, y_1, t_2, y_2) dy_1 dy_2
\nonumber
\end{align}
for every $\mathbb{R}$-valued bounded continuous function $F$ on 
$C([0, 1], \mathbb{R})$, $0<t<1$ and $0<t_1<t_2<1$,  
where the respective processes 
that appear in \eqref{twocurve_BESbridge_decomp_1} and \eqref{twocurve_BESbridge_decomp_2} are independent of each other. 
Moreover, for $0 < t < 1$, $0<t_1<t_2<1$ and $y\in (g^-(t), g^+(t))$, $y_i \in (g^-(t_i), g^+(t_i))$ $(i=1,2)$, 
the transition densities for $r_X^{a \to c, (g^-,g^+)}$ are given by
\begin{align*}
&P(r_X^{a \to c, (g^-,g^+)}(t) \in dy) = l(t, y)dy, \\
&P(r_X^{a \to c, (g^-,g^+)}(t_2) \in dy_2 \ \vert \ r_X^{a \to c, (g^-,g^+)}(t_1)=y_1) =l(t_1, y_1, t_2, y_2) dy_2.
\end{align*}
\end{thm}

\begin{rem}\label{Rem_F_canbetakento_Indicator_BESbridge}
Let $B\in \mathcal{B}(C([0,1], \mathbb{R}))$ be a measurable subset of $C([0,1], \mathbb{R})$. 
Then, it follows 
from the same argument in Remark~\ref{Rem_F_canbetakento_Indicator_MeanderBESbridge} 
that Theorem~\ref{Thm_Def_and_Decomp_curved_BESbridge_two_curve} holds true for $F=1_B$.
\end{rem}

Let $b:=g^+(1)$. 
We also prove that $r_X^{a\to b, (g^-,g^++\eta)}$ 
converges weakly to 
diffusion house-moving 
$X^{a, b, (g^-, g^+)}$ as $\eta \downarrow 0$.

\begin{thm}\label{Thm_WeakConv_BESbridge_to_Moving}
Let $b:=g^+(1)$. 
Assume that $g^+(1) < g^+_{\eta}(1)$ holds for $\eta>0$. 
For every $\mathbb{R}$-valued bounded continuous function $F$ on $C([0, 1], \mathbb{R})$, we have
\begin{align*}
&E\left[ F\big(X^{a, b, (g^-, g^+)}\big) \right]
=\lim_{\eta \downarrow 0} 
E[F(r_X^{a\to b, (g^-,g^+_{\eta})})].
\end{align*}
\end{thm}

\begin{rem}
Let $R=\{R(t)\}_{t\geq 0}$ be $3$-dimensional Bessel process (BES($3$) process for short) starting from $0$, 
and let $\tau_b$ $(b>0)$ denotes the first hitting time of the point $b$ by $R$:
\[
\tau_b:=\inf\{r\geq 0~|~ R(r)=b\}.
\]
In the case $\nu(x)=0$, $\sigma(x)=1$ $(x\in \mathbb{R})$, 
it has been shown in \cite{bib_IshitaniRinYanashima} that
diffusion house-moving $X^{0,b,(0,b)}=\{X^{0,b,(0,b)}(t) \}_{t \in [0,1]}$ satisfies
\begin{align*}
P\left(X^{0,b,(0,b)}(t)\in dy\right)
&=P\left(R(t)\in dy~|~\tau_b=1\right),\\
P\left(X^{0,b,(0,b)}(t)\in dy~|~X^{0,b,(0,b)}(s)=x\right)
&=P\left(R(t)\in dy~|~R(s)=x, \tau_b=1\right)\notag
\end{align*}
for $0<s<t<1$ and $x, y\in (0, b)$.
\end{rem}

\section{Preliminaries}

This section states and proves lemmas relating probabilities for the diffusion process/bridge to expectations under the law of the Brownian motion/bridge.

\begin{lem}[cf. Lemma 3.3 in~\cite{downes_borovkov}]\label{girsanov}
Let $0\leq s<t$. For $a, b\in \Phi(\mathbb{R})$ and $A \in \mathcal{B}( C([s,t], \mathbb{R}) )$, we have
\begin{align*}
&P(Y^{a \to b}_{[s,t]} \in A) 
= \frac{q_W(t-s,a,b)}{q_Y(t-s,a,b)} e^{\Psi(b) - \Psi(a)} 
E \left[ e^{-\frac{1}{2}N_{[s,t]}(B^{a \to b}_{[s,t]})} 1_A(B^{a \to b}_{[s,t]}) \right] .
\end{align*}
\end{lem}

\begin{lem}\label{girsanov_func}
Let $0 \leq t_1 < t_2$, $A \in \mathcal{B}( C([t_1,t_2] , \mathbb{R}) )$ and  $a, b\in \Phi(\mathbb{R})$.
For every $\mathbb{R}$-valued bounded continuous function $G$ on $C([t_1,t_2] , \mathbb{R})$,
we define
\begin{align*}
&I_B^{a \to b}(G ; A) :=  E[ G(B_{[t_1,t_2]}^{a \to b}) ; B_{[t_1,t_2]}^{a \to b} \in A ] , \\
&I_Y^{a \to b}(G ; A) :=  E[ G(Y_{[t_1,t_2]}^{a \to b}) ; Y_{[t_1,t_2]}^{a \to b} \in A ] ,\\
&\widehat{G}(w) := e^{-\frac{1}{2}N_{[t_1,t_2]}(w)} G(w) .
\end{align*}
Then, for every $\mathbb{R}$-valued bounded continuous function $F$ on $C([t_1,t_2] , \mathbb{R})$, we have
\begin{align}\label{girsanov_func_eq0}
I_Y^{a \to b}(F ; A) = \frac{ q_W(t_2 - t_1 , a, b) }{ q_Y(t_2 - t_1 , a, b) } e^{\Psi(b) - \Psi(a)} I_B^{a \to b}( \widehat{F} ; A ) . 
\end{align}
\end{lem}
\begin{proof}
By Lemma~\ref{girsanov}, we have
\begin{align*}
I_Y^{a \to b}( 1_C ; A ) 
= \frac{ q_W(t_2 - t_1 , a, b) }{ q_Y(t_2 - t_1 , a, b) } e^{\Psi(b) - \Psi(a)} 
E\left[ e^{-\frac{1}{2}N_{[t_1,t_2]}(B_{[t_1,t_2]}^{a \to b})} 1_C(B_{[t_1,t_2]}^{a \to b}) ; B_{[t_1,t_2]}^{a \to b} \in A \right]
\end{align*}
for every $C \in \mathcal{B}( C([t_1,t_2] , \mathbb{R}) )$.
Hence, by using the simple approximation theorem and the bounded convergence theorem, we obtain
\begin{align*}
I_Y^{a \to b}( F ; A ) 
&= \frac{ q_W(t_2 - t_1 , a, b) }{ q_Y(t_2 - t_1 , a, b) } e^{\Psi(b) - \Psi(a)} 
E\left[ e^{-\frac{1}{2}N_{[t_1,t_2]}(B_{[t_1,t_2]}^{a \to b})} F(B_{[t_1,t_2]}^{a \to b}) ; B_{[t_1,t_2]}^{a \to b} \in A \right] \\
&= \frac{ q_W(t_2 - t_1 , a, b) }{ q_Y(t_2 - t_1 , a, b) } e^{\Psi(b) - \Psi(a)} I_B^{a \to b}( \widehat{F} ; A )
\end{align*}
for every bounded continuous function $F$ on $C([t_1,t_2] , \mathbb{R})$.
Thus, the equation \eqref{girsanov_func_eq0} holds.
\end{proof}

\begin{lem}[cf. proof of Lemma 3.3 in~\cite{downes_borovkov}]\label{girsanov2}
Let $0\leq s<t$. For $a\in \Phi(\mathbb{R})$ and $A \in \mathcal{B}( C([s,t], \mathbb{R}) )$, we have
\begin{align*}
&P(Y^a_{[s,t]} \in A) = e^{- \Psi(a)} E \left[ e^{\Psi(W^a_{[s,t]}(t))} e^{-\frac{1}{2}N_{[s,t]}(W^a_{[s,t]})} 1_A(W^a_{[s,t]}) \right] .
\end{align*}
\end{lem}

\begin{lem}\label{girsanov_func2}
Let $0 \leq t_1 < t_2$, $A \in \mathcal{B}( C([t_1,t_2] , \mathbb{R}) )$ and  $a\in \Phi(\mathbb{R})$.
For every $\mathbb{R}$-valued bounded continuous function $G$ on $C([t_1,t_2] , \mathbb{R})$, we define
\begin{align*}
&I_W^a(G ; A) :=  E[ G(W_{[t_1,t_2]}^a) ; W_{[t_1,t_2]}^a \in A ] , \\
&I_Y^a(G ; A) :=  E[ G(Y_{[t_1,t_2]}^a) ; Y_{[t_1,t_2]}^a \in A ] , \\
&\acute{G}(w) := e^{\Psi(w(t_2))} e^{-\frac{1}{2}N_{[t_1,t_2]}(w)} G(w) .
\end{align*}
Then, for every $\mathbb{R}$-valued bounded continuous function $F$ on $C([t_1,t_2] , \mathbb{R})$, we have
\begin{align*}
I_Y^a(F ; A) &= e^{-\Psi(a)} I_W^a( \acute{F} ; A ). 
\end{align*}
\end{lem}

The proof of Lemma~\ref{girsanov_func2} is similar to that of Lemma~\ref{girsanov_func} and it can be deduced from Lemma~\ref{girsanov2}.

%%%%%%%%%%%%%%%%%%%%%%%%%%%%%%%%%%%%%%%%%%%%%%%%%%%%%%%%%%%%%%%%%
\section{Proofs of Lemma~\ref{Lem_for_H_gm_to_gp} and Lemma~\ref{Lem_WeakConv_Meander_btwn_2crvs}}
\label{section_proof_main_Lemma}
%%%%%%%%%%%%%%%%%%%%%%%%%%%%%%%%%%%%%%%%%%%%%%%%%%%%%%%%%%%%%%%%%

In this section, we prove Lemma~\ref{Lem_for_H_gm_to_gp} and Lemma~\ref{Lem_WeakConv_Meander_btwn_2crvs}.

%%%%%%%%%%%%%%%%%%%%%%%%%%%%%%%%%%%
\subsection{Proof of Lemma~\ref{Lem_for_H_gm_to_gp}}
%%%%%%%%%%%%%%%%%%%%%%%%%%%%%%%%%%%

Assume that $h^-$ and $h^+$ are $\mathbb{R}$-valued continuous functions defined on $[t_1, t_2]$ 
satisfying 
\begin{align*}
h^-(t_1)<a<h^+(t_1),\quad h^-(t_2)<b<h^+(t_2)\quad \mbox{and}  \quad
\min_{t_1 \leq t \leq t_2}(h^+(t) - h^-(t)) > 0.
\end{align*} 
Then, it follows from Lemma~\ref{girsanov} and Lemma~\ref{girsanov_func} that
\begin{align}
&E[ F( X^{a \to b}_{[t_1, t_2]}\mid_{K_{[t_1,t_2]}(h^-, h^+)} ) ] 
\nonumber \\
&= \frac{E[ F( X^{a \to b}_{[t_1, t_2]} ) ; X^{a \to b}_{[t_1, t_2]} \in K_{[t_1,t_2]}(h^-, h^+) ]}
{P( X^{a \to b}_{[t_1, t_2]} \in K_{[t_1,t_2]}(h^-, h^+) )} 
\nonumber \\
&= \frac{E[ (F \circ \phi^{-1})( Y^{\Phi(a) \to \Phi(b)}_{[t_1, t_2]} ) ; Y^{\Phi(a) \to \Phi(b)}_{[t_1, t_2]} \in K_{[t_1,t_2]}( \phi(h^-), \phi(h^+)) ]}
{P( Y^{\Phi(a) \to \Phi(b)}_{[t_1, t_2]} \in K_{[t_1,t_2]}( \phi(h^-), \phi(h^+)) )} 
\nonumber \\
&= \frac{E[ e^{-\frac{1}{2}N_{[t_1,t_2]}( B^{\Phi(a) \to \Phi(b)}_{[t_1, t_2]} )} (F \circ \phi^{-1})( B^{\Phi(a) \to \Phi(b)}_{[t_1, t_2]} ) ; 
B^{\Phi(a) \to \Phi(b)}_{[t_1, t_2]} \in K_{[t_1,t_2]}( \phi(h^-), \phi(h^+)) ]}
{E[ e^{-\frac{1}{2}N_{[t_1,t_2]}(B_{[t_1, t_2]}^{\Phi(a) \to \Phi(b)})} ; B^{\Phi(a) \to \Phi(b)}_{[t_1, t_2]} \in K_{[t_1,t_2]}( \phi(h^-), \phi(h^+)) ]} 
\nonumber \\
&= \frac{E[ e^{-\frac{1}{2}N_{[t_1,t_2]}( B^{\Phi(a) , \Phi(b) ,  ( \phi(h^-), \phi(h^+) ) }_{[t_1, t_2]} )} (F \circ \phi^{-1})( B^{\Phi(a) , \Phi(b) , ( \phi(h^-) , \phi(h^+) )}_{[t_1, t_2]} ) ]}{E[ e^{-\frac{1}{2}N_{[t_1,t_2]}(B_{[t_1, t_2]}^{\Phi(a) , \Phi(b)  , ( \phi(h^-), \phi(h^+) ) })} ]}.
\label{compound_girsanov_bridge}
\end{align}
Therefore, using \eqref{compound_girsanov_bridge} and Lemma~\ref{Lem_conpoundweakconv_2crvs_BESbridge}, 
we obtain Lemma~\ref{Lem_for_H_gm_to_gp}. 

%%%%%%%%%%%%%%%%%%%%%%%%%%%%%%%%%%%%%%%%%%%%%%
\subsection{Proof of Lemma~\ref{Lem_WeakConv_Meander_btwn_2crvs}}
%%%%%%%%%%%%%%%%%%%%%%%%%%%%%%%%%%%%%%%%%%%%%%

Assume that $h^-$ and $h^+$ are $\mathbb{R}$-valued continuous functions defined on $[t_1, t_2]$ 
satisfying 
\begin{align*}
h^-(t_1)<a<h^+(t_1)\quad \mbox{and}  \quad
\min_{t_1 \leq t \leq t_2}(h^+(t) - h^-(t)) > 0.
\end{align*} 
Then, it follows from Lemma~\ref{girsanov2} and Lemma~\ref{girsanov_func2} that
\begin{align}
&E[ F( X^{a}_{[t_1, t_2]}\mid_{K_{[t_1,t_2]}(h^-, h^+)} ) ] \nonumber \\
&= \frac{E[ F( X^{a}_{[t_1, t_2]} ) ; X^{a}_{[t_1, t_2]} \in K_{[t_1,t_2]}(h^-, h^+) ]}
{P( X^{a}_{[t_1, t_2]} \in K_{[t_1,t_2]}(h^-, h^+) )} 
\nonumber \\
&= \frac{E[ (F \circ \phi^{-1})( Y^{\Phi(a)}_{[t_1, t_2]} ) ; Y^{\Phi(a)}_{[t_1, t_2]} \in K_{[t_1,t_2]}( \phi(h^-), \phi(h^+)) ]}
{P( Y^{\Phi(a)}_{[t_1, t_2]} \in K_{[t_1,t_2]}( \phi(h^-), \phi(h^+)) )} 
\nonumber \\
&= \frac{E[ e^{\Psi(W^{\Phi(a)}_{[t_1,t_2]}(t_2))} e^{-\frac{1}{2}N_{[t_1,t_2]}( W^{\Phi(a)}_{[t_1, t_2]} )} (F \circ \phi^{-1})( W^{\Phi(a)}_{[t_1, t_2]} ) ; W^{\Phi(a)}_{[t_1, t_2]} \in K_{[t_1,t_2]}( \phi(h^-), \phi(h^+)) ]}
{E[ e^{\Psi(W^{\Phi(a)}_{[t_1,t_2]}(t_2))} e^{-\frac{1}{2}N_{[t_1,t_2]}(W_{[t_1, t_2]}^{\Phi(a)})} ; W^{\Phi(a)}_{[t_1, t_2]} \in K_{[t_1,t_2]}( \phi(h^-), \phi(h^+)) ]} 
\nonumber \\
&= \frac{E[ e^{\Psi(W^{\Phi(a),(\phi(h^-),\phi(h^+))}_{[t_1,t_2]}(t_2))} e^{-\frac{1}{2}N_{[t_1,t_2]}( W^{\Phi(a), (\phi(h^-), \phi(h^+)) }_{[t_1, t_2]} )} 
(F \circ \phi^{-1})(W^{\Phi(a), (\phi(h^-), \phi(h^+) )}_{[t_1, t_2]} ) ]}
{E[ e^{\Psi(W^{\Phi(a),(\phi(h^-),\phi(h^+))}_{[t_1,t_2]}(t_2))} e^{-\frac{1}{2}N_{[t_1,t_2]}(W_{[t_1, t_2]}^{\Phi(a), (\phi(h^-), \phi(h^+)) })} ]}.
\label{compound_girsanov_motion}
\end{align}
Therefore, using \eqref{compound_girsanov_motion} and Lemma~\ref{Lem_conpoundweakconv_2crvs_Meander}, 
we obtain Lemma~\ref{Lem_WeakConv_Meander_btwn_2crvs}.

%%%%%%%%%%%%%%%%%%%%%%%%%%%%%%%%%%
\section{Preparation for proofs of the main results}
\label{section_proof_Preparation_for_Proofs}
%%%%%%%%%%%%%%%%%%%%%%%%%%%%%%%%%%

In this section, we prove some technical lemmas 
in preparation for proofs of the main results. 

\begin{lem}\label{Lem_MarkovDecomp_B_Bridge}
Let $a, b\in \mathbb{R}$. 
Assume that $h^-$ and $h^+$ are $\mathbb{R}$-valued $C^2$-functions defined on $[0, 1]$ 
satisfying 
\begin{align*}
h^-(0)<a<h^+(0),\quad h^-(1)<b<h^+(1)\quad \mbox{and}  \quad
\min_{0 \leq t \leq 1}(h^+(t) - h^-(t)) > 0.
\end{align*} 
Then, for every $\mathbb{R}$-valued bounded continuous function $F$ on $C([0,1], \mathbb{R})$ 
and $0<s<t<1$, we have 
\begin{align}
& E\left[F(X^{a\to b}_{[0,1]})\ ;\ X^{a\to b}_{[0,1]} \in K_{[0, 1]}(h^-, h^+)\right] 
\cdot \frac{P(X^a_{[0,1]}(1) \in d\widetilde{b})}{d\widetilde{b}}\bigg\vert_{\widetilde{b}=b}
\nonumber \\
&\quad = 
\int_{h^-(t)}^{h^+(t)}  dy
\ E[F(X_{[0, t]}^{a,y,(h^-, h^+)} \oplus_{t} X_{[t, 1]}^{y, b,(h^-, h^+)})] 
\label{MarkovDecomp_at_t_Bridge} \\
&\qquad \qquad \quad \times 
P(X^{a\to y}_{[0, t]} \in K_{[0, t]}(h^-, h^+))
\cdot \frac{P(X^a_{[0, t]}(t) \in dy)}{dy}
\nonumber \\
&\qquad \qquad \quad \times 
P(X^{y\to b}_{[t,1]} \in K_{[t, 1]}(h^-, h^+))
\cdot \frac{P(X^y_{[t, 1]}(1) \in d\widetilde{b})}{d\widetilde{b}}\bigg\vert_{\widetilde{b}=b}\nonumber \\
&\quad = 
\int_{h^-(s)}^{h^+(s)} dx \int_{h^-(t)}^{h^+(t)} dy
\ E[F(X_{[0, s]}^{a,x,(h^-, h^+)} \oplus_{s} 
X_{[s,t]}^{x, y,(h^-, h^+)} \oplus_{t} 
X_{[t,1]}^{y, b,(h^-, h^+)})] \label{MarkovDecomp_at_st_Bridge} \\
&\qquad \qquad \qquad \times 
P(X^{a\to x}_{[0, s]} \in K_{[0, s]}(h^-, h^+))
\cdot \frac{P(X^a_{[0, s]}(s) \in dx)}{dx} \nonumber \\
&\qquad \qquad \qquad \times 
P(X^{x\to y}_{[s,t]} \in K_{[s,t]}(h^-, h^+))
\cdot \frac{P(X^x_{[s,t]}(t) \in dy)}{dy} \nonumber \\
&\qquad \qquad \qquad \times 
P(X^{y\to b}_{[t, 1]} \in K_{[t, 1]}(h^-, h^+))
\cdot \frac{P(X^y_{[t, 1]}(1) \in d\widetilde{b})}{d\widetilde{b}}\bigg\vert_{\widetilde{b}=b}, \nonumber 
\end{align}
where the respective processes 
that appear in \eqref{MarkovDecomp_at_t_Bridge} 
and \eqref{MarkovDecomp_at_st_Bridge} are independent of each other. 
\end{lem}
\begin{proof}
The Markov property of $X^a_{[0,1]}$ yields
\begin{align}
&E[F(X^a_{[0,1]})\ ;\ X^a_{[0,1]}(1) \in d\widetilde{b}, X^a_{[0,1]} \in K_{[0, 1]}(h^-, h^+)]
\nonumber \\
&\quad= \int_{h^-(t)}^{h^+(t)} 
E[F(X^a_{[0,1]})\ ;\ X^a_{[0,1]}(1) \in d\widetilde{b}, 
X^a_{[0,1]} \in K_{[0, 1]}(h^-, h^+), X^a_{[0,1]}(t) \in dy]
\nonumber \\
&\quad = \int_{h^-(t)}^{h^+(t)} 
E[F(X^a_{[0, t]} \oplus_{t} X^y_{[t, 1]})\ ; \ 
X^a_{[0, t]} \in K_{[0, t]}(h^-, h^+), X^a_{[0, t]}(t) \in dy, 
\label{MarkovDecomp_at_t_Proof_Bridge} \\
&\qquad \qquad \qquad \qquad \qquad \qquad \quad
X^y_{[t, 1]} \in K_{[t, 1]}(h^-, h^+), X^y_{[t, 1]}(1) \in d\widetilde{b}] \nonumber \\
&\quad = \int_{h^-(s)}^{h^+(s)} \int_{h^-(t)}^{h^+(t)} 
E[F(X^a_{[0, s]} \oplus_{s} X^x_{[s,t]} 
\oplus_{t} X^y_{[t, 1]})\ ; 
\label{MarkovDecomp_at_st_Proof_Bridge}\\
&\quad \qquad \qquad \qquad \qquad \qquad 
X^a_{[0, s]} \in K_{[0, s]}(h^-, h^+), 
X^a_{[0, s]}(s) \in dx, \nonumber \\
&\quad \qquad \qquad \qquad \qquad \qquad 
X^x_{[s,t]} \in K_{[s,t]}(h^-, h^+), 
X^x_{[s,t]}(t) \in dy, \nonumber \\
&\quad \qquad \qquad \qquad \qquad \qquad 
X^y_{[t, 1]} \in K_{[t, 1]}(h^-, h^+), 
X^y_{[t, 1]}(1) \in d\widetilde{b}] , \nonumber 
\end{align}
where the respective processes 
that appear in \eqref{MarkovDecomp_at_t_Proof_Bridge} 
and \eqref{MarkovDecomp_at_st_Proof_Bridge} are independent of each other. 
Using \eqref{MarkovDecomp_at_t_Proof_Bridge} and \eqref{MarkovDecomp_at_st_Proof_Bridge}, 
we obtain \eqref{MarkovDecomp_at_t_Bridge} and \eqref{MarkovDecomp_at_st_Bridge}, respectively. 
\end{proof}

In a similar manner to the above lemma, we can obtain the following.

\begin{lem}\label{Lem_MarkovDecomp_BM}
Under the same assumption as that of 
Lemma~\ref{Lem_MarkovDecomp_B_Bridge}, we have 
\begin{align}
& E[F(X^a_{[0,1]})\ ;\ X^a_{[0,1]} \in K_{[0, 1]}(h^-, h^+)] \nonumber \\
&\quad = \int_{h^-(t)}^{h^+(t)}  dy 
\ E[F(X_{[0, t]}^{a,y,(h^-, h^+)} \oplus_{t} X_{[t, 1]}^{y,(h^-, h^+)})] 
\label{MarkovDecomp_at_t_BM} \\
&\qquad \qquad \times 
P\left(X^{a\to y}_{[0, t]} \in K_{[0, t]}(h^-, h^+)\right)
\cdot \frac{P(X^a_{[0, t]}(t) \in dy)}{dy}\nonumber \\
&\qquad \qquad \times P(X^y_{[t, 1]} \in K_{[t, 1]}(h^-, h^+))\nonumber \\
&\quad = \int_{h^-(s)}^{h^+(s)} dx \int_{h^-(t)}^{h^+(t)} dy  
\ E[F(X_{[0, s]}^{a,x,(h^-, h^+)} \oplus_{s} 
X_{[s,t]}^{x, y,(h^-, h^+)} \oplus_{t} 
X_{[t,1]}^{y,(h^-, h^+)})] \label{MarkovDecomp_at_st_BM} \\
&\qquad \qquad \qquad \qquad \quad \times 
P\left(X^{a\to x}_{[0, s]} \in K_{[0, s]}(h^-, h^+)\right)
\cdot \frac{P(X^a_{[0, s]}(s) \in dx)}{dx}\nonumber \\
&\qquad \qquad \qquad \qquad \quad \times 
P\left(X^{x\to y}_{[s,t]} \in K_{[s,t]}(h^-, h^+)\right)
\cdot \frac{P(X^x_{[s,t]}(t) \in dy)}{dy} \nonumber \\
&\qquad \qquad \qquad \qquad \quad \times 
P(X^y_{[t, 1]} \in K_{[t, 1]}(h^-, h^+)), \nonumber 
\end{align}
for every $\mathbb{R}$-valued bounded continuous function $F$ on $C([0,1], \mathbb{R})$ 
and $0<s<t<1$, 
where the respective processes 
that appear in \eqref{MarkovDecomp_at_t_BM} 
and \eqref{MarkovDecomp_at_st_BM} are independent of each other. 
\end{lem}

\begin{lem}\label{Lem_density_ae_convergence}
Let $0\leq s<u<t\leq 1$ and $y>0$. For $\varepsilon>0$, let $y_{\varepsilon}$ be real numbers that satisfy 
$\lim_{\varepsilon \downarrow 0}y_{\varepsilon} = y$. 
Then we have
\begin{align}
\lim_{\varepsilon \downarrow 0}
\frac{P\left( W_{[s, t]} |_{K_{[s, t]}^+(-\varepsilon )}(u) \in dy_{\varepsilon} \right)}{dy_{\varepsilon}}
=\frac{P\left( W_{[s, t]}^+(u) \in dy \right)}{dy}.
\label{eq_Lem_density_ae_conv}
\end{align}
\end{lem}
\begin{proof}
From the scaling property we can easily show that it suffices to prove \eqref{eq_Lem_density_ae_conv} in the case when $s=0$ and $t=1$. 
Let $0<u<1$. Using Lemma~A.1 in \cite{bib_ishitani_hatakenaka_suzuki}, we have
\begin{align}
&P\left( W_{[0,1]}(u)\in dy, \ m_{[0,1]}(W_{[0,1]})\geq -\varepsilon \right)
\nonumber \\
&\quad = P(W_{[0,1]}(u)\in dy, \ m_{[0,u]}(W_{[0,1]})\geq -\varepsilon) \cdot P\left( m_{[0,1-u]}(W_{[0,1]})\geq -y-\varepsilon \right)
\nonumber \\
&\quad = 2\left( n_u(y) - n_u(y+2\varepsilon) \right) \int_0^{y+\varepsilon} n_{1-u}(x)dx dy, 
\quad (y \geq -\varepsilon).
\label{prepare_for_densityaeconv_meander}
\end{align}
Therefore, using \eqref{prepare_for_densityaeconv_meander} and L'H\^{o}pital's rule, we can deduce for $y>0$ that
\begin{align*}
&\frac{P\left( W_{[0, 1]} |_{K_{[0, 1]}^+(-\varepsilon )}(u) \in dy_{\varepsilon} \right)}{dy_{\varepsilon}} 
=\frac{n_u(y_{\varepsilon}) - n_u(y_{\varepsilon}+2\varepsilon) }{\int_0^{\varepsilon} n_1(x)dx}
\int_0^{y_{\varepsilon}+\varepsilon} n_{1-u}(x)dx \\
&\quad \to 2\sqrt{2\pi}\frac{y n_u(y)}{u} \int_0^y n_{1-u}(x)dx
=\frac{P\left( W^+(u) \in dy \right)}{dy}, \quad \varepsilon \downarrow 0.
\end{align*}
\end{proof}

%%%%%%%%%%%%%%%%%%%%%%%%%%%%%%%%%%%%%%%%%%%%%%%%%%%%
\section{Proofs of the main results in Subsection~\ref{subsection_Moving_between_2curves}}
\label{section_proof_Moving_between_2curves}
%%%%%%%%%%%%%%%%%%%%%%%%%%%%%%%%%%%%%%%%%%%%%%%%%%%%

In this section, we prove the main results in Subsection~\ref{subsection_Moving_between_2curves}.

\begin{lem}\label{Lem_pX_conv_1}
For $0 < t < 1$, we have
\begin{align}
&p^{(g^-,g^+),(\uparrow)}_{[0,t]}(y)
=\lim_{\varepsilon \downarrow 0} 
\frac{ P( X^{g^-(0) \to y}_{[0,t]} \in K_{[0,t]}(g^-_{\varepsilon}, g^+_{\varepsilon}) ) }{ g^-(0) - g^-_{\varepsilon}(0) } 
\cdot \frac{ P( X^{g^-(0)}_{[0,t]}(t) \in dy ) }{ dy } \nonumber \\
&\quad = 
e^{ \Psi(\Phi(y)) - \Psi(\Phi(g^-(0))) } 
E\left[ e^{-\frac{1}{2}N_{[0,t]}( B^{ \Phi(g^-(0)), \Phi(y), ( \phi(g^-), \phi(g^+) ) }_{[0,t]} )} \right] 
\cdot \frac{1}{\sigma(y)}
\label{p_upper_expression}\\
&\qquad \times 
E\left[ \widetilde{Z}^{\phi(g^-)}_{[0,t]}( r^{ 0 \to \Phi(y) - \Phi(g^-(t))}_{[0,t]}\mid_{K^-_{[0,t]}( \phi(g^+) - \phi(g^-) )} )^{-1} \right] 
\nonumber \\
&\qquad \times 
P\left( r^{0 \to \Phi(y) - \Phi(g^-(t))}_{[0,t]} \in  K^-_{[0,t]}(\phi(g^+) - \phi(g^-)) \right) 
\nonumber \\
&\qquad \times 
\frac{1}{ \sigma(g^-(0)) } \sqrt{\frac{2}{\pi t}}\cdot \left.\frac{P( W^+_{[0,t]}(t) \in dc  )}{ dc } \right|_{c = \Phi(y)- \Phi(g^-(t))} 
\qquad y \in (g^-(t), g^+(t)), 
\nonumber
\end{align}
and
\begin{align}
\lim_{\varepsilon \downarrow 0} \int_{\mathbb{R}} \left| 
\frac{ P( X^{g^-(0) \to y}_{[0,t]} \in K_{[0,t]}(g^-_{\varepsilon}, g^+_{\varepsilon}) ) }{ g^-(0) - g^-_{\varepsilon}(0) } 
\cdot \frac{ P( X^{g^-(0)}_{[0,t]}(t) \in dy ) }{ dy } 
- p^{(g^-,g^+),(\uparrow)}_{[0,t]}(y) 
\right| dy = 0 . \label{L1_conv_p_upper}
\end{align}
\end{lem}
\begin{proof}
Using Lemma~\ref{girsanov}, we obtain
\begin{align}
&\frac{P\left( X^{g^-(0) \to y}_{[0,t]} \in K_{[0,t]}(g^-_{\varepsilon}, g^+_{\varepsilon}) \right) }{ g^-(0) - g^-_{\varepsilon}(0) } 
\cdot \frac{ P( X^{g^-(0)}_{[0,t]}(t) \in dy ) }{ dy } 
\nonumber \\
&\quad = 
\frac{P\left( Y^{\Phi(g^-(0)) \to \Phi(y) }_{[0,t]} \in K_{[0,t]}( \phi(g^-_{\varepsilon}), \phi(g^+_{\varepsilon}) ) \right) }{ g^-(0) - g^-_{\varepsilon}(0) } 
\cdot \left.\frac{ P( Y^{\Phi(g^-(0))}_{[0,t]}(t) \in dc ) }{ dc }\right|_{c = \Phi(y)} 
\cdot \frac{1}{\sigma(y)}
\nonumber \\
&\quad =
e^{\Psi(\Phi(y)) - \Psi(\Phi(g^-(0)))} 
E\left[ e^{-\frac{1}{2}N_{[0,t]}( B_{[0,t]}^{ \Phi(g^-(0)), \Phi(y), ( \phi(g^-_{\varepsilon}), \phi(g^+_{\varepsilon}) ) } )} \right] 
\cdot \frac{1}{\sigma(y)} 
\nonumber \\
&\qquad \times 
\frac{P\left( B_{[0,t]}^{\Phi(g^-(0)) \to \Phi(y) } \in K_{[0,t]}( \phi(g^-_{\varepsilon}), \phi(g^+_{\varepsilon}) ) \right) }
{ g^-(0) - g^-_{\varepsilon}(0) } 
\cdot \left.\frac{ P( W^{\Phi(g^-(0))}_{[0,t]}(t) \in dc ) }{ dc }\right|_{c = \Phi(y)} 
\label{Eq1_Lem_pXconv1_Proof}
\end{align}
for $g^-_{\varepsilon}(t)<y<g^+_{\varepsilon}(t)$. 
It follows from Lemma~A.10 in \cite{bib_ishitani_hatakenaka_suzuki} that
\begin{align}
&P\left( B^{\Phi(g^-(0)) \to \Phi(y) }_{[0,t]} \in K_{[0,t]}( \phi(g^-_{\varepsilon}), \phi(g^+_{\varepsilon}) ) \right)  
\cdot \left.\frac{ P( W^{\Phi(g^-(0))}_{[0,t]}(t) \in dc ) }{ dc }\right|_{c = \Phi(y)} 
\nonumber \\
&\quad = 
\left. \frac{P\left( W^{\Phi(g^-(0))}_{[0,t]} \in K_{[0,t]}( \phi(g^-_{\varepsilon}), \phi(g^+_{\varepsilon}) ), 
\ W^{\Phi(g^-(0))}_{[0,t]}(t)\in dc \right) }{ dc }\right|_{c = \Phi(y)} 
\nonumber \\
&\quad = 
E\left[
\widetilde{Z}_{[0,t]}^{\phi(g^-_{\varepsilon})} 
\big(B^{0, \Phi(y)+\Phi(g^-_{\varepsilon}(0))-\Phi(g^-_{\varepsilon}(t))-\Phi(g^-(0)), 
(\Phi(g^-_{\varepsilon}(0))-\Phi(g^-(0)), \phi(g^+_{\varepsilon})-\phi(g^-_{\varepsilon}) +\Phi(g^-_{\varepsilon}(0))-\Phi(g^-(0)))}_{[0,t]}\big)^{-1} \right] 
\nonumber \\
&\qquad \times
P\Big( B^{0 \to \Phi(y)+\Phi(g^-_{\varepsilon}(0))-\Phi(g^-_{\varepsilon}(t))-\Phi(g^-(0))}_{[0,t]} \nonumber \\ 
&\quad \qquad \quad \in 
K_{[0,t]}\big(\Phi(g^-_{\varepsilon}(0))-\Phi(g^-(0)), \phi(g^+_{\varepsilon})-\phi(g^-_{\varepsilon}) +\Phi(g^-_{\varepsilon}(0))-\Phi(g^-(0))\big)\Big)
\nonumber \\
&\qquad \times 
\frac{P\left(W_{[0,t]}(t) \in dc\right)}{dc} \Bigg\vert_{ c = \Phi(y) - \Phi(g^-_{\varepsilon}(t))+\Phi(g^-_{\varepsilon}(0))-\Phi(g^-(0))} 
\qquad (g^-_{\varepsilon}(t)<y<g^+_{\varepsilon}(t)).
\label{Eq2_Lem_pXconv1_Proof} 
\end{align}
On the other hand, we have
\begin{align}
&P\Big( B^{0 \to \Phi(y)+\Phi(g^-_{\varepsilon}(0))-\Phi(g^-_{\varepsilon}(t))-\Phi(g^-(0))}_{[0,t]} 
\nonumber \\
&\quad 
\in K_{[0,t]}\big(\Phi(g^-_{\varepsilon}(0))-\Phi(g^-(0)), \phi(g^+_{\varepsilon})-\phi(g^-_{\varepsilon}) +\Phi(g^-_{\varepsilon}(0))-\Phi(g^-(0))\big)\Big)
\nonumber \\
&\times 
\frac{P\left(W_{[0,t]}(t) \in dc\right)}{dc} \Bigg\vert_{ c = \Phi(y) - \Phi(g^-_{\varepsilon}(t))+\Phi(g^-_{\varepsilon}(0))-\Phi(g^-(0))} 
\nonumber \\
%%%&= P \left( B^{0 \to \Phi(y)+\Phi(g^-_{\varepsilon}(0))-\Phi(g^-_{\varepsilon}(t))-\Phi(g^-(0))}_{[0,t]} 
%%%\big\vert_{ K^+_{[0,t]}(\Phi(g^-_{\varepsilon}(0))-\Phi(g^-(0)))} 
%%%\in K^-_{[0,t]}( \phi(g^+_{\varepsilon})-\phi(g^-_{\varepsilon}) +\Phi(g^-_{\varepsilon}(0)) -\Phi(g^-(0)))  \right) 
%%%\nonumber \\
%%%&\quad \times 
%%%P \left( B^{0 \to \Phi(y)+\Phi(g^-_{\varepsilon}(0))-\Phi(g^-_{\varepsilon}(t))-\Phi(g^-(0))}_{[0,t]} \in K^+_{[0,t]}(\Phi(g^-_{\varepsilon}(0))-\Phi(g^-(0))) \right) 
%%%\nonumber \\
%%%&\quad \times 
%%%\frac{P\left(W_{[0,t]}(t) \in dc\right)}{dc} \Bigg\vert_{ c = \Phi(y) -\Phi(g^-_{\varepsilon}(t))+\Phi(g^-_{\varepsilon}(0))-\Phi(g^-(0))} 
%%%\nonumber \\
%%%&= P \left( B^{0 \to \Phi(y)+\Phi(g^-_{\varepsilon}(0))-\Phi(g^-_{\varepsilon}(t))-\Phi(g^-(0))}_{[0,t]} \big\vert_{ K^+_{[0,t]}(\Phi(g^-_{\varepsilon}(0))-\Phi(g^-(0)))} 
%%%\in K^-_{[0,t]}( \phi(g^+_{\varepsilon})-\phi(g^-_{\varepsilon}) +\Phi(g^-_{\varepsilon}(0)) -\Phi(g^-(0)))  \right) 
%%%\nonumber \\
%%%&\quad \times 
%%%\frac{ P \left( W_{[0,t]} \in K^+_{[0,t]}(\Phi(g^-_{\varepsilon}(0))-\Phi(g^-(0))), W_{[0,t]}(t) \in dc \right) }
%%%{ dc } \Bigg\vert_{ c = \Phi(y) + \Phi(g^-_{\varepsilon}(0)) - \Phi(g^-_{\varepsilon}(t)) -\Phi(g^-(0))} \nonumber \\
&\quad = 
P \Big( B^{0 \to \Phi(y)+\Phi(g^-_{\varepsilon}(0))-\Phi(g^-_{\varepsilon}(t))-\Phi(g^-(0))}_{[0,t]} \big\vert_{ K^+_{[0,t]}(\Phi(g^-_{\varepsilon}(0))-\Phi(g^-(0)))} 
\nonumber \\
&\quad \qquad 
\in K^-_{[0,t]}( \phi(g^+_{\varepsilon})-\phi(g^-_{\varepsilon}) +\Phi(g^-_{\varepsilon}(0)) -\Phi(g^-(0)))  \Big) 
\nonumber \\
&\qquad \times 
P \left( W_{[0,t]} \in K^+_{[0,t]}(\Phi(g^-_{\varepsilon}(0))-\Phi(g^-(0))) \right) 
\nonumber \\
&\qquad \times 
\frac{ P \left( W_{[0,t]} \big\vert_{K^+_{[0,t]}(\Phi(g^-_{\varepsilon}(0))-\Phi(g^-(0)))} (t) \in dc \right) }{dc} 
\Bigg\vert_{ c = \Phi(y) + \Phi(g^-_{\varepsilon}(0)) - \Phi(g^-_{\varepsilon}(t)) -\Phi(g^-(0))} 
\label{Eq3_Lem_pXconv1_Proof} 
\end{align}
for $y \in (g^-_{\varepsilon}(t), g^+_{\varepsilon}(t))$. 
Using \eqref{Eq1_Lem_pXconv1_Proof}, \eqref{Eq2_Lem_pXconv1_Proof} and \eqref{Eq3_Lem_pXconv1_Proof}, 
we have
\begin{align}
&\frac{P\left( X^{g^-(0) \to y}_{[0,t]} \in K_{[0,t]}(g^-_{\varepsilon}, g^+_{\varepsilon}) \right) }{ g^-(0) - g^-_{\varepsilon}(0) } 
\cdot \frac{ P( X^{g^-(0)}_{[0,t]}(t) \in dy ) }{ dy } 
=\kappa_{\varepsilon}(y) 
\gamma_{\varepsilon}(y) \label{eq_kappa_gamma}
\end{align} 
where
\begin{align*}
&\kappa_{\varepsilon}(y)
:=e^{\Psi(\Phi(y)) - \Psi(\Phi(g^-(0)))} 
E\left[ e^{-\frac{1}{2}N_{[0,t]}( B_{[0,t]}^{ \Phi(g^-(0)), \Phi(y), ( \phi(g^-_{\varepsilon}), \phi(g^+_{\varepsilon}) ) } )} \right] 
\cdot \frac{1}{\sigma(y)} 
\nonumber \\
&\quad \times 
E\left[
\widetilde{Z}_{[0,t]}^{\phi(g^-_{\varepsilon})} 
\big(B^{0, \Phi(y)+\Phi(g^-_{\varepsilon}(0))-\Phi(g^-_{\varepsilon}(t))-\Phi(g^-(0)), 
(\Phi(g^-_{\varepsilon}(0))-\Phi(g^-(0)), \phi(g^+_{\varepsilon})-\phi(g^-_{\varepsilon}) +\Phi(g^-_{\varepsilon}(0))-\Phi(g^-(0)))}_{[0,t]}\big)^{-1} \right] 
\nonumber \\
&\quad \times 
P\Big( B^{0 \to \Phi(y)+\Phi(g^-_{\varepsilon}(0))-\Phi(g^-_{\varepsilon}(t)) - \Phi(g^-(0))}_{[0,t]} 
\big\vert_{K_{[0,t]}^+\big(\Phi(g^-_{\varepsilon}(0)) - \Phi(g^-(0)) \big)} 
\nonumber \\
&\qquad \quad \in K_{[0,t]}^-\big(\phi(g^+_{\varepsilon})-\phi(g^-_{\varepsilon}) +\Phi(g^-_{\varepsilon}(0)) - \Phi(g^-(0)) \big) \Big) 
\nonumber \\
&\quad \times 
\frac{P\left(W_{[0,t]} \in K_{[0,t]}^+\big(\Phi(g^-_{\varepsilon}(0))-\Phi(g^-(0))\big) \right)}{ g^-(0) - g^-_{\varepsilon}(0) } ,
\nonumber \\
&\gamma_{\varepsilon}(y)
:= \frac{ P \left( W_{[0,t]} \big\vert_{K^+_{[0,t]}(\Phi(g^-_{\varepsilon}(0))-\Phi(g^-(0)))} (t) \in dc \right) }{dc} 
\Bigg\vert_{ c = \Phi(y) + \Phi(g^-_{\varepsilon}(0)) - \Phi(g^-_{\varepsilon}(t)) -\Phi(g^-(0))}
\nonumber 
\end{align*}
for $y \in (g^-_{\varepsilon}(t), g^+_{\varepsilon}(t))$, and
\begin{align*}
\kappa_{\varepsilon}(y) &:=0, \quad \gamma_{\varepsilon}(y) := 0, \qquad y \notin (g^-_{\varepsilon}(t), g^+_{\varepsilon}(t)).
\end{align*}
It follows from Lemma~\ref{Lem_conpoundweakconv_2crvs_BESbridge} that
\begin{align}
\lim_{\varepsilon \downarrow 0} 
E\left[ e^{-\frac{1}{2}N_{[0,t]}( B_{[0,t]}^{ \Phi(g^-(0)), \Phi(y), ( \phi(g^-_{\varepsilon}), \phi(g^+_{\varepsilon}) ) } )} \right] 
= E\left[ e^{-\frac{1}{2}N_{[0,t]}( B_{[0,t]}^{ \Phi(g^-(0)), \Phi(y), ( \phi(g^-), \phi(g^+) ) } )} \right]
\label{conv_kappa_ep_parts_1}
\end{align}
for $y \in (g^-(t), g^+(t))$. 
Using 
Proposition~\ref{Prop_Boundary_purturb_WeakConv_for_Meander_BES3bridge} 
and 
Lemma~\ref{Lem_conpoundweakconv_2crvs_BESbridge}, 
we have
\begin{align}
&
E\left[
\widetilde{Z}_{[0,t]}^{\phi(g^-_{\varepsilon})} 
\big(B^{0, \Phi(y)+\Phi(g^-_{\varepsilon}(0))-\Phi(g^-_{\varepsilon}(t))-\Phi(g^-(0)), 
(\Phi(g^-_{\varepsilon}(0))-\Phi(g^-(0)), \phi(g^+_{\varepsilon})-\phi(g^-_{\varepsilon}) +\Phi(g^-_{\varepsilon}(0))-\Phi(g^-(0)))}_{[0,t]}\big)^{-1} \right] 
\nonumber \\
&\to E\left[ \widetilde{Z}^{\phi(g^-)}_{[0,t]}\big( r^{ 0 \to \Phi(y) - \Phi(g^-(t))}_{[0,t]}\big\vert_{K^-_{[0,t]}( \phi(g^+) - \phi(g^-) )} \big)^{-1} \right] 
\qquad (\varepsilon \downarrow 0),
\label{conv_kappa_ep_parts_2}
\end{align}
and
\begin{align}
&
P\Big( B^{0 \to \Phi(y)+\Phi(g^-_{\varepsilon}(0))-\Phi(g^-_{\varepsilon}(t)) - \Phi(g^-(0))}_{[0,t]} 
\big\vert_{K_{[0,t]}^+\big(\Phi(g^-_{\varepsilon}(0)) - \Phi(g^-(0)) \big)} 
\nonumber \\
&\quad \in K_{[0,t]}^-\big(\phi(g^+_{\varepsilon})-\phi(g^-_{\varepsilon}) +\Phi(g^-_{\varepsilon}(0)) - \Phi(g^-(0)) \big) \Big) 
\nonumber \\
&\to 
P\left( r^{0 \to \Phi(y) - \Phi(g^-(t)) }_{[0,t]}
\in K_{[0,t]}^-\big(\phi(g^+)-\phi(g^-) \big) \right) \qquad (\varepsilon \downarrow 0)
\label{conv_kappa_ep_parts_3}
\end{align}
for $y \in (g^-(t), g^+(t))$. Further, we have
\begin{align}
&\lim_{\varepsilon \downarrow 0} 
\frac{ P\left( W_{[0,t]} \in K^+_{[0,t]}( \Phi(g^-_{\varepsilon}(0)) - \Phi(g^-(0)) ) \right) }{ g^-(0) - g^-_{\varepsilon}(0) } 
= \frac{1}{ \sigma(g^-(0)) } \sqrt{\frac{2}{\pi t}} .
\label{conv_kappa_ep_parts_4}
\end{align}
Thus, it follows from 
\eqref{conv_kappa_ep_parts_1}, 
\eqref{conv_kappa_ep_parts_2}, 
\eqref{conv_kappa_ep_parts_3}, and 
\eqref{conv_kappa_ep_parts_4}
that
\begin{align}
\kappa(y) :=& \lim_{\varepsilon \downarrow 0} \kappa_{\varepsilon}(y) \nonumber \\
=& \ e^{\Psi(\Phi(y)) - \Psi(\Phi(g^-(0)))} 
E\left[ e^{-\frac{1}{2}N_{[0,t]}( B_{[0,t]}^{ \Phi(g^-(0)), \Phi(y), ( \phi(g^-), \phi(g^+) ) } )} \right] 
\cdot \frac{1}{\sigma(y)} 
\nonumber \\
&\times 
E\left[ \widetilde{Z}^{\phi(g^-)}_{[0,t]}\big( r^{ 0 \to \Phi(y) - \Phi(g^-(t))}_{[0,t]}\big\vert_{K^-_{[0,t]}( \phi(g^+) - \phi(g^-) )} \big)^{-1} \right] 
\nonumber \\
&\times 
P\left( r^{0 \to \Phi(y) - \Phi(g^-(t)) }_{[0,t]}
\in K_{[0,t]}^-\big(\phi(g^+)-\phi(g^-) \big) \right) 
\frac{1}{ \sigma(g^-(0)) } \sqrt{\frac{2}{\pi t}}  
\label{conv_kappa_ep}
\end{align}
for $y \in (g^-(t), g^+(t))$. 
Because 
$W_{[0,t]}\big\vert_{K_{[0,t]}^+\big(\Phi(g^-_{\varepsilon}(0))-\Phi(g^-(0))\big)} 
\overset{\mathcal{D}}{\to} 
W^+_{[0,t]}$ $(\varepsilon \downarrow 0)$ holds, we have
\begin{align}
\gamma(y) :=
\lim_{\varepsilon \downarrow 0} \gamma_{\varepsilon}(y) =
\frac{P\left( W_{[0,t]}^+ (t) \in dc\right)}{dc} \Bigg\vert_{ c = \Phi(y)-\Phi(g^-(t)) } , 
\qquad y \in (g^-(t), g^+(t)).
\label{conv_gamma_ep}
\end{align}
Therefore, using \eqref{eq_kappa_gamma}, \eqref{conv_kappa_ep}, and \eqref{conv_gamma_ep}, we obtain \eqref{p_upper_expression}.

For $y \notin (g^-(t), g^+(t))$, we define $\kappa(y) = \gamma(y) = 0$. Let 
\begin{align*}
&\delta := \sup_{0 < \varepsilon \leq 1} |g^-_{\varepsilon}(t) - g^-(t) | \vee | g^+_{\varepsilon}(t) - g^+(t) | .
\end{align*}
Then, using the change of variables $z = \Phi(y)$, we obtain the following inequality:
\begin{align}
&\int_{\mathbb{R}} | \gamma_{\varepsilon}(y) - \gamma(y) | \, dy 
= \int_{g^-(t)-\delta}^{g^-(t)+\delta} | \gamma_{\varepsilon}(y) - \gamma(y) | \, dy
\nonumber \\
&\quad = \int_{\Phi( g^-(t)-\delta )}^{\Phi( g^+(t) + \delta)}  
| \gamma_{\varepsilon}(\Phi^{-1}(z)) - \gamma(\Phi^{-1}(z)) | 
\cdot \sigma(\Phi^{-1}(z)) \, dz 
\nonumber \\
&\quad \leq \left\{ \sup_{y \in (g^-(t) - \delta, g^+(t) + \delta)} \sigma(y) \right\} \cdot 
\int_{\Phi( g^-(t) - \delta )}^{\Phi( g^+(t) + \delta )} 
| \gamma_{\varepsilon}(\Phi^{-1}(z)) - \gamma(\Phi^{-1}(z)) |  \, dz 
\qquad (0<\varepsilon \leq 1).
\label{ineq_intg_gammaep_mins_gamm_dy}
\end{align}
Thus, using \eqref{ineq_intg_gammaep_mins_gamm_dy} and Lemma~\ref{Lem_density_ae_convergence}, we obtain
\begin{align}\label{convergence_intg_gammaep_mins_gamm_dy}
\lim_{\varepsilon \downarrow 0} \int_{\mathbb{R}} | \gamma_{\varepsilon}(y) - \gamma(y) | dy = 0 .
\end{align}
Further, by the definition of $\kappa_{\varepsilon}(y)$, it holds that
\begin{align}\label{ineq_uniformbound_sup_esssup_kappaep_y}
\sup_{0 < \varepsilon \leq 1} \esssup_{y \in \mathbb{R}} | \kappa_{\varepsilon}(y) | < \infty .
\end{align}
Therefore, using 
\eqref{convergence_intg_gammaep_mins_gamm_dy}, 
\eqref{ineq_uniformbound_sup_esssup_kappaep_y}, and 
Lemma~\ref{Ap_Lem_ScheffeType}, we have
\begin{align*}
\lim_{\varepsilon \downarrow 0} \int_{\mathbb{R}} | \gamma_{\varepsilon}(y) \kappa_{\varepsilon}(y) - \gamma(y) \kappa(y) | dy = 0 .
\end{align*}
Hence, we obtain \eqref{L1_conv_p_upper}. 
\end{proof}

\begin{lem}\label{Lem_pX_conv_2}
For $0 < t < 1$, we have
\begin{align}
&p^{(g^-,g^+),(\downarrow)}_{[t,1]}(y)
=\lim_{\varepsilon \downarrow 0} 
\frac{ P( X^{y \to g^+(1) }_{[t,1]} \in K_{[t,1]}(g^-_{\varepsilon}, g^+_{\varepsilon}) ) }{ g^+_{\varepsilon}(1) - g^+(1) }  
\cdot
\left.\frac{ P( X^{y}_{[t,1]}(1) \in dc ) }{ dc }\right|_{c = g^+(1)} 
\nonumber \\
&\quad = 
e^{ \Psi(\Phi(g^+(1))) - \Psi(\Phi(y)) } 
E\left[ e^{-\frac{1}{2}N_{[t,1]}( B^{ \Phi(y), \Phi(g^+(1)), ( \phi(g^-), \phi(g^+) ) }_{[t,1]} )} \right] 
\cdot \frac{1}{\sigma(g^+(1))}
\label{p_down_expression} \\
&\qquad \times E\left[ \widetilde{Z}^{ (-1) \cdot \overleftarrow{\phi(g^+)} }_{[t,1]}( r^{ 0 \to \Phi(g^+(t)) - \Phi(y) }_{[t,1]}\mid_{K^-_{[t,1]}(\overleftarrow{\phi(g^+)} - \overleftarrow{\phi(g^-)})} )^{-1} \right] 
\nonumber \\
&\qquad \times 
P\left( r^{0 \to \Phi(g^+(t)) - \Phi(y)}_{[t,1]} \in  K^-_{[t,1]}(\overleftarrow{\phi(g^+)} - \overleftarrow{\phi(g^-)}) \right) 
\nonumber \\
&\qquad \times \frac{1}{ \sigma(g^+(1)) } \sqrt{\frac{2}{\pi (1-t)}}\cdot 
\frac{ P\left( W^+_{[t,1]} (1) \in dc \right)}{dc}\Bigg|_{c = \Phi(g^+(t))-\Phi(y)}
\qquad y\in (g^-(t), g^+(t)),
\nonumber 
\end{align}
and
\begin{align}\label{L1_conv_p_down}
\lim_{\varepsilon \downarrow 0} \int_{\mathbb{R}} \left| 
\frac{ P( X^{y \to g^+(1) }_{[t,1]} \in K_{[t,1]}(g^-_{\varepsilon}, g^+_{\varepsilon}) ) }{ g^+_{\varepsilon}(1) - g^+(1) } 
\cdot 
\left.\frac{ P( X^{y}_{[t,1]}(1) \in dc ) }{ dc }\right|_{c = g^+(1)} - p^{(g^-,g^+),(\downarrow)}_{[t,1]}(y)
 \right| dy = 0 .
\end{align}
\end{lem}
\begin{proof}
Using Lemma~\ref{girsanov}, we obtain
\begin{align}
&\left.\frac{P( X^{y \to g^+(1)}_{[t,1]} \in K_{[t,1]}(g^-_{\varepsilon}, g^+_{\varepsilon}) )}{ g^+_{\varepsilon}(1) - g^+(1) }  
\cdot 
\frac{ P( X^{y}_{[t,1]}(1) \in dc ) }{ dc } \right|_{c = g^+(1)} \nonumber \\
&\quad = 
\frac{P( Y^{\Phi(y) \to \Phi(g^+(1)) }_{[t,1]} \in K_{[t,1]}( \phi(g^-_{\varepsilon}), \phi(g^+_{\varepsilon}) ) )}{ g^+_{\varepsilon}(1) - g^+(1) }
\cdot
 \left.\frac{ P( Y^{\Phi(y)}_{[t,1]}(1) \in dc ) }{ dc }\right|_{c = \Phi(g^+(1))} \cdot \frac{1}{\sigma(g^+(1))}  \nonumber \\
&\quad =
e^{\Psi(\Phi(g^+(1))) - \Psi(\Phi(y))} 
E\left[ e^{-\frac{1}{2}N_{[t,1]}( B_{[t,1]}^{ \Phi(y), \Phi(g^+(1)), ( \phi(g^-_{\varepsilon}), \phi(g^+_{\varepsilon}) ) } )} \right] 
\cdot \frac{1}{\sigma(g^+(1))}  \nonumber \\
&\qquad \times 
\frac{P( B_{[t,1]}^{\Phi(y) \to \Phi(g^+(1)) } \in K_{[t,1]}( \phi(g^-_{\varepsilon}), \phi(g^+_{\varepsilon}) ) )}{ g^+_{\varepsilon}(1) - g^+(1) } 
\cdot
\left.\frac{ P( W^{\Phi(y)}_{[t,1]}(1) \in dc ) }{ dc }\right|_{c = \Phi(g^+(1))} \label{Eq1_Lem_pXconv1_Proof_down}
\end{align}
for $g^-_{\varepsilon}(t)<y<g^+_{\varepsilon}(t)$. 
It follows from Lemma~A.10 in \cite{bib_ishitani_hatakenaka_suzuki} that
\begin{align}
&P\left( B^{\Phi(y) \to \Phi(g^+(1)) }_{[t,1]} \in K_{[t,1]}( \phi(g^-_{\varepsilon}), \phi(g^+_{\varepsilon}) ) \right)  
\cdot \left.\frac{ P( W^{\Phi(y)}_{[t,1]}(1) \in dc ) }{ dc }\right|_{c = \Phi(g^+(1))} 
\nonumber \\
%%%&= 
%%%P\left( B^{\Phi(y) \to c }_{[t,1]} \in K_{[t,1]}( \phi(g^-_{\varepsilon}), \phi(g^+_{\varepsilon}) ) \right) 
%%%\cdot \left.\frac{ P( W^{\Phi(y)}_{[t,1]}(1) \in dc ) }{ dc }\right|_{c = \Phi(g^+(1))} 
%%%\nonumber \\
&= 
\left. \frac{P\left( W^{\Phi(y)}_{[t,1]} \in K_{[t,1]}( \phi(g^-_{\varepsilon}), \phi(g^+_{\varepsilon}) ), 
\ W^{\Phi(y)}_{[t,1]}(1)\in dc \right) }{dc}\right|_{c = \Phi(g^+(1))} 
\nonumber \\
%%%&= 
%%%E\left[
%%%\widetilde{Z}_{[t,1]}^{-\overleftarrow{\phi(g^+_{\varepsilon})}} \Big(
%%%B^{0, c-\Phi(y)+\Phi(g^+_{\varepsilon}(t))-\Phi(g^+_{\varepsilon}(1)), 
%%%(c-\Phi(g^+_{\varepsilon}(1)), \overleftarrow{\phi(g^+_{\varepsilon})}-\overleftarrow{\phi(g^-_{\varepsilon})}+c-\Phi(g^+_{\varepsilon}(1)))}_{[t,1]}\Big)^{-1} 
%%%\right] \nonumber \\ 
%%%&\quad \times
%%%P\left( B^{0\to c-\Phi(y)+\Phi(g^+_{\varepsilon}(t))-\Phi(g^+_{\varepsilon}(1))}_{[t,1]}
%%%\in K_{[t,1]}\big(c-\Phi(g^+_{\varepsilon}(1)), \overleftarrow{\phi(g^+_{\varepsilon})}-\overleftarrow{\phi(g^-_{\varepsilon})}+c-\Phi(g^+_{\varepsilon}(1))\big)\right)
%%%\nonumber \\
%%%&\quad \times \frac{P\Big(\Phi(g^+_{\varepsilon}(1))-\Phi(g^+_{\varepsilon}(t))+W^{\Phi(y)}_{[t,1]}(1) \in dc\Big)}{dc}
%%%\Bigg|_{c = \Phi(g^+(1))} 
%%%\nonumber \\
&= 
E\left[
\widetilde{Z}_{[t,1]}^{-\overleftarrow{\phi(g^+_{\varepsilon})}} \Big(
B^{0, \Phi(g^+(1))-\Phi(y)+\Phi(g^+_{\varepsilon}(t))-\Phi(g^+_{\varepsilon}(1)), 
(\Phi(g^+(1))-\Phi(g^+_{\varepsilon}(1)), \overleftarrow{\phi(g^+_{\varepsilon})}-\overleftarrow{\phi(g^-_{\varepsilon})}+\Phi(g^+(1))-\Phi(g^+_{\varepsilon}(1)))}_{[t,1]}\Big)^{-1} 
\right] \nonumber \\ 
&\quad \times
P\bigg( B^{0\to \Phi(g^+(1))-\Phi(y)+\Phi(g^+_{\varepsilon}(t))-\Phi(g^+_{\varepsilon}(1))}_{[t,1]} 
\nonumber \\
&\qquad \quad
\in K_{[t,1]}\big(\Phi(g^+(1))-\Phi(g^+_{\varepsilon}(1)), 
\overleftarrow{\phi(g^+_{\varepsilon})}-\overleftarrow{\phi(g^-_{\varepsilon})}+\Phi(g^+(1))-\Phi(g^+_{\varepsilon}(1))\big)\bigg) 
\nonumber \\
&\quad \times \frac{P\Big(W_{[t,1]}(1) \in dc\Big)}{dc}
\Bigg|_{c = \Phi(g^+(1)) -\Phi(y)+\Phi(g^+_{\varepsilon}(t))- \Phi(g^+_{\varepsilon}(1))} 
\qquad (g^-_{\varepsilon}(t)<y<g^+_{\varepsilon}(t)).
\label{Eq2_Lem_pXconv1_Proof_down} 
\end{align}
On the other hand, we have
\begin{align}
&P\Big( B^{0\to \Phi(g^+(1))-\Phi(y)+\Phi(g^+_{\varepsilon}(t))-\Phi(g^+_{\varepsilon}(1))}_{[t,1]}
\nonumber \\
&\quad
\in K_{[t,1]}\big(\Phi(g^+(1))-\Phi(g^+_{\varepsilon}(1)), 
\overleftarrow{\phi(g^+_{\varepsilon})}-\overleftarrow{\phi(g^-_{\varepsilon})}+\Phi(g^+(1))-\Phi(g^+_{\varepsilon}(1))\big)\Big) 
\nonumber \\
&\times 
\frac{P\Big(W_{[t,1]}(1) \in dc\Big)}{dc}
\Bigg|_{c = \Phi(g^+(1)) -\Phi(y)+\Phi(g^+_{\varepsilon}(t))- \Phi(g^+_{\varepsilon}(1))} 
\nonumber \\
%%%&=P\left( B^{0\to \Phi(g^+(1))-\Phi(y)+\Phi(g^+_{\varepsilon}(t))-\Phi(g^+_{\varepsilon}(1))}_{[t,1]}
%%%\big\vert_{K^+_{[t,1]}\big(\Phi(g^+(1))-\Phi(g^+_{\varepsilon}(1))\big)}
%%%\in K^-_{[t,1]}\big(\overleftarrow{\phi(g^+_{\varepsilon})}-\overleftarrow{\phi(g^-_{\varepsilon})}+\Phi(g^+(1))-\Phi(g^+_{\varepsilon}(1))\big)\right) 
%%%\nonumber \\
%%%&\quad \times 
%%%P\left( 
%%%B^{0\to \Phi(g^+(1))-\Phi(y)+\Phi(g^+_{\varepsilon}(t))-\Phi(g^+_{\varepsilon}(1))}_{[t,1]}
%%%\in K^+_{[t,1]}\big(\Phi(g^+(1))-\Phi(g^+_{\varepsilon}(1))\big)
%%%\right) 
%%%\nonumber \\
%%%&\quad \times 
%%%\frac{P\Big(W_{[t,1]}(1) \in dc\Big)}{dc}
%%%\Bigg|_{c = \Phi(g^+(1))-\Phi(y)+\Phi(g^+_{\varepsilon}(t))-\Phi(g^+_{\varepsilon}(1))} 
%%%\nonumber \\
%%%&=P\left( B^{0\to \Phi(g^+(1))-\Phi(y)+\Phi(g^+_{\varepsilon}(t))-\Phi(g^+_{\varepsilon}(1))}_{[t,1]}
%%%\big\vert_{K^+_{[t,1]}\big(\Phi(g^+(1))-\Phi(g^+_{\varepsilon}(1))\big)}
%%%\in K^-_{[t,1]}\big(\overleftarrow{\phi(g^+_{\varepsilon})}-\overleftarrow{\phi(g^-_{\varepsilon})}+\Phi(g^+(1))-\Phi(g^+_{\varepsilon}(1))\big)\right) 
%%%\nonumber \\
%%%&\quad \times
%%%\frac{ P\left( 
%%%W_{[t,1]} \in K^+_{[t,1]}\big(\Phi(g^+(1))-\Phi(g^+_{\varepsilon}(1))\big), 
%%%W_{[t,1]}(1) \in dc
%%%\right)}
%%%{P\left( W_{[t,1]}(1)\in dc\right)} 
%%%\nonumber \\
%%%&\quad \times 
%%%\frac{P\Big(W_{[t,1]}(1) \in dc\Big)}{dc}
%%%\Bigg|_{c = \Phi(g^+(1))-\Phi(y)+\Phi(g^+_{\varepsilon}(t))-\Phi(g^+_{\varepsilon}(1))} 
%%%\nonumber \\
&=P\Big( B^{0\to \Phi(g^+(1))-\Phi(y)+\Phi(g^+_{\varepsilon}(t))-\Phi(g^+_{\varepsilon}(1))}_{[t,1]}
\big\vert_{K^+_{[t,1]}\big(\Phi(g^+(1))-\Phi(g^+_{\varepsilon}(1))\big)}
\nonumber \\
&\qquad 
\in K^-_{[t,1]}\big(\overleftarrow{\phi(g^+_{\varepsilon})}-\overleftarrow{\phi(g^-_{\varepsilon})}+\Phi(g^+(1))-\Phi(g^+_{\varepsilon}(1))\big)\Big) 
\nonumber \\
&\quad \times
\frac{ P\left( 
W_{[t,1]} \in K^+_{[t,1]}\big(\Phi(g^+(1))-\Phi(g^+_{\varepsilon}(1))\big), 
W_{[t,1]}(1) \in dc
\right)}
{dc}
\Bigg|_{c = \Phi(g^+(1))-\Phi(y)+\Phi(g^+_{\varepsilon}(t))-\Phi(g^+_{\varepsilon}(1))} 
\nonumber \\
&=P\Big( B^{0\to \Phi(g^+(1))-\Phi(y)+\Phi(g^+_{\varepsilon}(t))-\Phi(g^+_{\varepsilon}(1))}_{[t,1]}
\big\vert_{K^+_{[t,1]}\big(\Phi(g^+(1))-\Phi(g^+_{\varepsilon}(1))\big)}
\nonumber \\
&\qquad 
\in K^-_{[t,1]}\big(\overleftarrow{\phi(g^+_{\varepsilon})}-\overleftarrow{\phi(g^-_{\varepsilon})}+\Phi(g^+(1))-\Phi(g^+_{\varepsilon}(1))\big)\Big) 
\nonumber \\
&\quad \times
P\left( 
W_{[t,1]} \in K^+_{[t,1]}\big(\Phi(g^+(1))-\Phi(g^+_{\varepsilon}(1))\big)
\right) \nonumber \\
&\quad \times
\frac{ P\left( 
W_{[t,1]}\vert_{K^+_{[t,1]}(\Phi(g^+(1))-\Phi(g^+_{\varepsilon}(1)))} (1) \in dc
\right)}{dc}
\Bigg|_{c = \Phi(g^+(1))-\Phi(y)+\Phi(g^+_{\varepsilon}(t))-\Phi(g^+_{\varepsilon}(1))} 
\label{Eq3_Lem_pXconv1_Proof_down} 
\end{align}
for $y \in (g^-_{\varepsilon}(t), g^+_{\varepsilon}(t))$. 
Using \eqref{Eq1_Lem_pXconv1_Proof_down}, \eqref{Eq2_Lem_pXconv1_Proof_down} and \eqref{Eq3_Lem_pXconv1_Proof_down}, 
we have
\begin{align}
&\left.\frac{P( X^{y \to g^+(1)}_{[t,1]} \in K_{[t,1]}(g^-_{\varepsilon}, g^+_{\varepsilon}) )}{ g^+_{\varepsilon}(1) - g^+(1) } 
\cdot \frac{ P( X^{y}_{[t,1]}(1) \in dc ) }{ dc } \right|_{c = g^+(1)}
=\kappa_{\varepsilon}(y) 
\gamma_{\varepsilon}(y) \label{eq_kappa_gamma_down}
\end{align} 
where
\begin{align*}
\kappa_{\varepsilon}(y)
&:=e^{\Psi(\Phi(g^+(1))) - \Psi(\Phi(y))} 
E[ e^{-\frac{1}{2}N_{[t,1]}( B_{[t,1]}^{ \Phi(y), \Phi(g^+(1)), ( \phi(g^-_{\varepsilon}), \phi(g^+_{\varepsilon}) ) } )} ] \cdot \frac{1}{\sigma(g^+(1))} 
\nonumber \\
&\quad \times 
E\left[
\widetilde{Z}_{[t,1]}^{-\overleftarrow{\phi(g^+_{\varepsilon})}} \Big(
B^{0, \Phi(g^+(1))-\Phi(y)+\Phi(g^+_{\varepsilon}(t))-\Phi(g^+_{\varepsilon}(1)), 
(\Phi(g^+(1))-\Phi(g^+_{\varepsilon}(1)), \overleftarrow{\phi(g^+_{\varepsilon})}-\overleftarrow{\phi(g^-_{\varepsilon})}+\Phi(g^+(1))-\Phi(g^+_{\varepsilon}(1)))}_{[t,1]}\Big)^{-1} 
\right] \nonumber \\ 
&\quad \times 
P\Big( B^{0\to \Phi(g^+(1))-\Phi(y)+\Phi(g^+_{\varepsilon}(t))-\Phi(g^+_{\varepsilon}(1))}_{[t,1]}
\big\vert_{K^+_{[t,1]}\big(\Phi(g^+(1))-\Phi(g^+_{\varepsilon}(1))\big)}
\nonumber \\
&\qquad \quad
\in K^-_{[t,1]}\big(\overleftarrow{\phi(g^+_{\varepsilon})}-\overleftarrow{\phi(g^-_{\varepsilon})}+\Phi(g^+(1))-\Phi(g^+_{\varepsilon}(1))\big)\Big)  
\nonumber \\
&\quad \times \frac{ P\left( 
W_{[t,1]} \in K^+_{[t,1]}\big(\Phi(g^+(1))-\Phi(g^+_{\varepsilon}(1))\big)
\right) }{ g^+_{\varepsilon}(1) - g^+(1) }, \nonumber \\
\gamma_{\varepsilon}(y)
&:= \frac{ P\left( 
W_{[t,1]}\vert_{K^+_{[t,1]}(\Phi(g^+(1))-\Phi(g^+_{\varepsilon}(1)))} (1) \in dc
\right)}{dc}
\Bigg|_{c = \Phi(g^+(1))-\Phi(y)+\Phi(g^+_{\varepsilon}(t))-\Phi(g^+_{\varepsilon}(1))} , \nonumber
\end{align*}
for $y \in (g^-_{\varepsilon}(t), g^+_{\varepsilon}(t))$, and
\begin{align*}
\kappa_{\varepsilon}(y) :=0, \quad  \gamma_{\varepsilon}(y) := 0, 
\qquad y \notin (g^-_{\varepsilon}(t), g^+_{\varepsilon}(t)).
\end{align*}
It follows from Lemma~\ref{Lem_conpoundweakconv_2crvs_BESbridge} that
\begin{align}
\lim_{\varepsilon \downarrow 0} 
E\left[ e^{-\frac{1}{2}N_{[t,1]}( B_{[t,1]}^{ \Phi(y), \Phi(g^+(1)), ( \phi(g^-_{\varepsilon}), \phi(g^+_{\varepsilon}) ) } )} \right] 
= E\left[ e^{-\frac{1}{2}N_{[t,1]}( B_{[t,1]}^{ \Phi(y), \Phi(g^+(1)), ( \phi(g^-), \phi(g^+) ) } )} \right]
\label{conv_dwn_kappa_ep_parts_1}
\end{align}
for $y \in (g^-(t), g^+(t))$. 
Using 
Proposition~\ref{Prop_Boundary_purturb_WeakConv_for_Meander_BES3bridge} 
and 
Lemma~\ref{Lem_conpoundweakconv_2crvs_BESbridge}, 
we have
\begin{align}
&
E\Big[\widetilde{Z}_{[t,1]}^{-\overleftarrow{\phi(g_{\varepsilon}^+)}} \big(B^{0 , \Phi(g^+(1)) - \Phi(y) + \Phi(g_{\varepsilon}^+(t)) - \Phi(g_{\varepsilon}^+(1)) , ( \Phi(g^+(1)) - \Phi(g_{\varepsilon}^+(1)) , \overleftarrow{\phi(g_{\varepsilon}^+)} - \overleftarrow{\phi(g_{\varepsilon}^-)} + \Phi(g^+(1)) - \Phi(g_{\varepsilon}^+(1)) ) }_{[t,1]}\big)^{-1} \Big] 
\nonumber \\
&\to E\left[ \widetilde{Z}^{ -\overleftarrow{\phi(g^+)} }_{[t,1]}( r^{ 0 \to \Phi(g^+(t)) - \Phi(y)}_{[t,1]}\mid_{K^-_{[t,1]}( \overleftarrow{\phi(g^+)} - \overleftarrow{\phi(g^-)} )} )^{-1} \right] \qquad (\varepsilon \downarrow 0),
\label{conv_dwn_kappa_ep_parts_2}
\end{align}
and
\begin{align}
& 
P\Big( B^{0 \to \Phi(g^+(1)) - \Phi(y) + \Phi(g^+_{\varepsilon}(t)) - \Phi(g^+_{\varepsilon}(1))}_{[t,1]}\mid_{ K^+_{[t,1]}( \Phi(g^+(1)) - \Phi(g_{\varepsilon}^+(1)) ) } 
\nonumber \\
&\quad \in K^-_{[t,1]}( \overleftarrow{\phi(g_{\varepsilon}^+)} - \overleftarrow{\phi(g_{\varepsilon}^-)} + \Phi(g^+(1)) - \Phi(g_{\varepsilon}^+(1))) \Big) 
\nonumber \\
&\to P\Big( r^{0 \to \Phi(g^+(t)) - \Phi(y)}_{[t,1]} \in  K^-_{[t,1]}(\overleftarrow{\phi(g^+)} - \overleftarrow{\phi(g^-)}) \Big) 
\qquad (\varepsilon \downarrow 0)
\label{conv_dwn_kappa_ep_parts_3}
\end{align}
for $y \in (g^-(t), g^+(t))$.
Further, we have
\begin{align}\label{conv_dwn_kappa_ep_parts_4}
&\lim_{\varepsilon \downarrow 0} 
\frac{P\big(W_{[t,1]} \in K^+_{[t,1]}\big(\Phi(g^+(1))-\Phi(g^+_{\varepsilon}(1))\big)}{ g^+_{\varepsilon}(1) - g^+(1) } 
= \frac{1}{ \sigma(g^+(1)) } \sqrt{\frac{2}{\pi (1-t)}} .
\end{align}
Thus, it follows from 
\eqref{conv_dwn_kappa_ep_parts_1}, 
\eqref{conv_dwn_kappa_ep_parts_2}, 
\eqref{conv_dwn_kappa_ep_parts_3}, and 
\eqref{conv_dwn_kappa_ep_parts_4}
that
\begin{align}
\kappa(y) :=& \lim_{\varepsilon \downarrow 0} \kappa_{\varepsilon}(y) \nonumber \\
=& \ e^{\Psi(\Phi(g^+(1))) - \Psi(\Phi(y))}  
E\left[ e^{-\frac{1}{2}N_{[t,1]}( B_{[t,1]}^{ \Phi(y), \Phi(g^+(1)), ( \phi(g^-), \phi(g^+) ) } )} \right] 
\cdot \frac{1}{\sigma(g^+(1))} \nonumber \\
&\times 
E\left[ \widetilde{Z}^{ -\overleftarrow{\phi(g^+)} }_{[t,1]}
( r^{ 0 \to \Phi(g^+(t)) - \Phi(y)}_{[t,1]}\mid_{K^-_{[t,1]}( \overleftarrow{\phi(g^+)} - \overleftarrow{\phi(g^-)} )} )^{-1} \right]
\nonumber \\
&\times 
P\left( r^{0 \to \Phi(g^+(t)) - \Phi(y)}_{[t,1]} \in  K^-_{[t,1]}\Big(\overleftarrow{\phi(g^+)} - \overleftarrow{\phi(g^-)}\Big) \right) 
\frac{1}{ \sigma(g^+(1)) } \sqrt{\frac{2}{\pi (1-t)}} \label{conv_kappa_ep_down} 
\end{align}
for $y \in (g^-(t), g^+(t))$.
Because 
$W_{[t,1]}\vert_{K^+_{[t,1]}(\Phi(g^+(1))-\Phi(g^+_{\varepsilon}(1)))} \overset{\mathcal{D}}{\to} W^+_{[t,1]}$ 
$(\varepsilon \downarrow 0)$ holds, we have
\begin{align}
\gamma(y):=&\lim_{\varepsilon \downarrow 0} \gamma_{\varepsilon}(y) =
\frac{ P\left( 
W^+_{[t,1]} (1) \in dc
\right)}{dc}
\Bigg|_{c = \Phi(g^+(t))-\Phi(y)} ,
\quad y \in (g^-(t), g^+(t)).
\label{conv_gamma_ep_down} 
\end{align}
Therefore, using \eqref{eq_kappa_gamma_down}, \eqref{conv_kappa_ep_down}, and \eqref{conv_gamma_ep_down}, 
we obtain \eqref{p_down_expression}.

For $y \notin (g^-(t), g^+(t))$, 
we define $\kappa(y) = \gamma(y) = 0$ . Let 
\begin{align*}
\delta := \sup_{0 < \varepsilon \leq 1} |g^-_{\varepsilon}(t) - g^-(t) | \vee | g^+_{\varepsilon}(t) - g^+(t) |.
\end{align*}
Then, usign the change of variables $z = \Phi(y)$, we obtain the following inequality:
\begin{align}
&\int_{\mathbb{R}} | \gamma_{\varepsilon}(y) - \gamma(y) | dy 
= \int_{g^-(t) -\delta }^{ g^+(t)+\delta } 
| \gamma_{\varepsilon}(y) - \gamma(y) | \, dy \nonumber \\
&\quad = \int_{\Phi( g^-(t) -\delta)}^{\Phi( g^+(t)+\delta )}  
| \gamma_{\varepsilon}(\Phi^{-1}(z)) - \gamma(\Phi^{-1}(z)) |
\cdot \sigma(\Phi^{-1}(z)) \, dz \nonumber \\
&\quad \leq \left\{ \sup_{y \in (g^-(t) - \delta, g^+(t) + \delta)} \sigma(y) \right\}\cdot 
\int_{\Phi( g^-(t) - \delta )}^{\Phi( g^+(t) + \delta )} 
| \gamma_{\varepsilon}(\Phi^{-1}(z)) - \gamma(\Phi^{-1}(z)) | \, dz
\qquad (0<\varepsilon \leq 1).
\label{ineq_dwn_intg_gammaep_mins_gamm_dy}
\end{align}
Then, using \eqref{ineq_dwn_intg_gammaep_mins_gamm_dy} and Lemma~\ref{Lem_density_ae_convergence}, we obtain
\begin{align}\label{convergence_intg_dwn_gammaep_mins_gamm_dy}
\lim_{\varepsilon \downarrow 0} \int_{\mathbb{R}} | \gamma_{\varepsilon}(y) - \gamma(y) | dy = 0 .
\end{align}
Further, by the definition of $\kappa_{\varepsilon}(y)$, it holds that
\begin{align}\label{ineq_dwn_uniformbound_sup_esssup_kappaep_y}
\sup_{0 < \varepsilon \leq 1} \esssup_{y \in \mathbb{R}} | \kappa_{\varepsilon}(y) | < \infty .
\end{align}
Therefore, by using 
\eqref{convergence_intg_dwn_gammaep_mins_gamm_dy}, 
\eqref{ineq_dwn_uniformbound_sup_esssup_kappaep_y}, and 
Lemma~\ref{Ap_Lem_ScheffeType}, we have
\begin{align*}
\lim_{\varepsilon \downarrow 0} \int_{\mathbb{R}} | \gamma_{\varepsilon}(y) \kappa_{\varepsilon}(y) - \gamma(y) \kappa(y) | dy = 0 .
\end{align*}
Hence, we obtain \eqref{L1_conv_p_down}. 
\end{proof}

Using the fact that 
$P\left( W_{[s, t]} |_{K_{[s, t]}^+(-\varepsilon )}(u) \in dy \right) / dy$ and $P\left( W_{[s, t]}^+(u) \in dy \right) / dy$ 
are uniformly bounded with respect to $y \in \mathbb{R}$ and $0 < \varepsilon \leq 1$, 
we can prove the following Lemma in the same way as 
Lemma~\ref{Lem_pX_conv_1} and Lemma~\ref{Lem_pX_conv_2}.

\begin{lem}\label{scheffe_pp_intg}
Let $\Gamma$ and $\Gamma_{\varepsilon}$ $(\varepsilon>0)$ be $\mathbb{R}$-valued functions on $\mathbb{R}$ that satisfy
\begin{align*}
&\esssup_{y \in \mathbb{R}} | \Gamma(y) | < \infty ,\quad
\sup_{0 < \varepsilon \leq 1} \esssup_{y \in \mathbb{R}} | \Gamma_{\varepsilon}(y) | < \infty , \\
&\lim_{\varepsilon \downarrow 0} \esssup_{y \in \mathbb{R}} | \Gamma_{\varepsilon}(y) - \Gamma(y) | =0.
\end{align*}
For $0 < t< 1$, we have
\begin{align*}
&\lim_{\varepsilon \downarrow 0} 
\int_{\mathbb{R}} \, \bigg| \, 
\Gamma_{\varepsilon}(y)\cdot
\frac{ P( X^{g^-(0) \to y}_{[0,t]} \in K_{[0,t]}(g^-_{\varepsilon}, g^+_{\varepsilon}) ) }{ g^-(0) - g^-_{\varepsilon}(0) } 
\cdot \frac{ P( X^{g^-(0)}_{[0,t]}(t) \in dy ) }{ dy } \\
&\qquad \qquad \quad
\times \frac{ P( X^{y \to g^+(1) }_{[t_2,1]} \in K_{[t,1]}(g^-_{\varepsilon}, g^+_{\varepsilon}) ) }{ g^+_{\varepsilon}(1) - g^+(1) } 
\cdot 
\frac{ P( X^{y}_{[t,1]}(1) \in dc ) }{ dc }\bigg\vert_{c = g^+(1)}  \\
&\qquad \ \quad 
-\Gamma(y)\cdot p^{(g^-,g^+),(\uparrow)}_{[0,t]}(y) \cdot p^{(g^-,g^+),(\downarrow)}_{[t,1]}(y) \, \bigg| \, dy = 0.
\end{align*}
Let $\Delta$ and $\Delta_{\varepsilon}$ $(\varepsilon>0)$ be $\mathbb{R}$-valued functions on $\mathbb{R}^2$ that satisfy
\begin{align*}
&\esssup_{(x,y) \in \mathbb{R}^2} | \Delta(x, y) | < \infty ,\quad
\sup_{0 < \varepsilon \leq 1} \esssup_{{(x,y) \in \mathbb{R}^2}} | \Delta_{\varepsilon}(x, y) | < \infty , \\
&\lim_{\varepsilon \downarrow 0} \esssup_{{(x,y) \in \mathbb{R}^2}} | \Delta_{\varepsilon}(x, y) - \Delta(x, y) | =0.
\end{align*}
For $0 < t_1 <t_2 < 1$, we have
\begin{align*}
&\lim_{\varepsilon \downarrow 0} 
\int_{\mathbb{R}} \int_{\mathbb{R}} \, \bigg| \, 
\Delta_{\varepsilon}(x, y)\cdot 
\frac{ P( X^{g^-(0) \to x}_{[0,t_1]} \in K_{[0,t_1]}(g^-_{\varepsilon}, g^+_{\varepsilon}) ) }{ g^-(0) - g^-_{\varepsilon}(0) } 
\cdot \frac{ P( X^{g^-(0)}_{[0,t_1]}(t_1) \in dx) }{ dx } \\
&\qquad \qquad \qquad \quad \ \times 
\frac{ P( X^{y \to g^+(1) }_{[t_2,1]} \in K_{[t_2,1]}(g^-_{\varepsilon}, g^+_{\varepsilon}) ) }{ g^+_{\varepsilon}(1) - g^+(1) } 
\cdot 
\frac{ P( X^{y}_{[t_2,1]}(1) \in dc ) }{ dc }\bigg\vert_{c = g^+(1)} \\
&\qquad \qquad \quad  - 
\Delta(x, y)\cdot 
p^{(g^-,g^+),(\uparrow)}_{[0,t_1]}(x) \cdot 
p^{(g^-,g^+),(\downarrow)}_{[t_2,1]}(y) \, \bigg| \, dx \, dy = 0.
\end{align*}
\end{lem}

%%%%%%%%%%%%%%%%%%%%%%%%%%%%%%%%%%%%%%%%%%%%%%
\subsection{Proof of Theorem~\ref{Thm_Def_and_Decomp_curved_Moving}}
%%%%%%%%%%%%%%%%%%%%%%%%%%%%%%%%%%%%%%%%%%%%%%

In this subsection, we assume that all $X_{[s, t]}^{x,y,(h^-, h^+)}$ are independent. 
For each $\mathbb{R}$-valued bounded continuous function $G$ on $C([0,1], \mathbb{R})$ and $\varepsilon>0$, we define
\begin{align*}
I(\varepsilon, G):=
E[G(X^{a\to b}_{[0,1]})\ ;\ 
X^{a\to b}_{[0,1]} \in K_{[0, 1]}(g^-_{\varepsilon}, g^+_{\varepsilon})]
\cdot \frac{P(X^a_{[0,1]}(1) \in d\widetilde{b})}{d\widetilde{b}}\bigg\vert_{\widetilde{b}=b}.
\end{align*}
Then, we have
\begin{equation}\label{suzue1}
E[F(X_{[0,1]}^{a \to b} |_{K_{[0, 1]}(g^-_{\varepsilon}, g^+_{\varepsilon})})] 
=\frac{I(\varepsilon, F)}{I(\varepsilon, 1)}.
\end{equation}
Further, by Lemma~\ref{Lem_MarkovDecomp_B_Bridge}, we obtain
\begin{align}
I(\varepsilon, F) 
&= 
\int_{g^-_{\varepsilon}(t)}^{g^+_{\varepsilon}(t)}  dy
\ E\left[F\Big(X_{[0, t]}^{a,y,(g^-_{\varepsilon}, g^+_{\varepsilon})} 
\oplus_{t} 
X_{[t, 1]}^{y, b,(g^-_{\varepsilon}, g^+_{\varepsilon})}\Big)\right] \nonumber \\
&\qquad \qquad \quad \times 
P(X^{a\to y}_{[0, t]} \in K_{[0, t]}(g^-_{\varepsilon}, g^+_{\varepsilon}))
\cdot \frac{P(X^a_{[0, t]}(t) \in dy)}{dy}
\nonumber \\
&\qquad \qquad \quad \times 
P(X^{y\to b}_{[t, 1]} \in K_{[t, 1]}(g^-_{\varepsilon}, g^+_{\varepsilon}))
\cdot \frac{P(X^y_{[t, 1]}(1) \in d\widetilde{b})}{d\widetilde{b}}\bigg\vert_{\widetilde{b}=b} 
\label{Moving_MarkovDecomp_single}
\end{align}
It follows from Lemma~\ref{Lem_for_H_gm_to_gp}, 
Lemma~\ref{Ap_Lem_weakconv_product} and Lemma~\ref{scheffe_pp_intg} that
\begin{align}
&I(F):=\lim_{\varepsilon \downarrow 0} 
\frac{I(\varepsilon, F)}{ ( g^-(0) - g^-_{\varepsilon}(0) ) ( g^+_{\varepsilon}(1) - g^+(1) ) }\nonumber \\
&= \lim_{\varepsilon \downarrow 0} 
\int_{g^-_{\varepsilon}(t)}^{g^+_{\varepsilon}(t)}  dy
\ E\left[F\Big(X_{[0, t]}^{a,y,(g^-_{\varepsilon}, g^+_{\varepsilon})} 
\oplus_{t} 
X_{[t, 1]}^{y,b,(g^-_{\varepsilon}, g^+_{\varepsilon})}\Big)\right] \nonumber \\
&\qquad \qquad \qquad \times 
\frac{P(X^{a\to y}_{[0, t]} \in K_{[0, t]}(g^-_{\varepsilon}, g^+_{\varepsilon}))}{ g^-(0) - g^-_{\varepsilon}(0) }
\cdot \frac{P(X^a_{[0, t]}(t) \in dy)}{dy}
\nonumber \\
&\qquad \qquad \qquad \times 
\frac{P(X^{y\to b}_{[t, 1]} \in K_{[t, 1]}(g^-_{\varepsilon}, g^+_{\varepsilon}))}{ g^+_{\varepsilon}(1) - g^+(1) }
\cdot \frac{P(X^y_{[t, 1]}(1) \in d\widetilde{b})}{d\widetilde{b}}\bigg\vert_{\widetilde{b}=b}\nonumber \\
&=%%\frac{2}{\pi} 
\int_{g^-(t)}^{g^+(t)}
E\left[F\left(X_{[0, t]}^{a,y,(g^-, g^+)} \oplus_{t} X_{[t, 1]}^{y, b,(g^-, g^+)}\right)\right] 
p^{(g^-,g^+), (\uparrow)}_{[0,t]}(y)p^{(g^-,g^+), (\downarrow)}_{[t,1]}(y)dy .
\label{IF_t_decomp}
\end{align}
According to \eqref{IF_t_decomp}, we get
\begin{align}
&\int_{g^-(t)}^{g^+(t)}
p^{(g^-,g^+), (\uparrow)}_{[0,t]}(y) p^{(g^-,g^+), (\downarrow)}_{[t,1]}(y)dy =I(1)
=C_{g^-,g^+} . 
\label{I1_t_decomp}
\end{align}
Using \eqref{suzue1}, \eqref{IF_t_decomp} and \eqref{I1_t_decomp}, we obtain
\begin{align*}
C_{g^-,g^+}=\int_{g^-(t)}^{g^+(t)}
p^{(g^-,g^+), (\uparrow)}_{[0,t]}(y)
p^{(g^-,g^+), (\downarrow)}_{[t,1]}(y)dy \in (0, \infty)
\end{align*}
and
\begin{align*}
\lim_{\varepsilon \downarrow 0} 
E[F(X_{[0, 1]}^{a \to b} |_{K_{[0, 1]}(g^-_{\varepsilon}, g^+_{\varepsilon})})]
= \frac{I(F)}{I(1)} 
= \int_{g^-(t)}^{g^+(t)} 
E\left[ F\left(X_{[0,t]}^{a,y,(g^-, g^+)} \oplus_{t} X_{[t,1]}^{y,b,(g^-, g^+)} \right) \right] 
h(t, y)dy .
\end{align*}
Therefore, we can define the probability measure $\widetilde{P}_H$
on $(C([0,1], \mathbb{R}), \mathcal{B}(C([0,1], \mathbb{R})))$ as
\begin{align*}
\widetilde{P}_H(A):= \int_{g^-(t)}^{g^+(t)} 
P\left( X_{[0,t]}^{a,y,(g^-, g^+)} \oplus_{t} X_{[t,1]}^{y,b,(g^-, g^+)} \in A \right) h(t,y) dy
\quad (A\in \mathcal{B}(C([0,1], \mathbb{R}))),
\end{align*}
and there exists an $\mathbb{R}$-valued 
continuous stochastic process $X^{a,b,(g^-, g^+)}=\{X^{a,b,(g^-, g^+)}(t)\}_{t\in [0,1]}$
that satisfies (\ref{curvedMoving_def_eq1}) and (\ref{6.3}). 
Thus, a limit argument on $F$ yields 
\begin{align*}
P(X^{a,b,(g^-, g^+)}(t) \in dy) = h(t, y)dy\quad (y \in (g^-(t), g^+(t))).
\end{align*}

On the other hand, by Lemma~\ref{Lem_MarkovDecomp_B_Bridge}, we obtain
\begin{align}
&I(\varepsilon, F) 
= \int_{g^-_{\varepsilon}(t_2)}^{g^+_{\varepsilon}(t_2)} dy_2
\int_{g^-_{\varepsilon}(t_1)}^{g^+_{\varepsilon}(t_1)} dy_1 
\nonumber \\
&\quad \times E\left[F\left(X_{[0, t_1]}^{a,y_1,(g^-_{\varepsilon}, g^+_{\varepsilon})} \oplus_{t_1} 
X_{[t_1,t_2]}^{y_1,y_2,(g^-_{\varepsilon}, g^+_{\varepsilon})} \oplus_{t_2} 
X_{[t_2,1]}^{y_2,b,(g^-_{\varepsilon}, g^+_{\varepsilon})}\right)\right] 
\nonumber \\
&\quad \times 
P(X^{a\to y_1}_{[0, t_1]} \in K_{[0, t_1]}(g^-_{\varepsilon}, g^+_{\varepsilon}))
\cdot \frac{P(X^a_{[0, t_1]}(t_1) \in dy_1)}{dy_1} \nonumber \\
&\quad \times 
P(X^{y_1\to y_2}_{[t_1,t_2]} \in K_{[t_1,t_2]}(g^-_{\varepsilon}, g^+_{\varepsilon}))
\cdot \frac{P(X^{y_1}_{[t_1,t_2]}(t_2) \in dy_2)}{dy_2} \nonumber \\
&\quad \times 
P(X^{y_2\to b}_{[t_2, 1]} \in K_{[t_2, 1]}(g^-_{\varepsilon}, g^+_{\varepsilon}))
\cdot \frac{P(X^{y_2}_{[t_2, 1]}(1) \in d\widetilde{b})}{d\widetilde{b}}\bigg\vert_{\widetilde{b}=b} 
\label{Moving_MarkovDecomp_double} 
\end{align}
By \eqref{Moving_MarkovDecomp_double},
Lemma~\ref{Lem_for_H_gm_to_gp},
Lemma~\ref{Ap_Lem_weakconv_product} and Lemma~\ref{scheffe_pp_intg}, 
$I(F)$ satisfies
\begin{align}
I(F)&=\lim_{\varepsilon \downarrow 0} 
\frac{I(\varepsilon, F)}
{ ( g^-(0) - g^-_{\varepsilon}(0) ) ( g^+_{\varepsilon}(1) - g^+(1) ) } \nonumber \\
&=
\int_{g^-(t_1)}^{g^+(t_1)} \int_{g^-(t_2)}^{g^+(t_2)} 
E\left[F\left(X_{[0,t_1]}^{a,y_1,(g^-, g^+)} \oplus_{t_1} 
X_{[t_1,t_2]}^{y_1,y_2,(g^-, g^+)} \oplus_{t_2} 
X_{[t_2,1]}^{y_2,b,(g^-, g^+)}\right)\right] \nonumber \\
&\qquad \qquad \qquad \times 
p^{(g^-,g^+), (\uparrow)}_{[0,t_1]}(y_1)p^{(g^-,g^+)}_{[t_1,t_2]}(y_1, y_2) p^{(g^-,g^+), (\downarrow)}_{[t_2,1]}(y_2) dy_1 dy_2 
\nonumber \\
&=C_{g^-,g^+}
\int_{g^-(t_1)}^{g^+(t_1)} \int_{g^-(t_2)}^{g^+(t_2)} 
E\left[F\left(X_{[0,t_1]}^{a,y_1,(g^-, g^+)} \oplus_{t_1} 
X_{[t_1,t_2]}^{y_1,y_2,(g^-, g^+)} \oplus_{t_2} 
X_{[t_2, 1]}^{y_2,b,(g^-, g^+)}\right) \right] \label{IF_doubleIntg_h_t1t2}\\
&\qquad \qquad \qquad \qquad \quad \times 
h(t_1, y_1) h(t_1, y_1, t_2, y_2) dy_1 dy_2 .
\nonumber 
\end{align}
It follows from \eqref{suzue1}, \eqref{I1_t_decomp} and \eqref{IF_doubleIntg_h_t1t2} that
\begin{align*}
&E[F(X^{a,b,(g^-, g^+)})] 
=\lim_{\varepsilon \downarrow 0} 
E[F(X_{[0,1]}^{a \to b} |_{K_{[0, 1]}(g^-_{\varepsilon}, g^+_{\varepsilon})})]\\
&\quad =\int_{g^-(t_1)}^{g^+(t_1)} \int_{g^-(t_2)}^{g^+(t_2)} 
E\left[F\left(X_{[0,t_1]}^{a,y_1,(g^-, g^+)} \oplus_{t_1} 
X_{[t_1,t_2]}^{y_1, y_2,(g^-, g^+)}\oplus_{t_2} 
X_{[t_2, 1]}^{y_2, b,(g^-, g^+)}\right)\right] \\
&\qquad \qquad \qquad \qquad \times h(t_1, y_1)h(t_1, y_1, t_2, y_2) dy_1 dy_2 .
\end{align*}
Hence, (\ref{6.4}) holds. 
Similarly, using a limit argument on $F$, we can deduce for 
$y_1 \in (g^-(t_1), g^+(t_1))$ and $y_2 \in (g^-(t_2), g^+(t_2))$ that
\begin{align}
P(X^{a,b,(g^-, g^+)}(t_1) \in dy_1, X^{a,b,(g^-, g^+)}(t_2) \in dy_2) &= h(t_1, y_1)h(t_1, y_1, t_2, y_2) dy_1 dy_2, 
\nonumber \\
P(X^{a,b,(g^-, g^+)}(t_2) \in dy_2 \ \vert \ X^{a,b,(g^-, g^+)}(t_1)=y_1) &=h(t_1, y_1, t_2, y_2) dy_2.
\label{Proof_H_transdensity}
\end{align}

If we define $I_{t_1}(\varepsilon , y_1)$ $(y_1 \in (g^-(t_1), g^+(t_1)))$ to be
\begin{align*}
&I_{t_1}(\varepsilon ,y_1):= 
P\left(X^{y_1 \to b}_{[t_1, 1]} \in K_{[t_1, 1]}(g^-_{\varepsilon}, g^+_{\varepsilon})\right)
\cdot \frac{P\big(X^{y_1}_{[t_1,1]}(1) \in d\widetilde{b}\big)}{d\widetilde{b}}\bigg\vert_{\widetilde{b}=b},
\end{align*}
then 
\begin{align}
I_{t_1}(y_1):=&\lim_{\varepsilon \downarrow 0} \frac{I_{t_1}(\varepsilon ,y_1)}{ g^+_{\varepsilon}(1) - g^+(1) }
=p^{(g^-,g^+), (\downarrow)}_{[t_1,1]}(y_1)
\qquad (y_1 \in (g^-(t_1), g^+(t_1)))
\label{It1y1_expression_1}
\end{align}
holds by Lemma~\ref{Lem_pX_conv_2}. 
On the other hand, because we have
\begin{align*}
I_{t_1}(\varepsilon , y_1) 
&= \int_{g^-_{\varepsilon}(t_2)}^{g^+_{\varepsilon}(t_2)} dy_2 
\ P\left(X^{y_1 \to y_2}_{[t_1,t_2]} \in K_{[t_1,t_2]}(g^-_{\varepsilon}, g^+_{\varepsilon})\right) 
\cdot \frac{P\left(X^{y_1}_{[t_1,t_2]}(t_2)\in dy_2\right)}{dy_2}\\
&\qquad  \qquad \qquad \times 
P\left(X^{y_2\to b}_{[t_2, 1]} \in K_{[t_2, 1]}(g^-_{\varepsilon}, g^+_{\varepsilon})\right)
\cdot \frac{P\left(X^{y_2}_{[t_2,1]}(1)\in d\widetilde{b}\right)}{d\widetilde{b}}\bigg\vert_{\widetilde{b}=b} 
\end{align*}
for $y_1 \in (g^-(t_1), g^+(t_1))$.
It follows from Lemma~\ref{Lem_pX_conv_2} that
\begin{align}
I_{t_1}(y_1)
&=\lim_{\varepsilon \downarrow 0} \frac{I_{t_1}(\varepsilon ,y_1)}{ g^+_{\varepsilon}(1) - g^+(1) }
\nonumber \\
&=\int_{g^-(t_2)}^{g^+(t_2)} 
p^{(g^-,g^+)}_{[t_1,t_2]} (y_1, y_2) p^{(g^-,g^+), (\downarrow)}_{[t_2,1]}(y_2) dy_2 
\qquad (y_1 \in (g^-(t_1), g^+(t_1))).
\label{It1y1_expression_2}
\end{align} 
Using \eqref{It1y1_expression_1} and \eqref{It1y1_expression_2}, we obtain
\begin{align*}
p^{(g^-,g^+), (\downarrow)}_{[t_1,1]}(y_1)
&=\int_{g^-(t_2)}^{g^+(t_2)} 
p^{(g^-,g^+)}_{[t_1,t_2]} (y_1, y_2) p^{(g^-,g^+), (\downarrow)}_{[t_2,1]}(y_2) dy_2 
\qquad (y_1 \in (g^-(t_1), g^+(t_1)))
\end{align*}
and
\begin{align}
\int_{g^-(t_2)}^{g^+(t_2)} h(t_1, y_1, t_2, y_2) dy_2 = 1
\qquad (y_1 \in (g^-(t_1), g^+(t_1))).
\label{ht1y1t2y2_density}
\end{align}
Assume that $t_3$ satisfies $0<t_1<t_2<t_3<1$. 
Because we have
\begin{align}\label{Chapman_id_p_y1_y3}
p^{(g^-,g^+)}_{[t_1,t_3]} (y_1, y_3)=
\int_{g^-(t_2)}^{g^+(t_2)} 
p^{(g^-,g^+)}_{[t_1,t_2]} (y_1, y_2)
p^{(g^-,g^+)}_{[t_2,t_3]} (y_2, y_3) dy_2
\end{align}
for $y_1 \in (g^-(t_1), g^+(t_1))$ and $y_3 \in (g^-(t_3), g^+(t_3))$, 
we can deduce for 
$y_1 \in (g^-(t_1), g^+(t_1))$ and $y_3 \in (g^-(t_3), g^+(t_3))$ that 
the following Chapman--Kolmogorov identity holds: 
\begin{align}\label{ht1y1t2y2_transdensity}
h(t_1, y_1, t_3, y_3) 
=\int_{g^-(t_2)}^{g^+(t_2)} h(t_1, y_1, t_2, y_2) h(t_2, y_2, t_3, y_3) dy_2 .
\end{align}
Therefore, \eqref{Proof_H_transdensity}, \eqref{ht1y1t2y2_density} and \eqref{ht1y1t2y2_transdensity} imply that 
$X^{a,b,(g^-, g^+)}=\{X^{a,b,(g^-, g^+)}(t)\}_{t\in [0,1]}$ is a Markov process.

%%%%%%%%%%%%%%%%%%%%%%%%%%%%%%%%%%%%%%
\subsection{Proof of Corollary~\ref{Cor_30_curve_max_0t}}
%%%%%%%%%%%%%%%%%%%%%%%%%%%%%%%%%%%%%%

Let $A_i$ $(i=1,2)$ be closed subsets of $C([0,1], \mathbb{R})$ given by
\begin{align*}
&A_1:=\left\{ w \in C([0,1], \mathbb{R})\ \big\vert \ \min_{u\in [0,t]}\left\{ g(u)-w(u) \right\}= 0 \right\}, \\
&A_2:=\left\{ w \in C([0,1], \mathbb{R})
\ \big\vert \ \min_{u\in [0,t]}\left\{ g(u)-w(u) \right\} \geq 0,\ w(t) \leq z \right\}.
\end{align*}
Remark~\ref{Rem_F_canbetakento_Indicator_Moving} implies that 
Theorem~\ref{Thm_Def_and_Decomp_curved_Moving} 
can be applied for $F=1_{A_i}$ $(i=1,2)$. Thus, we obtain
\begin{align}
&P\left( \min_{u\in [0,t]} \left\{ g(u)-X^{a,b,(g^-, g^+)}(u) \right\} = 0 \right)
=\int_{g^-(t)}^{g(t)} P\left( X_{[0,t]}^{a,y,(g^-, g^+)}\in \partial K_{[0,t]}^-(g) \right) h(t,y)dy,
\label{Moving_Decomp_touch_g} \\
&P\left( \min_{u\in [0,t]}\left\{ g(u)-X^{a,b,(g^-, g^+)}(u) \right\} \geq 0, X^{a,b,(g^-, g^+)}(t) \leq z \right) 
\nonumber \\
&\quad =\int_{g^-(t)}^z
P\left( X_{[0,t]}^{a,y,(g^-, g^+)}\in K_{[0,t]}^-(g)
\right) h(t,y)dy .
\label{Moving_Decomp_leq_g}
\end{align}
It follows from $(55)$ in \cite{bib_Ishitani_moving} that
\begin{align}\label{eq_E_1_partialKg_eq_0}
P\left( B_{[0,t]}^{a,y,(g^-, g^+)}\in \partial K_{[0,t]}^-(g) \right)=0 \qquad (g^-(t)<y<g(t)).
\end{align}
Using Remark~\ref{Rem_F_canbetakento_Indicator_MeanderBESbridge}, Lemma~\ref{Lem_for_H_gm_to_gp} and \eqref{eq_E_1_partialKg_eq_0}, 
we have 
\begin{align}\label{Cor2_eq_prob_zero}
P\left( X_{[0,t]}^{a,y,(g^-, g^+)}\in \partial K_{[0,t]}^-(g) \right)=0
\qquad (g^-(t)<y<g(t)).
\end{align} 
Thus, by \eqref{Moving_Decomp_touch_g} and \eqref{Cor2_eq_prob_zero}, 
we obtain \eqref{eq_curve_max_0t_bdry}.
On the other hand, we have
\begin{align}
&P\left( X_{[0,t]}^{a,y,(g^-, g^+)}\in K_{[0,t]}^-(g) \right)
\nonumber \\
&\quad = \lim_{\varepsilon \downarrow 0}
\dfrac{P\left( X_{[0,t]}^{a \to y} \in K_{[0,t]}(g^-_{\varepsilon}, g) \right)}
{P\left(X_{[0,t]}^{a \to y} \in K_{[0,t]}(g^-_{\varepsilon}, g^+_{\varepsilon}) \right)}
= \frac{p^{(g^-,g), (\uparrow)}_{[0,t]}(y)}{p^{(g^-,g^+), (\uparrow)}_{[0,t]}(y)}
\qquad (g^-(t)<y<g(t))
\label{ProbChange_curve_max_0t_g}
\end{align}
by Remark~\ref{Rem_F_canbetakento_Indicator_MeanderBESbridge}. 
Using \eqref{Moving_Decomp_leq_g} and \eqref{ProbChange_curve_max_0t_g}, 
we obtain \eqref{eq_curve_max_0t_jointdist}.

%%%%%%%%%%%%%%%%%%%%%%%%%%%%%%%%%%
\subsection{Proof of Corollary~\ref{Cor_30_curve_min_t1}}
%%%%%%%%%%%%%%%%%%%%%%%%%%%%%%%%%%

The proof of Corollary~\ref{Cor_30_curve_min_t1} is quite similar to the proof of Corollary~\ref{Cor_30_curve_max_0t}. 
Therefore, below we outline the proof of Corollary~\ref{Cor_30_curve_min_t1}. 
%%%Let $b = g^+(1)$. 
Then, we obtain
\begin{align}
&P\left( \min_{u\in [t,1]} \left\{ X^{a,b,(g^-, g^+)}(u)-g(u) \right\} = 0 \right)
=\int_{g(t)}^{g^+(t)} P\left( X_{[t,1]}^{y,b,(g^-, g^+)}\in \partial K_{[t,1]}^+(g) \right) h(t,y)dy,
\label{Moving_Decomp_geq_g_touch_g} \\
&P\left( \min_{u\in [t,1]}\left\{ X^{a,b,(g^-, g^+)}(u) - g(u) \right\} \geq 0, X^{a,b,(g^-, g^+)}(t) \leq z \right) 
\nonumber \\
&\quad =\int_{g(t)}^z
P\left( X_{[t,1]}^{y,b,(g^-, g^+)} \in K_{[t,1]}^+(g)
\right) h(t,y)dy .
\label{Moving_Decomp_geq_g}
\end{align}
On the other hand, we have 
\begin{align}
&P\left( X_{[t,1]}^{y,b,(g^-, g^+)}\in \partial K_{[t,1]}^+(g) \right)=0, \label{Cor3_eq_prob_zero}\\
&P\left( X_{[t,1]}^{y, b,(g^-, g^+)} \in K_{[t,1]}^+(g) \right) 
= \lim_{\varepsilon \downarrow 0}
\dfrac{P\left( X_{[t,1]}^{y \to b} \in K_{[t,1]}(g, g^+_{\varepsilon}) \right)}
{P\left(X_{[t,1]}^{y \to b} \in K_{[t,1]}(g^-_{\varepsilon}, g^+_{\varepsilon}) \right)}
= \frac{p^{(g,g^+), (\downarrow)}_{[t,1]}(y)}{p^{(g^-,g^+), (\downarrow)}_{[t,1]}(y)}
\label{Change_curve_gmin_to_g}
\end{align}
for $y \in (g(t), g^+(t))$. 
Using \eqref{Moving_Decomp_geq_g_touch_g}, \eqref{Moving_Decomp_geq_g}, 
\eqref{Cor3_eq_prob_zero} and \eqref{Change_curve_gmin_to_g}, we obtain 
\eqref{eq_curve_min_t1_bdry} and \eqref{eq_curve_min_t1_jointdist}.

%%%%%%%%%%%%%%%%%%%%%%%%%%%%%%%%%%%%%%%%%%%%%%%%%%%%%%
\section{Proofs of the main results in Subsection~\ref{subsection_Meander_between_2curves}}
\label{section_proof_Meander_between_2curves}
%%%%%%%%%%%%%%%%%%%%%%%%%%%%%%%%%%%%%%%%%%%%%%%%%%%%%%

In this section, we prove the main results in Subsection~\ref{subsection_Meander_between_2curves}.

%%%%%%%%%%%%%%%%%%%%%%%%%%%%%%%%%%%%%%%%%%%%%%%%%%%%%
\subsection{Proofs of Theorem~\ref{Thm_Def_and_Decomp_curved_Meander_two_curve} }
\label{subsection_proof_Meander_between_2curves}
%%%%%%%%%%%%%%%%%%%%%%%%%%%%%%%%%%%%%%%%%%%%%%%%%%%%%

In this subsection, we assume that all $X_{[s, t]}^{x,y,(h^-, h^+)}$ and $X_{[s, t]}^{z,(h^-, h^+)}$ are independent. 
For each $\mathbb{R}$-valued bounded continuous function $G$ on $C([0,1], \mathbb{R})$ and $\varepsilon>0$, we define
\begin{align*}
I(\varepsilon, G):=E[G(X^a_{[0,1]})\ ;\ X^a_{[0,1]} \in K_{[0, 1]}(g^-_{\varepsilon}, g^+_{\varepsilon})].
\end{align*}
Then, we have
\begin{equation}\label{ishitani_meander_2curves}
E[F(X^a_{[0,1]} |_{K_{[0, 1]}(g^-_{\varepsilon}, g^+_{\varepsilon})})] 
= \frac{E[F(X^a_{[0,1]})\ ;\ X^a_{[0,1]} \in K_{[0, 1]}(g^-_{\varepsilon}, g^+_{\varepsilon})]}
{P(X^a_{[0,1]} \in K_{[0, 1]}(g^-_{\varepsilon}, g^+_{\varepsilon}))}
=\frac{I(\varepsilon, F)}{I(\varepsilon, 1)}.
\end{equation}
Further, by Lemma~\ref{Lem_MarkovDecomp_BM}, we obtain
\begin{align}
I(\varepsilon, F) 
&= \int_{g^-_{\varepsilon}(t)}^{g^+_{\varepsilon}(t)}  dy 
\ E[F(X_{[0, t]}^{a,y,(g^-_{\varepsilon}, g^+_{\varepsilon})} \oplus_{t} X_{[t, 1]}^{y,(g^-_{\varepsilon}, g^+_{\varepsilon})})] 
\nonumber \\
&\qquad \qquad \times 
P\left(X^{a\to y}_{[0, t]} \in K_{[0, t]}(g^-_{\varepsilon}, g^+_{\varepsilon})\right)
\cdot \frac{P(X^a_{[0, t]}(t) \in dy)}{dy}\nonumber \\
&\qquad \qquad \times P(X^y_{[t, 1]} \in K_{[t, 1]}(g^-_{\varepsilon}, g^+_{\varepsilon})).
\label{Meander_MarkovDecomp_single}
\end{align}
It follows from \eqref{Meander_MarkovDecomp_single}, 
Lemma~\ref{Lem_for_H_gm_to_gp}, Lemma~\ref{Lem_WeakConv_Meander_btwn_2crvs},
Lemma~\ref{Ap_Lem_weakconv_product} and Lemma~\ref{Lem_pX_conv_1} that
\begin{align}
I(F):=&\lim_{\varepsilon \downarrow 0} 
\frac{I(\varepsilon, F)}{ g^-(0) - g^-_{\varepsilon}(0) }
\nonumber \\
=& \int_{g^-(t)}^{g^+(t)}
E\left[F\left(X_{[0, t]}^{a,y,(g^-, g^+)} \oplus_{t} X_{[t, 1]}^{y, (g^-, g^+)}\right)\right] 
p^{(g^-,g^+), (\uparrow)}_{[0,t]}(y) 
p^{(g^-,g^+)}_{[t,1]}(y) dy.
\label{IF_t_Meander_decomp}
\end{align}
According to \eqref{IF_t_Meander_decomp}, we get
\begin{align}\label{I1_t_Meander_decomp}
\int_{g^-(t)}^{g^+(t)}
p^{(g^-,g^+), (\uparrow)}_{[0,t]}(y)p^{(g^-,g^+)}_{[t,1]}(y)dy 
=I(1)=\lim_{\varepsilon \downarrow 0} 
\frac{P(X^a \in K_{[0, 1]}(g^-_{\varepsilon}, g^+_{\varepsilon}))}{ g^-(0) - g^-_{\varepsilon}(0) }
=\widetilde{C}_{g^-,g^+} . 
\end{align}
Using \eqref{ishitani_meander_2curves}, \eqref{IF_t_Meander_decomp} and \eqref{I1_t_Meander_decomp}, we obtain
\begin{align*}
\widetilde{C}_{g^-,g^+}=\int_{g^-(t)}^{g^+(t)}
p^{(g^-,g^+), (\uparrow)}_{[0,t]}(y)p^{(g^-,g^+)}_{[t,1]}(y) dy \in (0, \infty)
\end{align*}
and
\begin{align*}
\lim_{\varepsilon \downarrow 0} 
E[F(X^a_{[0, 1]} |_{K_{[0, 1]}(g^-_{\varepsilon}, g^+)})]
=\frac{I(F)}{I(1)}
=\int_{g^-(t)}^{g^+(t)} 
E\left[ F(X_{[0,t]}^{a,y,(g^-, g^+)} \oplus_{t} X_{[t,1]}^{y,(g^-, g^+)}) \right]  k(t,y) dy.
\end{align*}
Therefore, we can define the probability measure $\widetilde{P}_+$
on $(C([0,1], \mathbb{R}), \mathcal{B}(C([0,1], \mathbb{R})))$ as
\begin{align*}
\widetilde{P}_+(A):= \int_{g^-(t)}^{g^+(t)} 
P\left( X_{[0,t]}^{a,y,(g^-, g^+)} \oplus_{t} X_{[t,1]}^{y,(g^-, g^+)} \in A \right) k(t,y) dy
\quad (A\in \mathcal{B}(C([0,1], \mathbb{R}))),
\end{align*}
and there exists an $\mathbb{R}$-valued 
continuous stochastic process $X^{+,(g^-,g^+)}=\{X^{+,(g^-,g^+)}(t)\}_{t\in [0,1]}$
that satisfies \eqref{curvedMeander_twocurve_def_eq1} 
and \eqref{twocurve_meander_decomp_1}.
Thus, a limit argument on $F$ yields 
\begin{align*}
P(X^{+,(g^-,g^+)}(t) \in dy) = k(t, y)dy\quad (y \in (g^-(t), g^+(t))).
\end{align*}
On the other hand, by Lemma~\ref{Lem_MarkovDecomp_BM}, we obtain
\begin{align}
I(\varepsilon, F) 
&= \int_{g^-_{\varepsilon}(t_1)}^{g^+_{\varepsilon}(t_1)} dy_1 \int_{g^-_{\varepsilon}(t_2)}^{g^+_{\varepsilon}(t_2)} dy_2  
\ E[F(X_{[0, t_1]}^{a,y_1,(g^-_{\varepsilon}, g^+_{\varepsilon})} \oplus_{t_1} 
X_{[t_1,t_2]}^{y_1, y_2,(g^-_{\varepsilon}, g^+_{\varepsilon})} \oplus_{t_2} 
X_{[t_2, 1]}^{y_2,(g^-_{\varepsilon}, g^+_{\varepsilon})})] \nonumber \\
&\qquad \qquad \qquad \qquad \quad \times 
P\left(X^{a\to y_1}_{[0, t_1]} \in K_{[0, t_1]}(g^-_{\varepsilon}, g^+_{\varepsilon})\right)
\cdot \frac{P(X^a_{[0, t_1]}(t_1) \in dy_1)}{dy_1}\nonumber \\
&\qquad \qquad \qquad \qquad \quad \times 
P\left(X^{y_1\to y_2}_{[t_1,t_2]} \in K_{[t_1,t_2]}(g^-_{\varepsilon}, g^+_{\varepsilon})\right)
\cdot \frac{P(X^{y_1}_{[t_1,t_2]}(t_2) \in dy_2)}{dy_2} \nonumber \\
&\qquad \qquad \qquad \qquad \quad \times 
P(X^{y_2}_{[t_2, 1]} \in K_{[t_2, 1]}(g^-_{\varepsilon}, g^+_{\varepsilon})). \label{Meander_MarkovDecomp_double}
\end{align}
By \eqref{Meander_MarkovDecomp_double},
Lemma~\ref{Lem_for_H_gm_to_gp}, Lemma~\ref{Lem_WeakConv_Meander_btwn_2crvs}, 
Lemma~\ref{Ap_Lem_weakconv_product} and Lemma~\ref{Lem_pX_conv_1}, 
$I(F)$ satisfies
\begin{align}
I(F)=&\lim_{\varepsilon \downarrow 0} 
\frac{I(\varepsilon, F)}{ g^-(0) - g^-_{\varepsilon}(0) } \nonumber \\
=&
\int_{g^-(t_1)}^{g^+(t_1)} \int_{g^-(t_2)}^{g^+(t_2)} 
E\left[F\left(X_{[0,t_1]}^{a,y_1,(g^-, g^+)} \oplus_{t_1} 
X_{[t_1,t_2]}^{y_1, y_2,(g^-, g^+)} \oplus_{t_2} 
X_{[t_2, 1]}^{y_2, (g^-, g^+)}\right)\right] \nonumber \\
&\qquad \qquad \qquad \times 
p^{(g^-,g^+), (\uparrow)}_{[0,t_1]}(y_1)p^{(g^-,g^+)}_{[t_1,t_2]}(y_1, y_2) p^{(g^-,g^+)}_{[t_2,1]}(y_2) dy_1 dy_2 \nonumber \\
=&\widetilde{C}_{g^-,g^+}
\int_{g^-(t_1)}^{g^+(t_1)} \int_{g^-(t_2)}^{g^+(t_2)} 
E\left[F\left(X_{[0,t_1]}^{a,y_1,(g^-, g^+)} \oplus_{t_1} 
X_{[t_1,t_2]}^{y_1, y_2,(g^-, g^+)} \oplus_{t_2} 
X_{[t_2, 1]}^{y_2, (g^-, g^+)}\right)\right] \label{IF_doubleIntg_kt1t2}\\
&\qquad \qquad \qquad \qquad \times 
k(t_1, y_1) k(t_1, y_1, t_2, y_2) dy_1 dy_2 . \nonumber
\end{align}
It follows from \eqref{ishitani_meander_2curves}, \eqref{I1_t_Meander_decomp} and \eqref{IF_doubleIntg_kt1t2} that
\begin{align*}
&E[F(X^{+,(g^-,g^+)})]  
=\lim_{\varepsilon \downarrow 0} 
E[F(X^a_{[0, 1]} |_{K_{[0, 1]}(g^-_{\varepsilon}, g^+_{\varepsilon})})]\\
&\quad =\int_{g^-(t_1)}^{g^+(t_1)} \int_{g^-(t_2)}^{g^+(t_2)} 
E\left[F\left(X_{[0,t_1]}^{a,y_1,(g^-, g^+)} \oplus_{t_1} 
X_{[t_1,t_2]}^{y_1, y_2,(g^-, g^+)}\oplus_{t_2} 
X_{[t_2, 1]}^{y_2, (g^-, g^+)}\right)\right] \nonumber \\
&\qquad \qquad \qquad \qquad \times k(t_1, y_1)k(t_1, y_1, t_2, y_2) dy_1 dy_2 .
\end{align*}
Hence, \eqref{twocurve_meander_decomp_2} holds. 
Similarly, using a limit argument on $F$, we can deduce for 
$y_1 \in (g^-(t_1), g^+(t_1))$ and $y_2 \in (g^-(t_2), g^+(t_2))$ that
\begin{align}
P(X^{+,(g^-,g^+)}(t_1) \in dy_1, X^{+,(g^-,g^+)}(t_2) \in dy_2) 
&= k(t_1, y_1)k(t_1, y_1, t_2, y_2) dy_1 dy_2, \nonumber \\
P(X^{+,(g^-,g^+)}(t_2) \in dy_2 \ \vert \ X^{+,(g^-,g^+)}(t_1)=y_1) 
&=k(t_1, y_1, t_2, y_2) dy_2.
\label{Proof_Wplus_transdensity}
\end{align}
Because we have
\begin{align*}
p^{(g^-,g^+)}_{[t_1,1]}(y_1) &= 
\int_{g^-(t_2)}^{g^+(t_2)} 
p^{(g^-,g^+)}_{[t_1,t_2]} (y_1, y_2) p^{(g^-,g^+)}_{[t_2,1]}(y_2) dy_2 
\qquad (y_1 \in (g^-(t_1), g^+(t_1))),
\end{align*}
we can deduce that
\begin{align}\label{kt1y1t2y2_density}
\int_{g^-(t_2)}^{g^+(t_2)} k(t_1, y_1, t_2, y_2) dy_2 = 1
\qquad (y_1 \in (g^-(t_1), g^+(t_1))).
\end{align}
The Chapman--Kolmogorov identity for $k$ is obtained by \eqref{Chapman_id_p_y1_y3}.
Therefore, \eqref{Proof_Wplus_transdensity} and \eqref{kt1y1t2y2_density} imply that 
$X^{+,(g^-,g^+)}=\{X^{+,(g^-,g^+)}(t)\}_{t\in [0,1]}$ is a Markov process.

%%%%%%%%%%%%%%%%%%%%%%%%%%%%%%%%%%%%%%%%%%%%%%%%%%
\subsection{Proof of Theorem~\ref{Thm_abs_conti}}
%%%%%%%%%%%%%%%%%%%%%%%%%%%%%%%%%%%%%%%%%%%%%%%%%%

Let $F$ be an $\mathbb{R}$-valued bounded continuous function on $C([0, t], \mathbb{R})$. 
By the Markov property of $X_{[0,1]}^{a \to b}$ and Lemma~\ref{Lem_abs_conti_B_motion_and_B_bridge}, we obtain
\begin{align}
&E[F(\pi_{[0,t]} ( X_{[0,1]}^{a \to b} )) 1_{K_{[0, 1]}(g^-_{\varepsilon}, g^+_{\varepsilon})}(X_{[0,1]}^{a \to b} )]
\nonumber \\
&\quad =\int_{C([0,1],\mathbb{R})}F(\pi_{[0,t]} (w)) 
1_{K_{[0, t]}(g^-_{\varepsilon}, g^+_{\varepsilon})}(\pi_{[0,t]} (w))
\nonumber \\
&\quad \qquad \qquad \times P\big( X_{[t,1]}^{w(t) \to b}\in K_{[t, 1]}(g^-_{\varepsilon}, g^+_{\varepsilon})\big)
P(X_{[0,1]}^{a \to b} \in dw)\nonumber \\
&\quad =\int_{C([0,t],\mathbb{R})}F(w) 
1_{K_{[0, t]}(g^-_{\varepsilon}, g^+_{\varepsilon})}(w)\nonumber \\
&\quad \qquad \qquad \times P\big( X_{[t,1]}^{w(t) \to b}\in K_{[t, 1]}(g^-_{\varepsilon}, g^+_{\varepsilon})\big)
P(\pi_{[0,t]}(X_{[0,1]}^{a \to b})\in dw) \nonumber \\ 
&\quad =\int_{C([0,t],\mathbb{R})}F(w) 
P\big( X_{[t,1]}^{w(t) \to b}\in K_{[t, 1]}(g^-_{\varepsilon}, g^+_{\varepsilon})\big)
\label{bunshi_Markov_B0tob} \\
&\quad \qquad \qquad \times \frac{q_X(1-t,w(t),b)}{q_X(1,a,b)}
1_{K_{[0,t]}(g^-_{\varepsilon}, g^+_{\varepsilon})}(w) P(\pi_{[0,t]}(X_{[0,1]}^a)\in dw).
\nonumber
\end{align}
Thus, it holds from \eqref{bunshi_Markov_B0tob} that
\begin{align}
&E[F(\pi_{[0,t]} ( X_{[0,1]}^{a \to b} |_{K_{[0, 1]}(g^-_{\varepsilon}, g^+_{\varepsilon})}))]\nonumber \\
&\quad =\frac{E[F(\pi_{[0,t]} ( X_{[0,1]}^{a \to b} ))\ ;\ X_{[0,1]}^{a \to b} \in K_{[0, 1]}(g^-_{\varepsilon}, g^+_{\varepsilon})]}
{P(X_{[0,1]}^{a \to b} \in K_{[0, 1]}(g^-_{\varepsilon}, g^+_{\varepsilon}))} \nonumber \\
&\quad =
\frac{ ( g^-(0) - g^-_{\varepsilon}(0) ) ( g^+_{\varepsilon}(1) - g^+(1) )  }
{P\big(X_{[0,1]}^{a \to b} \in K_{[0, 1]}(g^-_{\varepsilon}, g^+_{\varepsilon}) \big)}
\cdot \frac{1}{q_X(1, a, b)}
\cdot
\frac{P\big( X^a_{[0,t]} \in K_{[0,t]}(g^-_{\varepsilon}, g^+_{\varepsilon}) \big)}{ g^-(0) - g^-_{\varepsilon}(0) } \nonumber \\
&\qquad \times \int_{C([0,t],\mathbb{R})}F(w) 
\frac{P\big( X_{[t,1]}^{w(t) \to b}\in K_{[t, 1]}(g^-_{\varepsilon}, g^+_{\varepsilon})\big)}{ g^+_{\varepsilon}(1) - g^+(1) }
\cdot q_X(1-t, w(t), b) 
\label{bunshi_bunbo_Markov_B0tob} \\
&\qquad \qquad \qquad \times 
P\big(X^a_{[0,t]}\vert_{K_{[0,t]}(g^-_{\varepsilon}, g^+_{\varepsilon})} \in dw \big). 
\nonumber
\end{align}
Using Lemma~\ref{Lem_pX_conv_2}, we obtain
\begin{align}
&\lim_{\varepsilon \downarrow 0}
\frac{P\big( X_{[t,1]}^{c \to b}\in K_{[t, 1]}(g^-_{\varepsilon}, g^+_{\varepsilon})\big)}{ g^+_{\varepsilon}(1) - g^+(1) } 
\cdot q_X(1-t, c, b)
= p^{(g^-,g^+), (\downarrow)}_{[t,1]}(c)
\qquad (g^-(t)<c<g^+(t)).
\label{Lem_for_radon_nikodym_conv1}
\end{align}
On the other hand, using Lemma~\ref{girsanov2} and Lemma~A.10 in \cite{bib_ishitani_hatakenaka_suzuki}, we have
\begin{align}
&P\big( X^a_{[0,t]} \in K_{[0,t]}(g^-_{\varepsilon}, g^+_{\varepsilon}) \big) 
\nonumber \\
&\quad = P\big( Y^{\Phi(a)}_{[0,t]} \in K_{[0,t]}(\phi(g^-_{\varepsilon}), \phi(g^+_{\varepsilon})) \big) 
\nonumber \\
&\quad = e^{-\Psi(\Phi(a))} 
E\Big[ e^{\Psi( W^{\Phi(a)}_{[0,t]}(t) )} e^{-\frac{1}{2}N_{[0,t]}(W^{\Phi(a)}_{[0,t]})} ; W^{\Phi(a)}_{[0,t]} \in K_{[0,t]}(\phi(g^-_{\varepsilon}), \phi(g^+_{\varepsilon})) \Big] 
\nonumber \\
&\quad = e^{-\Psi(\Phi(a))} 
E\Big[ e^{\Psi( W^{\Phi(a), (\phi(g^-_{\varepsilon}) , \phi(g^+_{\varepsilon}))}_{[0,t]}(t) )} e^{-\frac{1}{2}N_{[0,t]}(W^{\Phi(a), (\phi(g^-_{\varepsilon}) , \phi(g^+_{\varepsilon}))}_{[0,t]})} \Big] 
\nonumber \\
&\qquad \times P\big( W^{\Phi(a)}_{[0, t]} \in K_{[0,t]}(\phi(g^-_{\varepsilon}), \phi(g^+_{\varepsilon}))\big) 
\label{Prob_Xa_in_Kgpepgmep}
\end{align}
and
\begin{align}
&\frac{P\big( W_{[0,t]}^{\Phi(a)} \in K_{[0,t]}(\phi(g^-_{\varepsilon}), \phi(g^+_{\varepsilon})) \big)}
{ g^-(0) - g^-_{\varepsilon}(0) }
\nonumber \\
&\quad =
\frac{P\big( W_{[0,t]} \in K_{[0,t]}(\phi(g^-_{\varepsilon})-\Phi(g^-(0)), \phi(g^+_{\varepsilon})-\Phi(g^-(0))) \big)}
{ g^-(0) - g^-_{\varepsilon}(0) }
\nonumber \\
&\quad = 
E\Big[ \widetilde{Z}_{[0,t]}^{\phi(g^-_{\varepsilon})} 
\big(W_{[0,t]}\vert_{K_{[0,t]}(\Phi(g^-_{\varepsilon}(0))-\Phi(g^-(0)), \phi(g^+_{\varepsilon})-\phi(g^-_{\varepsilon}) +\Phi(g^-_{\varepsilon}(0))-\Phi(g^-(0)))}\big)^{-1} \Big]
\nonumber \\
&\qquad \times P\big( W_{[0,t]}\vert_{K_{[0,t]}^+(\Phi(g^-_{\varepsilon}(0))-\Phi(g^-(0)))} 
\in K_{[0,t]}^-\big(\phi(g^+_{\varepsilon})-\phi(g^-_{\varepsilon}) +\Phi(g^-_{\varepsilon}(0))-\Phi(g^-(0))\big)\big)
\nonumber \\
&\qquad \times 
\frac{P\big( W_{[0,t]} \in K_{[0,t]}^+\big(\Phi(g^-_{\varepsilon}(0))-\Phi(g^-(0))\big)\big)}{\Phi(g^-_{\varepsilon}(0))-\Phi(g^-(0))}
\cdot \frac{\Phi(g^-_{\varepsilon}(0))-\Phi(g^-(0))}{ g^-(0) - g^-_{\varepsilon}(0) }.
\label{frac_Prob_WPhia_in_Kphigpepphigmep}
\end{align}
Then, it holds from 
\eqref{Prob_Xa_in_Kgpepgmep}, \eqref{frac_Prob_WPhia_in_Kphigpepphigmep} 
and Proposition~\ref{Prop_Boundary_purturb_WeakConv_for_Meander_BES3bridge} that
\begin{align}
D(t):=&\lim_{\varepsilon \downarrow 0}\frac{P\big( X^a_{[0,t]} \in K_{[0,t]}(g^-_{\varepsilon}, g^+_{\varepsilon}) \big)}{ g^-(0) - g^-_{\varepsilon}(0) } 
\nonumber \\
=& e^{-\Psi(\Phi(g^-(0)))} E[ e^{\Psi( W^{\Phi(g^-(0)), (\phi(g^-) , \phi(g^+))}_{[0,t]}(t) )} e^{-\frac{1}{2}N_{[0,t]}(W^{\Phi(g^-(0)), (\phi(g^-) , \phi(g^+))}_{[0,t]})} ] 
\nonumber \\
&\times 
E\left[\widetilde{Z}_{[0, t]}^{\phi(g^-)} 
(W_{[0, t]}^+|_{K_{[0, t]}^-( \phi(g^+) - \phi(g^-) )})^{-1} \right]
\nonumber \\
&\times 
P\left(W_{[0, t]}^+ \in K_{[0, t]}^-( \phi(g^+) - \phi(g^-) ) \right)
\sqrt{\frac{2}{\pi t}} \cdot \frac{1}{\sigma(g^-(0))}
\label{Lem_for_radon_nikodym_conv2}
\end{align}
Therefore, using 
Theorem~\ref{Thm_Def_and_Decomp_curved_Moving}, 
\eqref{bunshi_bunbo_Markov_B0tob}, 
\eqref{Def_Cgmgp}, 
\eqref{Lem_for_radon_nikodym_conv1}, 
\eqref{Lem_for_radon_nikodym_conv2} and 
Lemma~\ref{Lem_WeakConv_Meander_btwn_2crvs}, 
we obtain
\begin{align}
&E\left[ F(\pi_{[0,t]}(X^{a,b,(g^-,g^+)})) \right]
=\lim_{\varepsilon \downarrow 0} 
E[F(\pi_{[0,t]}(X_{[0,1]}^{a \to b} |_{K_{[0, 1]}(g^-_{\varepsilon}, g^+_{\varepsilon})}))]
\nonumber \\
&\quad =\frac{D(t)}{C_{g^-,g^+}}\cdot
E\left[ F\left(X_{[0,t]}^{a,(g^-,g^+)}\right) 
p^{(g^-,g^+), (\downarrow)}_{[t,1]}\Big(X_{[0,t]}^{a,(g^-,g^+)}(t) \Big) \right] .
\label{CVF_Meander} 
\end{align}
Further, combining \eqref{CVF_Meander}, Lemma~\ref{Lem_WeakConv_Meander_btwn_2crvs}, and 
a change of measure formula between Brownian meander and BES$(3)$-process (\cite{bib_Imhof}), we obtain
\begin{align}
&E\left[ F(\pi_{[0,t]}(X^{a,b,(g^-,g^+)})) \right] \nonumber \\
&\quad 
=\frac{e^{-\Psi(\Phi(g^-(0)))} }{C_{g^-,g^+} \sigma(g^-(0))}
\sqrt{\frac{2}{\pi t}} 
\nonumber \\
&\qquad \times
E \Bigg[ 
\dfrac{ e^{\Psi\big(W_{[0,t]}^+(t) + \phi(g^-)(t)\big)}
e^{-\frac{1}{2}N_{[0,t]}\big(W_{[0,t]}^+ + \phi(g^-)\big)} 
p^{(g^-,g^+), (\downarrow)}_{[t,1]}\Big(
\phi^{-1}\big( W_{[0,t]}^+ +\phi(g^-) \big)(t)
\Big) }
{ Z_{[0,t]}^{\phi(g^-)} \big(W_{[0,t]}^+ + \phi(g^-) \big) }
\nonumber \\
&\qquad \qquad \times F \Big( \phi^{-1} \big(W_{[0,t]}^+ +\phi(g^-)\big) \Big)
\cdot 1_{K_{[0, t]}^-( g^+)}\Big(\phi^{-1}\big( W_{[0, t]}^+ + \phi(g^-) \big) \Big)
\Bigg] \nonumber \\
&\quad 
=\frac{e^{-\Psi(\Phi(g^-(0)))} }{C_{g^-,g^+} \sigma(g^-(0))}
\nonumber \\
&\qquad \times
E \Bigg[ 
\dfrac{ e^{\Psi\big(R_{[0,t]}(t) + \phi(g^-)(t)\big)}
e^{-\frac{1}{2}N_{[0,t]}\big(R_{[0,t]} + \phi(g^-)\big)} 
p^{(g^-,g^+), (\downarrow)}_{[t,1]}\Big(
\phi^{-1}\big( R_{[0,t]} +\phi(g^-) \big) (t)
\Big) }
{ R_{[0,t]}(t)\cdot Z_{[0,t]}^{\phi(g^-)} \big(R_{[0,t]} + \phi(g^-) \big) }
\nonumber \\
&\qquad \qquad \times F \Big( \phi^{-1} \big(R_{[0,t]} +\phi(g^-)\big) \Big)
\cdot 1_{K_{[0, t]}^-(g^+)}\Big( \phi^{-1} \big(R_{[0,t]} +\phi(g^-)\big) \Big)
\Bigg] .
\label{proof_abs_conti_Functional_rep}
\end{align}
Using \eqref{proof_abs_conti_Functional_rep}, 
Lebesgue's dominated convergence theorem, and Dynkin's lemma, 
we can prove Theorem~\ref{Thm_abs_conti}. 
\qed

%%%%%%%%%%%%%%%%%%%%%%%%%%%%%%%%%%%%%%%%%%%%%%%%%%%%%%
\section{Proofs of the main results in Subsection~\ref{subsection_diffbridge_between_2curves}}
\label{section_proof_conddiffbridge_between_2curves}
%%%%%%%%%%%%%%%%%%%%%%%%%%%%%%%%%%%%%%%%%%%%%%%%%%%%%%

In this section, we prove 
Theorems~\ref{Thm_Def_and_Decomp_curved_BESbridge_two_curve} and \ref{Thm_WeakConv_BESbridge_to_Moving}.

%%%%%%%%%%%%%%%%%%%%%%%%%%%%%%%%%%%%%%%%%%%%%%%%%%%%
\subsection{Proof of Theorem~\ref{Thm_Def_and_Decomp_curved_BESbridge_two_curve}}
%%%%%%%%%%%%%%%%%%%%%%%%%%%%%%%%%%%%%%%%%%%%%%%%%%%%

In this subsection, we assume that all $X_{[s, t]}^{x, y,(h^-, h^+)}$ are independent. 
For each $\mathbb{R}$-valued bounded continuous function $G$ on $C([0,1], \mathbb{R})$ and $\varepsilon>0$, we define
\begin{align*}
I(\varepsilon, G):=E[G(X^{a\to c}_{[0,1]})\ ;\ X^{a\to c}_{[0,1]} \in K_{[0, 1]}(g^-_{\varepsilon}, g^+_{\varepsilon})]
\cdot \frac{P(X^a_{[0,1]}(1) \in d\widetilde{c})}{d\widetilde{c}}\bigg\vert_{\widetilde{c}=c}.
\end{align*}
Then, we have
\begin{equation}\label{ishitani_BESbridge_2curves}
E[F(X_{[0,1]}^{a\to c} |_{K_{[0, 1]}(g^-_{\varepsilon}, g^+_{\varepsilon})})] 
= \frac{I(\varepsilon, F)}{I(\varepsilon, 1)}.
\end{equation}
Further, by Lemma~\ref{Lem_MarkovDecomp_B_Bridge}, we obtain
\begin{align}
I(\varepsilon, F) 
&=\int_{g^-_{\varepsilon}(t)}^{g^+_{\varepsilon}(t)}  dy
\ E[F(X_{[0, t]}^{a,y,(g^-_{\varepsilon}, g^+_{\varepsilon})} \oplus_{t} X_{[t, 1]}^{y,c,(g^-_{\varepsilon}, g^+_{\varepsilon})})] 
\label{BESbridge_MarkovDecomp_single} \\
&\qquad \qquad \quad \times 
P(X^{a\to y}_{[0, t]} \in K_{[0, t]}(g^-_{\varepsilon}, g^+_{\varepsilon}))
\cdot \frac{P(X^a_{[0, t]}(t) \in dy)}{dy}
\nonumber \\
&\qquad \qquad \quad \times 
P(X^{y\to c}_{[t, 1]} \in K_{[t, 1]}(g^-_{\varepsilon}, g^+_{\varepsilon}))
\cdot \frac{P(X^y_{[t, 1]}(1) \in d\widetilde{c})}{d\widetilde{c}}\bigg\vert_{\widetilde{c}=c}.
\nonumber 
\end{align}
It follows from \eqref{BESbridge_MarkovDecomp_single}, 
Lemma~\ref{Lem_for_H_gm_to_gp}, 
Lemma~\ref{Ap_Lem_weakconv_product} and Lemma~\ref{Lem_pX_conv_1} that
\begin{align}
I(F):=&\lim_{\varepsilon \downarrow 0} 
\frac{I(\varepsilon, F)}{ g^-(0) - g^-_{\varepsilon}(0) }
= \int_{g^-(t)}^{g^+(t)}
E\left[F\left(X_{[0, t]}^{a,y,(g^-, g^+)} \oplus_{t} X_{[t, 1]}^{y,c,(g^-, g^+)}\right)\right] 
p^{(g^-,g^+), (\uparrow)}_{[0,t]}(y) 
p^{(g^-,g^+)}_{[t,1]}(y, c) dy .
\label{IF_t_BESbridge_decomp}
\end{align}
According to \eqref{IF_t_BESbridge_decomp}, we get
\begin{align}
&\int_{g^-(t)}^{g^+(t)}
p^{(g^-,g^+), (\uparrow)}_{[0,t]}(y)p^{(g^-,g^+)}_{[t,1]}(y,c)dy 
\nonumber \\
&\quad 
=I(1)=
\lim_{\varepsilon \downarrow 0}
\frac{P(X^{a\to c}_{[0,1]} \in K_{[0, 1]}(g^-_{\varepsilon}, g^+_{\varepsilon}))}{ g^-(0) - g^-_{\varepsilon}(0) }
\cdot \frac{P(X^a_{[0,1]}(1) \in d\widetilde{c})}{d\widetilde{c}}\bigg\vert_{\widetilde{c}=c}
=\widehat{C}_{g^-,g^+} . 
\label{I1_t_BESbridge_decomp}
\end{align}
Using \eqref{ishitani_BESbridge_2curves}, \eqref{IF_t_BESbridge_decomp} and \eqref{I1_t_BESbridge_decomp}, we obtain
\begin{align*}
\widehat{C}_{g^-,g^+}=\int_{g^-(t)}^{g^+(t)}
p^{(g^-,g^+), (\uparrow)}_{[0,t]}(y)p^{(g^-,g^+)}_{[t,1]}(y, c) dy \in (0, \infty)
\end{align*}
and
\begin{align*}
&\lim_{\varepsilon \downarrow 0} 
E[F(X_{[0, 1]}^{a\to c} |_{K_{[0, 1]}(g^-_{\varepsilon}, g^+_{\varepsilon})})]
= \frac{I(F)}{I(1)} 
= \int_{g^-(t)}^{g^+(t)} 
E\left[ F\left(X_{[0,t]}^{a,y,(g^-, g^+)} \oplus_{t} X_{[t,1]}^{y,c,(g^-, g^+)}\right) \right] l(t,y) dy .
\end{align*}
Therefore, we can define the probability measure $\widehat{P}^{a\to c}$
on $(C([0,1], \mathbb{R}), \mathcal{B}(C([0,1], \mathbb{R})))$ as
\begin{align*}
\widehat{P}^{a\to c}(A):= \int_{g^-(t)}^{g^+(t)} 
P\left( X_{[0,t]}^{a,y,(g^-, g^+)} \oplus_{t} X_{[t,1]}^{y,c,(g^-, g^+)} \in A \right) l(t,y) dy
\quad (A\in \mathcal{B}(C([0,1], \mathbb{R}))),
\end{align*}
and there exists an $\mathbb{R}$-valued 
continuous stochastic process $r_X^{a \to c, (g^-, g^+)}=\{r_X^{a \to c, (g^-, g^+)}(t)\}_{t\in [0,1]}$
that satisfies \eqref{curvedBESbridge_twocurve_def_eq1} and \eqref{twocurve_BESbridge_decomp_1}. 
Thus, a limit argument on $F$ yields 
\begin{align*}
P(r_X^{a \to c, (g^-, g^+)}(t) \in dy) = l(t, y)dy\quad (y \in (g^-(t), g^+(t))).
\end{align*}
On the other hand, by Lemma~\ref{Lem_MarkovDecomp_B_Bridge}, we obtain
\begin{align}
I(\varepsilon, F) 
&= \int_{g^-_{\varepsilon}(t_1)}^{g^+_{\varepsilon}(t_1)} dy_1 \int_{g^-_{\varepsilon}(t_2)}^{g^+_{\varepsilon}(t_2)} dy_2 
\label{BESbridge_MarkovDecomp_double} \\
&\qquad \times E[F(X_{[0, t_1]}^{a,y_1,(g^-_{\varepsilon}, g^+_{\varepsilon})} \oplus_{t_1} 
X_{[t_1,t_2]}^{y_1, y_2,(g^-_{\varepsilon}, g^+_{\varepsilon})} \oplus_{t_2} 
X_{[t_2, 1]}^{y_2,c,(g^-_{\varepsilon}, g^+_{\varepsilon})})] \nonumber \\
&\qquad \times 
P(X^{a\to y_1}_{[0, t_1]} \in K_{[0, t_1]}(g^-_{\varepsilon}, g^+_{\varepsilon}))
\cdot \frac{P(X^a_{[0, t_1]}(t_1) \in dy_1)}{dy_1} \nonumber \\
&\qquad \times 
P(X^{y_1\to y_2}_{[t_1,t_2]} \in K_{[t_1,t_2]}(g^-_{\varepsilon}, g^+_{\varepsilon}))
\cdot \frac{P(X^{y_1}_{[t_1,t_2]}(t_2) \in dy_2)}{dy_2} \nonumber \\
&\qquad \times 
P(X^{y_2\to c}_{[t_2, 1]} \in K_{[t_2, 1]}(g^-_{\varepsilon}, g^+_{\varepsilon}))
\cdot \frac{P(X^{y_2}_{[t_2, 1]}(1) \in d\widetilde{c})}{d\widetilde{c}}\bigg\vert_{\widetilde{c}=c}. \nonumber
\end{align}
By \eqref{BESbridge_MarkovDecomp_double},
Lemma~\ref{Lem_for_H_gm_to_gp}, 
Lemma~\ref{Ap_Lem_weakconv_product} and Lemma~\ref{Lem_pX_conv_1}, 
$I(F)$ satisfies
\begin{align}
I(F)=&\lim_{\varepsilon \downarrow 0} 
\frac{I(\varepsilon, F)}
{ g^-(0) - g^-_{\varepsilon}(0) } 
\nonumber \\
=&%%%\sqrt{\frac{2}{\pi}}
\int_{g^-(t_1)}^{g^+(t_1)} \int_{g^-(t_2)}^{g^+(t_2)} 
E\left[F\left(X_{[0,t_1]}^{a,y_1,(g^-, g^+)} \oplus_{t_1} 
X_{[t_1,t_2]}^{y_1, y_2,(g^-, g^+)} \oplus_{t_2} 
X_{[t_2, 1]}^{y_2, c,(g^-, g^+)}\right)\right] 
\nonumber \\
&\qquad \qquad \qquad \times 
p^{(g^-,g^+), (\uparrow)}_{[0,t_1]}(y_1)p^{(g^-,g^+)}_{[t_1,t_2]}(y_1, y_2) p^{(g^-,g^+)}_{[t_2,1]}(y_2, c) dy_1 dy_2
\nonumber \\
=&%%%\sqrt{\frac{2}{\pi}}
\widehat{C}_{g^-,g^+}
\int_{g^-(t_1)}^{g^+(t_1)} \int_{g^-(t_2)}^{g^+(t_2)} 
E\left[F\left(X_{[0,t_1]}^{a,y_1,(g^-, g^+)} \oplus_{t_1} 
X_{[t_1,t_2]}^{y_1, y_2,(g^-, g^+)} \oplus_{t_2} 
X_{[t_2, 1]}^{y_2, c,(g^-, g^+)}\right)\right] \label{IF_doubleIntg_lt1t2}\\
& \qquad \qquad \qquad \qquad \times 
l(t_1, y_1) l(t_1, y_1, t_2, y_2) dy_1 dy_2 . \nonumber
\end{align}
It follows from \eqref{ishitani_BESbridge_2curves}, \eqref{I1_t_BESbridge_decomp} and \eqref{IF_doubleIntg_lt1t2} that
\begin{align*}
&E[F(X^{a, c,(g^-,g^+)})] 
=\lim_{\varepsilon \downarrow 0} 
E[F(X_{[0, 1]}^{a\to c} |_{K_{[0, 1]}(g^-_{\varepsilon}, g^+_{\varepsilon})})]\\
&\quad =\int_{g^-(t_1)}^{g^+(t_1)} \int_{g^-(t_2)}^{g^+(t_2)} 
E\left[F\left(X_{[0,t_1]}^{a,y_1,(g^-, g^+)} \oplus_{t_1} 
X_{[t_1,t_2]}^{y_1, y_2,(g^-, g^+)}\oplus_{t_2} 
X_{[t_2, 1]}^{y_2, c,(g^-, g^+)}\right)\right] \\
&\qquad \qquad \qquad \qquad \times 
l(t_1, y_1)l(t_1, y_1, t_2, y_2) dy_1 dy_2 .
\end{align*}
Hence, \eqref{twocurve_BESbridge_decomp_2} holds. 
Similarly, using a limit argument on $F$, we can deduce for 
$y_1 \in (g^-(t_1), g^+(t_1))$ and $y_2 \in (g^-(t_2), g^+(t_2))$ that
\begin{align}
P(r_X^{a \to c, (g^-, g^+)}(t_1) \in dy_1, r_X^{a \to c, (g^-, g^+)}(t_2) \in dy_2) &= l(t_1, y_1)l(t_1, y_1, t_2, y_2) dy_1 dy_2, 
\nonumber \\
P(r_X^{a \to c, (g^-, g^+)}(t_2) \in dy_2 \ \vert \ r_X^{a \to c, (g^-, g^+)}(t_1)=y_1) &=l(t_1, y_1, t_2, y_2) dy_2.
\label{r0toc_transdensity}
\end{align}
Because we have
\begin{align*}
p^{(g^-,g^+)}_{[t_1,1]}(y_1, c) &= 
\int_{g^-(t_2)}^{g^+(t_2)} 
p^{(g^-,g^+)}_{[t_1,t_2]} (y_1, y_2) p^{(g^-,g^+)}_{[t_2,1]}(y_2, c) dy_2 
\qquad (y_1 \in (g^-(t_1), g^+(t_1))),
\end{align*}
we can deduce that
\begin{align}\label{lt1y1t2y2_density}
\int_{g^-(t_2)}^{g^+(t_2)} l(t_1, y_1, t_2, y_2) dy_2 = 1
\qquad (y_1 \in (g^-(t_1), g^+(t_1))).
\end{align}
The Chapman--Kolmogorov identity for $l$ is obtained by \eqref{Chapman_id_p_y1_y3}. 
Therefore, \eqref{r0toc_transdensity} and \eqref{lt1y1t2y2_density} imply that 
$r_X^{a \to c, (g^-, g^+)}=\{r_X^{a \to c, (g^-, g^+)}(t)\}_{t\in [0,1]}$ is a Markov process.

%%%%%%%%%%%%%%%%%%%%%%%%%%%%%%%%%%%%%%%%%%%%%%%%%%
\subsection{Proof of Theorem~\ref{Thm_WeakConv_BESbridge_to_Moving}}
%%%%%%%%%%%%%%%%%%%%%%%%%%%%%%%%%%%%%%%%%%%%%%%%%%

In this subsection, we assume that all $X_{[s, t]}^{x, y,(h^-, h^+)}$ are independent. 
For each $\mathbb{R}$-valued bounded continuous function $G$ on $C([0,1], \mathbb{R})$, $\varepsilon>0$ and $\eta>0$, we define
\begin{align*}
I(\varepsilon, \eta, G):=E[G(X^{a\to b}_{[0,1]})\ ;\ X^{a\to b}_{[0,1]} \in K_{[0, 1]}(g^-_{\varepsilon}, g^+_{\eta})]
\cdot \frac{P(X^a_{[0,1]}(1) \in d\widetilde{b})}{d\widetilde{b}}\bigg\vert_{\widetilde{b}=b}.
\end{align*}
Then, we have
\begin{equation}\label{ishitani_Xbridge_2curves_pmepeta}
E[F(X_{[0,1]}^{a\to b} |_{K_{[0, 1]}(g^-_{\varepsilon}, g^+_{\eta})})] 
= \frac{I(\varepsilon, \eta, F)}{I(\varepsilon, \eta, 1)}.
\end{equation}
Further, by Lemma~\ref{Lem_MarkovDecomp_B_Bridge}, we obtain
\begin{align}
I(\varepsilon, \eta, F) 
&=\int_{g^-_{\varepsilon}(t)}^{g^+_{\eta}(t)}  dy
\ E[F(X_{[0, t]}^{a,y,(g^-_{\varepsilon}, g^+_{\eta})} \oplus_{t} X_{[t, 1]}^{y,b,(g^-_{\varepsilon}, g^+_{\eta})})] 
\label{Xbridge_MarkovDecomp_2curves_pmepeta} \\
&\qquad \qquad \quad \times 
P(X^{a\to y}_{[0, t]} \in K_{[0, t]}(g^-_{\varepsilon}, g^+_{\eta}))
\cdot \frac{P(X^a_{[0, t]}(t) \in dy)}{dy}
\nonumber \\
&\qquad \qquad \quad \times 
P(X^{y\to b}_{[t, 1]} \in K_{[t, 1]}(g^-_{\varepsilon}, g^+_{\eta}))
\cdot \frac{P(X^y_{[t, 1]}(1) \in d\widetilde{b})}{d\widetilde{b}}\bigg\vert_{\widetilde{b}=b}.
\nonumber 
\end{align}
It follows from \eqref{I1_t_decomp}, \eqref{Xbridge_MarkovDecomp_2curves_pmepeta}, 
Lemma~\ref{Lem_for_H_gm_to_gp}, 
Lemma~\ref{Ap_Lem_weakconv_product}, Lemma~\ref{Lem_pX_conv_1} and Lemma~\ref{Lem_pX_conv_2} that
\begin{align}
I(F) := & \lim_{\eta \downarrow 0} \lim_{\varepsilon \downarrow 0}
\frac{I(\varepsilon, \eta, F)}{ ( g^-(0) - g^-_{\varepsilon}(0) )  ( g^+_{\eta}(1) - g^+(1) ) } 
\nonumber \\
= & \int_{g^-(t)}^{g^+(t)}
E\left[F\left(X_{[0, t]}^{a,y,(g^-, g^+)} \oplus_{t} X_{[t, 1]}^{y,b,(g^-, g^+)}\right)\right] 
p^{(g^-,g^+), (\uparrow)}_{[0,t]}(y) 
p^{(g^-,g^+), (\downarrow)}_{[t,1]}(y) dy ,
\label{IF_t_decomp_WeakConv_BESbridge_to_Moving} \\
I(1):= & \lim_{\eta \downarrow 0} \lim_{\varepsilon \downarrow 0}
\frac{I(\varepsilon, \eta, 1)}{ (g^-(0) - g^-_{\varepsilon}(0))  ( g^+_{\eta}(1) - g^+(1) ) }
\nonumber \\
= & \int_{g^-(t)}^{g^+(t)}
p^{(g^-,g^+), (\uparrow)}_{[0,t]}(y)p^{(g^-,g^+), (\downarrow)}_{[t,1]}(y)dy = C_{g^-,g^+} . 
\label{I1_t_decomp_WeakConv_BESbridge_to_Moving}
\end{align}
Using Theorem~\ref{Thm_Def_and_Decomp_curved_BESbridge_two_curve}, 
\eqref{ishitani_Xbridge_2curves_pmepeta}, 
\eqref{IF_t_decomp_WeakConv_BESbridge_to_Moving} and 
\eqref{I1_t_decomp_WeakConv_BESbridge_to_Moving}, we obtain
\begin{align*}
&\lim_{\eta \downarrow 0} 
E[F(r_X^{a \to b, (g^-,g^+_{\eta})})]=
\lim_{\eta \downarrow 0} \lim_{\varepsilon \downarrow 0} 
E[F(X_{[0, 1]}^{a\to b} |_{K_{[0, 1]}(g^-_{\varepsilon}, g^+_{\eta})})]\\
&\quad = \frac{I(F)}{I(1)} 
= \int_{g^-(t)}^{g^+(t)} 
E\left[ F\left(X_{[0,t]}^{a,y,(g^-, g^+)} \oplus_{t} X_{[t,1]}^{y,b,(g^-, g^+)}\right) \right] h(t,y) dy =E[F(X^{a,b,(g^-, g^+)})].
\end{align*}

%%%%%%%%%%%%%%%%%%%%%%%
%%%%%%%%%%%%%%%%%%%%%%%
%%%%%%%%%%%%%%%%%%%%%%%
\appendix
\section{Appendix}
%%%%%%%%%%%%%%%%%%%%%%%
%%%%%%%%%%%%%%%%%%%%%%%
%%%%%%%%%%%%%%%%%%%%%%%

In this appendix, we prepare several results. 
Although some of the results in this appendix are either well known or easy to obtain, 
we prove some of them for completeness.

\subsection{Some basic results}

Here, we prepare some basic results which are used throughout the paper.

The following lemma is obtained by Scheff\'e's lemma and the dominated convergence theorem.
\begin{lem}\label{Ap_Lem_ScheffeType}
Let $(X, \mathfrak{M}, \mu)$ be a measure space. 
Let $\gamma_n$ be a sequence of $\mu$-integrable functions 
that converge almost everywhere to another $\mu$-integrable function $\gamma$. 
Let $\kappa_n$ be a sequence of $\mathfrak{M}$-measurable functions 
that converge almost everywhere to another $\mathfrak{M}$-measurable function $\kappa$. 
Assume that
\begin{align*}
\lim_{n\to \infty}\int_X \vert \gamma_n(x)\vert \mu(dx)=\int_X \vert \gamma(x)\vert \mu(dx)
\quad \mbox{and}\quad
\sup_{n\in \mathbb{N}}\left\{ \esssup_{x\in X}\vert \kappa_n(x)\vert \right\} < \infty .
\end{align*}
Then we obtain
\begin{align*}
\lim_{n\to \infty}\int_X \vert \kappa_n(x)\gamma_n(x) - \kappa(x)\gamma(x)\vert \mu(dx)=0.
\end{align*}
\end{lem}

The following lemma is obtained by Prohorov's theorem.
\begin{lem}\label{Ap_Lem_weakconv_product}
Let $S_1$ and $S_2$ be Polish spaces.  
Let $X_n$ and $Y_n$ be random variables 
defined on $(\Omega_n, \mathcal{F}_n, P_n)$ 
that take values in $S_1$ and $S_2$, respectively. 
If $X_n$ and $Y_n$ are independent and 
$P_n \circ X_n^{-1}$ and $P_n \circ Y_n^{-1}$ converge weakly to 
probability measures $Q$ on $S_1$ and $R$ on $S_2$, respectively, 
then $P_n \circ (X_n, Y_n)^{-1}$ 
converges weakly to the product measure $Q \times R$.
\end{lem}

The following lemma is obtained by Skorohod's theorem.
\begin{lem}\label{Ap_Lem_conv_restricted_Expectation}
Let $S$ be a Polish space. 
Let $X_n$ and $X$ be random variables defined on 
$(\Omega_n, \mathcal{F}_n, P_n)$ 
and $(\Omega, \mathcal{F}, P)$ that take values in $S$. 
Assume that $X_n \xrightarrow{\mathcal{D}} X$ holds and 
$A \in \mathcal{B}(S)$ satisfies $P(X \in \partial A)=0$. 
Then, for every $\mathbb{R}$-valued bounded continuous function $G$ on $S$, we have
\begin{align*}
\lim_{n \to \infty} E_n[G(X_n)\ ;\ X_n \in A] = E[G(X)\ ;\ X \in A], 
\end{align*}
where $E_n$ denotes the expectation operator under $P_n$.
\end{lem}

The following lemma is obtained by Lemma~A.4 and Corollary~A.2 in \cite{bib_ishitani_hatakenaka_suzuki}.
\begin{lem}\label{Lem_Int_and_Bdry_Prob_Meander_BESbridge}
Let $g$ be an $\mathbb{R}$-valued $C^1$-function defined on $[0,1]$ that satisfies $\min_{0 \leq t \leq 1} g(t) > 0$.  
Then, for $b\in [0, g(1))$, we have $P(r^{0 \to b} \in \partial K_{[0,1]}^-(g))=P(W^+ \in \partial K_{[0,1]}^-(g)) = 0$, 
$P(r^{0 \to b} \in K_{[0,1]}^-(g))>0$ and $P(W^+ \in K_{[0,1]}^-(g))>0$.
\end{lem}

Proposition~\ref{Prop_Boundary_purturb_WeakConv_for_Meander_BES3bridge} is used in 
Sections~\ref{section_proof_Preparation_for_Proofs},
\ref{section_proof_Meander_between_2curves} and Appendix.
In Proposition~\ref{Prop_Boundary_purturb_WeakConv_for_Meander_BES3bridge} and its proof, 
$W$ denotes one-dimensional Brownian motion on the time interval $[0,1]$.

\begin{prop}\label{Prop_Boundary_purturb_WeakConv_for_Meander_BES3bridge}
Let $g$ be an $\mathbb{R}$-valued $C^1$-function defined on $[0,1]$ that satisfies $\min_{0 \leq t \leq 1} g(t) > 0$.  
Let $0\leq b<g(1)$. 
Assume that $\{ g_{\varepsilon} \}_{\varepsilon>0}\subset C([0,1], \mathbb{R})$ and 
$\{b(\varepsilon)\}_{\varepsilon>0}\subset \mathbb{R}$
satisfy
\begin{align}
\lim_{\varepsilon \downarrow 0} 
\sup_{0\leq t\leq 1}\vert g_{\varepsilon}(t) - g(t)\vert =0, 
\quad
\lim_{\varepsilon \downarrow 0} b(\varepsilon)=b
\quad \mbox{and}\quad
b(\varepsilon)>-\varepsilon \quad (\varepsilon>0).
\end{align} 
Then, we have
\begin{align}
\lim_{\varepsilon \downarrow 0}P(W |_{K_{[0,1]}^+(-\varepsilon)}\in K_{[0,1]}^-(g_{\varepsilon}))&=P(W^+ \in K_{[0,1]}^-(g)), 
\label{Prop_Boundary_purturb_ProbConv_for_Meander}\\
\lim_{\varepsilon \downarrow 0}P(B^{0\to b(\varepsilon)} |_{K_{[0,1]}^+(-\varepsilon)}\in K_{[0,1]}^-(g_{\varepsilon}))&=P(r^{0\to b} \in K_{[0,1]}^-(g)).
\label{Prop_Boundary_purturb_ProbConv_for_BES3bridge}
\end{align}
Further, for every $\mathbb{R}$-valued bounded continuous function $F$ on $C([0,1],\mathbb{R})$, we have
\begin{align}
\lim_{\varepsilon \downarrow 0}E[F(W |_{K_{[0,1]}(- \varepsilon, g_{\varepsilon})})]&=E[F(W^+ |_{K_{[0,1]}^-(g)})], 
\label{Prop_Boundary_purturb_WeakConv_for_Meander}\\
\lim_{\varepsilon \downarrow 0}E[F(B^{0\to b(\varepsilon)} |_{K_{[0,1]}(- \varepsilon, g_{\varepsilon})})]&=E[F(r^{0\to b} |_{K_{[0,1]}^-(g)})].
\label{Prop_Boundary_purturb_WeakConv_for_BES3bridge}
\end{align}
\end{prop}
\begin{proof}
Using Lemma~\ref{Lem_Int_and_Bdry_Prob_Meander_BESbridge} 
and the fact that $W |_{K_{[0,1]}^+(-\varepsilon)}$ converges weakly to $W^+$ (\cite{bib_Durret_1977}), we obtain 
\begin{align}
P(W^+ \in K_{[0,1]}^-(g-\delta))
&=\lim_{\varepsilon \downarrow 0}P(W |_{K_{[0,1]}^+(-\varepsilon)}\in K_{[0,1]}^-(g-\delta))\nonumber \\
&\leq \varliminf_{\varepsilon \downarrow 0}P(W |_{K_{[0,1]}^+(-\varepsilon)}\in K_{[0,1]}^-(g_{\varepsilon})), 
\label{ineq1_prob_liminf_meander}\\
\varlimsup_{\varepsilon \downarrow 0}P(W |_{K_{[0,1]}^+(-\varepsilon)}\in K_{[0,1]}^-(g_{\varepsilon}))
&\leq \lim_{\varepsilon \downarrow 0}P(W |_{K_{[0,1]}^+(-\varepsilon)}\in K_{[0,1]}^-(g+ \delta))\nonumber \\
&= P(W^+ \in K_{[0,1]}^-(g+\delta))
\label{ineq2_prob_liminf_meander}
\end{align}
for every $\delta > 0$.
Similarly, using Lemma~\ref{Lem_Int_and_Bdry_Prob_Meander_BESbridge} and the fact that 
$B^{0\to b(\varepsilon)} |_{K_{[0,1]}^+(-\varepsilon)} \overset{\mathcal{D}}{=} 
r^{\varepsilon \to b(\varepsilon)+\varepsilon} - \varepsilon$ 
converges weakly to $r^{0\to b}$ (Proposition~2.4 in \cite{bib_Wu}), we obtain
\begin{align}
P(r^{0 \to b} \in K_{[0,1]}^-(g-\delta))
&=\lim_{\varepsilon \downarrow 0}P(B^{0\to b(\varepsilon)} |_{K_{[0,1]}^+(-\varepsilon)}\in K_{[0,1]}^-(g - \delta ))\nonumber \\
&\leq \varliminf_{\varepsilon \downarrow 0}P(B^{0\to b(\varepsilon)} |_{K_{[0,1]}^+(-\varepsilon)}\in K_{[0,1]}^-(g_{\varepsilon})), 
\label{ineq1_prob_liminf_bes3bridge}\\
\varlimsup_{\varepsilon \downarrow 0}P(B^{0\to b(\varepsilon)} |_{K_{[0,1]}^+(-\varepsilon)}\in K_{[0,1]}^-(g_{\varepsilon}))
&\leq \lim_{\varepsilon \downarrow 0}P(B^{0\to b(\varepsilon)} |_{K_{[0,1]}^+(-\varepsilon)}\in K_{[0,1]}^-(g + \delta ))\nonumber \\
&= P(r^{0 \to b} \in K_{[0,1]}^-(g+\delta))
\label{ineq2_prob_liminf_bes3bridge}
\end{align}
for every $\delta > 0$.
Then, it follows from 
\eqref{ineq1_prob_liminf_meander}, \eqref{ineq2_prob_liminf_meander}, 
\eqref{ineq1_prob_liminf_bes3bridge}, \eqref{ineq2_prob_liminf_bes3bridge}, 
Lemma~\ref{Lem_Int_and_Bdry_Prob_Meander_BESbridge} and 
\begin{align*}
&\bigcup_{\delta>0}K_{[0,1]}^-(g- \delta )=\mbox{\rm int} (K_{[0,1]}^-(g)), \quad
\bigcap_{\delta>0}K_{[0,1]}^-(g+ \delta )=K_{[0,1]}^-(g), \\
&\partial K_{[0,1]}^-(g)=K_{[0,1]}^-(g)\setminus \mbox{\rm int} (K_{[0,1]}^-(g))
\end{align*} 
that 
\eqref{Prop_Boundary_purturb_ProbConv_for_Meander} and 
\eqref{Prop_Boundary_purturb_ProbConv_for_BES3bridge} hold. 

Because $W |_{K_{[0,1]}^+(-\varepsilon)}$ converges weakly to $W^+$ (\cite{bib_Durret_1977}), 
Lemmas~\ref{Lem_Int_and_Bdry_Prob_Meander_BESbridge} and \ref{Ap_Lem_conv_restricted_Expectation} imply that 
\begin{align}
E[F(W |_{K_{[0,1]}(- \varepsilon, g)})]
&=\frac{E[F(W |_{K_{[0,1]}^+(-\varepsilon)}) \ ;\ W |_{K_{[0,1]}^+(-\varepsilon)} \in K_{[0,1]}^-(g)]}{P(W |_{K_{[0,1]}^+(-\varepsilon)} \in K_{[0,1]}^-(g))} 
\nonumber \\
&\to \frac{E[F(W^+) \ ;\ W^+\in K_{[0,1]}^-(g)]}{P(W^+ \in K_{[0,1]}^-(g))}=
E[F(W^+ |_{K_{[0,1]}^-(g)})]\qquad (\varepsilon \downarrow 0). 
\label{EQ_Pre_Result}
\end{align}
On the other hand, because we have
\begin{align*}
\varDelta(\varepsilon)
\coloneqq &
\left\vert E[F(W |_{K_{[0,1]}(- \varepsilon, g_{\varepsilon})})] - E[F(W |_{K_{[0,1]}(- \varepsilon, g)})] \right\vert \\
\leq &
\frac{E\Big[\big\vert F(W |_{K_{[0,1]}^+(-\varepsilon)}) \big\vert \ ;\ W |_{K_{[0,1]}^+(-\varepsilon)} \in K_{[0,1]}^-(g_{\varepsilon})\setminus K_{[0,1]}^-(g)\Big]}
{P\left(W |_{K_{[0,1]}^+(-\varepsilon)} \in K_{[0,1]}^-(g_{\varepsilon})\right)} \\ 
&+
\frac{E\Big[\big\vert F(W |_{K_{[0,1]}^+(-\varepsilon)}) \big\vert \ ;\ W |_{K_{[0,1]}^+(-\varepsilon)} \in K_{[0,1]}^-(g)\setminus K_{[0,1]}^-(g_{\varepsilon})\Big]}
{P\left(W |_{K_{[0,1]}^+(-\varepsilon)} \in K_{[0,1]}^-(g_{\varepsilon})\right)} \\
&+ \left\vert \dfrac{E[F(W |_{K_{[0,1]}^+(-\varepsilon)}) \ ;\ W |_{K_{[0,1]}^+(-\varepsilon)} \in K_{[0,1]}^-(g) ]}{P(W |_{K_{[0,1]}^+(-\varepsilon)} \in K_{[0,1]}^-(g))}
\cdot
\left\{ \frac{P\left(W |_{K_{[0,1]}^+(-\varepsilon)} \in K_{[0,1]}^-(g)\right)}{P\left(W |_{K_{[0,1]}^+(-\varepsilon)} \in K_{[0,1]}^-(g_{\varepsilon})\right)} - 1 \right\} \right\vert \\
\leq &
3\|F\|_{\infty}
\frac{P\left(W |_{K_{[0,1]}^+(-\varepsilon)}\in K_{[0,1]}^-(g+\delta)\right)
-P\left(W |_{K_{[0,1]}^+(-\varepsilon)} \in K_{[0,1]}^-(g-\delta)\right)}{P\left(W |_{K_{[0,1]}^+(-\varepsilon)} \in K_{[0,1]}^-(g-\delta)\right)} \\
&\quad \left(\varepsilon>0, \quad 0 < \delta <\min_{0 \leq t \leq 1} g(t)\wedge \min_{0 \leq t \leq 1} g_{\varepsilon}(t) \right)
\end{align*} 
for $\| F \|_{\infty} := \sup_{w\in C([0,1],\mathbb{R})} \vert F(w)\vert$, 
we can deduce that
\begin{align}
\varlimsup_{\varepsilon \downarrow 0}
\varDelta(\varepsilon) 
&\leq 3\|F\|_{\infty}
\frac{P(W^+\in K_{[0,1]}^-(g+\delta))-P(W^+ \in K_{[0,1]}^-(g-\delta))}{P(W^+ \in K_{[0,1]}^-(g-\delta))} 
\to 0 \qquad (\delta \downarrow 0)
\label{varlimsupvarDelta_eq_zero}
\end{align}
by \eqref{Prop_Boundary_purturb_ProbConv_for_Meander} and Lemma~\ref{Lem_Int_and_Bdry_Prob_Meander_BESbridge}. 
Therefore, \eqref{EQ_Pre_Result} and \eqref{varlimsupvarDelta_eq_zero} imply 
\eqref{Prop_Boundary_purturb_WeakConv_for_Meander}. 

Because $B^{0\to b(\varepsilon)} |_{K_{[0,1]}^+(-\varepsilon)}$ converges weakly to $r^{0\to b}$ (\cite{bib_ishitani_hatakenaka_suzuki}),
Lemmas~\ref{Lem_Int_and_Bdry_Prob_Meander_BESbridge} and \ref{Ap_Lem_conv_restricted_Expectation}
imply that 
\begin{align}
E[F(B^{0\to b(\varepsilon)} |_{K_{[0,1]}(- \varepsilon, g)})]
&=\frac{E[F(B^{0\to b(\varepsilon)} |_{K_{[0,1]}^+(-\varepsilon)}) \ ;\ B^{0\to b(\varepsilon)} |_{K_{[0,1]}^+(-\varepsilon)} \in K_{[0,1]}^-(g)]}
{P(B^{0\to b(\varepsilon)} |_{K_{[0,1]}^+(-\varepsilon)} \in K_{[0,1]}^-(g))} 
\nonumber \\
&\to \frac{E[F(r^{0 \to b}) \ ;\ r^{0 \to b}\in K_{[0,1]}^-(g)]}{P(r^{0 \to b} \in K_{[0,1]}^-(g))}=
E[F(r^{0 \to b} |_{K_{[0,1]}^-(g)})]\qquad (\varepsilon \downarrow 0). 
\label{EQ_Pre_Result_BES3bridge}
\end{align}
On the other hand, because we have
\begin{align*}
\widehat{\varDelta}(\varepsilon)
\coloneqq &
\left\vert 
E[F(B^{0\to b(\varepsilon)} |_{K_{[0,1]}(- \varepsilon, g_{\varepsilon})})] - E[F(B^{0\to b(\varepsilon)} |_{K_{[0,1]}(- \varepsilon, g)})] 
\right\vert \\
\leq & 
\frac{E\left[\big\vert F(B^{0\to b(\varepsilon)} |_{K_{[0,1]}^+(-\varepsilon)}) \big\vert 
\ ;\ B^{0\to b(\varepsilon)} |_{K_{[0,1]}^+(-\varepsilon)} \in K_{[0,1]}^-(g_{\varepsilon})\setminus K_{[0,1]}^-(g)\right]}
{P\left(B^{0\to b(\varepsilon)} |_{K_{[0,1]}^+(-\varepsilon)} \in K_{[0,1]}^-(g_{\varepsilon})\right)} \\ 
&+
\frac{E\left[\big\vert F(B^{0\to b(\varepsilon)} |_{K_{[0,1]}^+(-\varepsilon)}) \big\vert 
\ ;\ B^{0\to b(\varepsilon)} |_{K_{[0,1]}^+(-\varepsilon)} \in K_{[0,1]}^-(g)\setminus K_{[0,1]}^-(g_{\varepsilon})\right]}
{P\left(B^{0\to b(\varepsilon)} |_{K_{[0,1]}^+(-\varepsilon)} \in K_{[0,1]}^-(g_{\varepsilon})\right)} \\
&+ \left\vert \dfrac{E\left[F(B^{0\to b(\varepsilon)} |_{K_{[0,1]}^+(-\varepsilon)}) \ ;\ B^{0\to b(\varepsilon)} |_{K_{[0,1]}^+(-\varepsilon)} \in K_{[0,1]}^-(g) 
\right]}{P\left(B^{0\to b(\varepsilon)} |_{K_{[0,1]}^+(-\varepsilon)} \in K_{[0,1]}^-(g)\right)} \right\vert \\
&\quad \times 
\left\vert 
\frac{P\left(B^{0\to b(\varepsilon)} |_{K_{[0,1]}^+(-\varepsilon)} \in K_{[0,1]}^-(g)\right)}
{P\left(B^{0\to b(\varepsilon)} |_{K_{[0,1]}^+(-\varepsilon)} \in K_{[0,1]}^-(g_{\varepsilon})\right)} - 1 \right\vert \\
\leq &
3\|F\|_{\infty}
\frac{P\left(B^{0\to b(\varepsilon)} |_{K_{[0,1]}^+(-\varepsilon)}\in K_{[0,1]}^-(g+\delta)\right)
-P\left(B^{0\to b(\varepsilon)} |_{K_{[0,1]}^+(-\varepsilon)} \in K_{[0,1]}^-(g-\delta)\right)}
{P\left(B^{0\to b(\varepsilon)} |_{K_{[0,1]}^+(-\varepsilon)} \in K_{[0,1]}^-(g-\delta)\right)} \\
&\quad \left(\varepsilon>0, \quad 
0 < \delta <\min_{0 \leq t \leq 1} g(t)\wedge \min_{0 \leq t \leq 1} g_{\varepsilon}(t), \quad
b(\varepsilon)+\delta < g(1)\wedge g_{\varepsilon}(1)
\right)
\end{align*} 
for $\| F \|_{\infty} := \sup_{w\in C([0,1],\mathbb{R})} \vert F(w)\vert$, 
we can deduce that
\begin{align}
\varlimsup_{\varepsilon \downarrow 0}
\widehat{\varDelta}(\varepsilon) 
&\leq 3\|F\|_{\infty}
\frac{P\left(r^{0 \to b}\in K_{[0,1]}^-(g+\delta)\right)-P\left(r^{0 \to b} \in K_{[0,1]}^-(g-\delta)\right)}{P\left(r^{0 \to b} \in K_{[0,1]}^-(g-\delta)\right)} 
\to 0 \qquad (\delta \downarrow 0)
\label{varlimsupvarDelta_eq_zero_BES3bridge}
\end{align}
by \eqref{Prop_Boundary_purturb_ProbConv_for_BES3bridge} and Lemma~\ref{Lem_Int_and_Bdry_Prob_Meander_BESbridge}. 
Therefore, \eqref{EQ_Pre_Result_BES3bridge} and \eqref{varlimsupvarDelta_eq_zero_BES3bridge} imply 
\eqref{Prop_Boundary_purturb_WeakConv_for_BES3bridge}. 
\end{proof}
%%%%%%%%%%%%%%%%%%%%%%%%%%%%%%%%%%%%%%%%%%%%%
\subsection{Preparation for proofs of 
Lemma~\ref{Lem_for_H_gm_to_gp} and 
Lemma~\ref{Lem_WeakConv_Meander_btwn_2crvs}}
%%%%%%%%%%%%%%%%%%%%%%%%%%%%%%%%%%%%%%%%%%%%%

Let $0 \leq t_1 < t_2 \leq 1$. 
Let $f^-$ and $f^+$ be 
$\mathbb{R}$-valued $C^2$-functions defined on $[t_1, t_2]$ that satisfy
\begin{align*}
\min_{t_1 \leq t \leq t_2}(f^+(t) - f^-(t)) > 0.
\end{align*} 
For $\varepsilon>0$, let $f^-_{\varepsilon}$ and $f^+_{\varepsilon}$ be 
$\mathbb{R}$-valued $C^2$-functions defined on $[t_1, t_2]$ that satisfy
\begin{align*}
&\min_{t_1 \leq t \leq t_2}(f^+_{\varepsilon}(t) - f^-_{\varepsilon}(t)) > 0 \qquad (\varepsilon>0), \\
&\lim_{\varepsilon \downarrow 0}\sup_{t_1\leq t\leq t_2}\left\{ \vert f^-_{\varepsilon}(t)-f^-(t)\vert 
+ \left\vert \frac{d}{dt}f^-_{\varepsilon}(t) - \frac{d}{dt}f^-(t)\right\vert 
+ \left\vert \frac{d^2}{dt^2}f^-_{\varepsilon}(t) - \frac{d^2}{dt^2}f^-(t)\right\vert \right\}
=0, \\
&\lim_{\varepsilon \downarrow 0}\sup_{t_1\leq t\leq t_2}\left\{ \vert f^+_{\varepsilon}(t)-f^+(t)\vert 
+ \left\vert \frac{d}{dt}f^+_{\varepsilon}(t) - \frac{d}{dt}f^+(t)\right\vert 
+ \left\vert \frac{d^2}{dt^2}f^+_{\varepsilon}(t) - \frac{d^2}{dt^2}f^+(t)\right\vert \right\}
=0.
\end{align*} 

\begin{lem}\label{Lem_conpoundweakconv_2crvs_BESbridge}
For every $\mathbb{R}$-valued bounded continuous function $F$ on $C([t_1,t_2], \mathbb{R})$,
\begin{itemize}
\item[{\rm ({\mbox{\rm A}})}] 
if $c_1=f^-(t_1)>f^-_{\varepsilon}(t_1)$, $f^-(t_2)\vee f^-_{\varepsilon}(t_2)\leq c_2 < f^+(t_2)$ \ \ $(\varepsilon>0)$, then
\begin{align}
& E\left[F\big(B^{c_1, c_2, (f^-, f^+)}_{[t_1,t_2]}\big) \right] 
= \lim_{\varepsilon \downarrow 0}
E\left[F\big(B^{c_1, c_2, (f^-_{\varepsilon}, f^+_{\varepsilon})}_{[t_1,t_2]}\big) \right] , 
\label{weakconv_G_bridge_on_t1t2_touch_at_t1}
\end{align}
\item[{\rm ({B})}] 
if $f^-(t_1)<c_1\leq f^+(t_1)\wedge f^+_{\varepsilon}(t_1)$, $c_2 = f^+(t_2)<f^+_{\varepsilon}(t_2)$\ \ $(\varepsilon>0)$, then
\begin{align}
& E\left[F\big(B^{c_1, c_2, (f^-, f^+)}_{[t_1,t_2]}\big) \right] 
= \lim_{\varepsilon \downarrow 0}
E\left[F\big(B^{c_1, c_2, (f^-_{\varepsilon}, f^+_{\varepsilon})}_{[t_1,t_2]}\big) \right] ,
\label{weakconv_G_bridge_on_t1t2_touch_at_t2}
\end{align}
\item[{\rm ({C})}] 
if $c_1=f^-(t_1)>f^-_{\varepsilon}(t_1)$ and $c_2 = f^+(t_2)<f^+_{\varepsilon}(t_2)$\ \ $(\varepsilon>0)$, then
\begin{align}
& E\left[F\big(B^{c_1, c_2, (f^-, f^+)}_{[t_1,t_2]}\big) \right] 
= \lim_{\varepsilon \downarrow 0}
E\left[F\big(B^{c_1, c_2, (f^-_{\varepsilon}, f^+_{\varepsilon})}_{[t_1,t_2]}\big) \right] .
\label{weakconv_G_bridge_on_t1t2_touch_at_t1t2}
\end{align}
\end{itemize}
\end{lem}
\begin{proof}
Assume that $h^-$ and $h^+$ are $\mathbb{R}$-valued $C^2$-functions defined on $[t_1, t_2]$ satisfying
\begin{align*}
h^-(t_1)<c_1<h^+(t_1), \quad h^-(t_2)<c_2<h^+(t_2) \quad \mbox{and} \quad
\min_{t\in [t_1,t_2]}(h^+(t)-h^-(t))>0.
\end{align*} 
It follows from Lemma~A.10 in \cite{bib_ishitani_hatakenaka_suzuki} that
\begin{align}
&E[F(W_{[t_1,t_2]}^{c_1})\ ;\ W_{[t_1,t_2]}^{c_1}(t_2) \in dc_2, 
W_{[t_1,t_2]}^{c_1} \in K_{[t_1,t_2]}(h^-, h^+)] \nonumber \\
&\quad = 
E\Big[F\big(h^--h^-(t_1)+B^{c_1\to c_2+h^-(t_1)-h^-(t_2)}_{[t_1,t_2]}\big)
\widetilde{Z}_{[t_1,t_2]}^{h^-} \big(B^{c_1\to c_2+h^-(t_1)-h^-(t_2)}_{[t_1,t_2]}\big)^{-1} \ ;\ \nonumber \\
&\qquad \qquad 
B^{c_1\to c_2+h^-(t_1)-h^-(t_2)}_{[t_1,t_2]} \in K_{[t_1,t_2]}\big(h^-(t_1), h^+-h^- +h^-(t_1)\big)\Big]\nonumber \\
&\qquad \times P\big(h^-(t_2)-h^-(t_1)+W_{[t_1,t_2]}^{c_1}(t_2) \in dc_2\big)
\label{Girsanov_pinned_formula} \\
&\quad = 
E\Big[F\Big(h^++c_2 -h^+(t_2)- \overleftarrow{B}^{0\to c_2-c_1+h^+(t_1)-h^+(t_2)}_{[t_1,t_2]}\Big)
\widetilde{Z}_{[t_1,t_2]}^{-\overleftarrow{h^+}} \big(B^{0\to c_2-c_1+h^+(t_1)-h^+(t_2)}_{[t_1,t_2]}\big)^{-1} 
\ ;\ \nonumber \\ 
&\qquad \qquad 
B^{0\to c_2-c_1+h^+(t_1)-h^+(t_2)}_{[t_1,t_2]}
\in K_{[t_1,t_2]}\big(c_2-h^+(t_2), \overleftarrow{h^+}-\overleftarrow{h^-}+c_2-h^+(t_2)\big)\Big]
\nonumber \\
&\qquad \times P\big(h^+(t_2)-h^+(t_1)+W^{c_1}_{[t_1,t_2]}(t_2) \in dc_2\big).
\label{Girsanov_pinned_invinv_formula}
\end{align}
Then, using \eqref{Girsanov_pinned_formula} and \eqref{Girsanov_pinned_invinv_formula}, we obtain
\begin{align}
&E\left[F(B_{[t_1,t_2]}^{c_1, c_2, (h^-, h^+)})\right] \nonumber \\
&\quad = 
E\Big[F\big(h^--h^-(t_1)+B^{c_1, c_2+h^-(t_1)-h^-(t_2), (h^-(t_1), h^+-h^- +h^-(t_1))}_{[t_1,t_2]}\big)\nonumber \\
&\quad \qquad \times 
\widetilde{Z}_{[t_1,t_2]}^{h^-} \big(B^{c_1, c_2+h^-(t_1)-h^-(t_2), (h^-(t_1), h^+-h^- +h^-(t_1))}_{[t_1,t_2]}\big)^{-1} \Big]\nonumber \\
&\qquad \times 
\left\{E\Big[
\widetilde{Z}_{[t_1,t_2]}^{h^-} \big(B^{c_1, c_2+h^-(t_1)-h^-(t_2), (h^-(t_1), h^+-h^- +h^-(t_1))}_{[t_1,t_2]}\big)^{-1} \Big]\right\}^{-1}
\label{Girsanov_pinned_formula_5} \\
&\quad = 
E\Big[F\Big(h^++c_2 -h^+(t_2)- \overleftarrow{B}^{0\to c_2-c_1+h^+(t_1)-h^+(t_2)}_{[t_1,t_2]}
\big\vert_{K_{[t_1,t_2]}\big(c_2-h^+(t_2), h^+-h^-+c_2-h^+(t_2)\big)}\Big) \nonumber \\
&\quad \qquad \times 
\widetilde{Z}_{[t_1,t_2]}^{-\overleftarrow{h^+}} 
\big(B^{0\to c_2-c_1+h^+(t_1)-h^+(t_2)}_{[t_1,t_2]}
\big\vert_{K_{[t_1,t_2]}\big(c_2-h^+(t_2), \overleftarrow{h^+}-\overleftarrow{h^-}+c_2-h^+(t_2)\big)}\big)^{-1} 
\Big] \nonumber \\
&\qquad \times
\left\{ E\Big[\widetilde{Z}_{[t_1,t_2]}^{-\overleftarrow{h^+}} 
\big(B^{0\to c_2-c_1+h^+(t_1)-h^+(t_2)}_{[t_1,t_2]}
\big\vert_{K_{[t_1,t_2]}\big(c_2-h^+(t_2), \overleftarrow{h^+}-\overleftarrow{h^-}+c_2-h^+(t_2)\big)}\big)^{-1} 
\Big] \right\}^{-1} .
\label{Girsanov_pinned_invinv_formula_5}
\end{align}

First, we prove \eqref{weakconv_G_bridge_on_t1t2_touch_at_t1}. 
Using \eqref{Girsanov_pinned_formula_5}, 
Proposition~\ref{Prop_Boundary_purturb_WeakConv_for_Meander_BES3bridge} \eqref{Prop_Boundary_purturb_WeakConv_for_BES3bridge}, 
and Lemma~3.1 in \cite{bib_ishitani_hatakenaka_suzuki}, we have
\begin{align*}
&\lim_{\varepsilon \downarrow 0}
E\left[F(B_{[t_1,t_2]}^{c_1, c_2, (f^-_{\varepsilon}, f^+_{\varepsilon})})\right] \\
&\quad = 
\dfrac{E\Big[F\big(f^-+r^{0 \to c_2-f^-(t_2)}_{[t_1,t_2]}\vert_{K_{[t_1,t_2]} (0, f^+-f^-)}\big)
\widetilde{Z}_{[t_1,t_2]}^{f^-} \big(r^{0 \to c_2-f^-(t_2)}_{[t_1,t_2]}\vert_{K_{[t_1,t_2]} (0, f^+-f^-)}\big)^{-1} \Big]}
{E\Big[\widetilde{Z}_{[t_1,t_2]}^{f^-} \big(r^{0 \to c_2-f^-(t_2)}_{[t_1,t_2]}\vert_{K_{[t_1,t_2]} (0, f^+-f^-)}\big)^{-1} \Big]}\\
&\quad = E\left[F(B_{[t_1,t_2]}^{c_1, c_2, (f^-, f^+)})\right] .
\end{align*}
Thus, we obtain \eqref{weakconv_G_bridge_on_t1t2_touch_at_t1}. 

Second, we prove \eqref{weakconv_G_bridge_on_t1t2_touch_at_t2}.
Using \eqref{Girsanov_pinned_invinv_formula_5} and 
Proposition~\ref{Prop_Boundary_purturb_WeakConv_for_Meander_BES3bridge} 
\eqref{Prop_Boundary_purturb_WeakConv_for_BES3bridge}, 
and Lemma~3.1 in \cite{bib_ishitani_hatakenaka_suzuki}, we have
\begin{align*}
&\lim_{\varepsilon \downarrow 0}
E\left[F(B_{[t_1,t_2]}^{c_1, c_2, (f^-_{\varepsilon}, f^+_{\varepsilon})})\right] \\
&\quad = 
\frac{
E\Big[F\Big(f^+- \overleftarrow{r}^{0\to f^+(t_1)-c_1}_{[t_1,t_2]}
\big\vert_{K_{[t_1,t_2]}^-(f^+-f^-)}\Big)  
\widetilde{Z}_{[t_1,t_2]}^{-\overleftarrow{f^+}} 
\Big(r^{0\to f^+(t_1)-c_1}_{[t_1,t_2]}
\big\vert_{K_{[t_1,t_2]}^-(\overleftarrow{f^+}-\overleftarrow{f^-})}\Big)^{-1} 
\Big] }
{ E\Big[\widetilde{Z}_{[t_1,t_2]}^{-\overleftarrow{f^+}} 
\Big(r^{0\to f^+(t_1)-c_1}_{[t_1,t_2]}
\big\vert_{K_{[t_1,t_2]}(\overleftarrow{f^+}-\overleftarrow{f^-})}\Big)^{-1} 
\Big]  }\\
&\quad = E\left[F(B_{[t_1,t_2]}^{c_1, c_2, (f^-, f^+)})\right] .
\end{align*}
Thus, we obtain
\eqref{weakconv_G_bridge_on_t1t2_touch_at_t2}.

Finally, we prove \eqref{weakconv_G_bridge_on_t1t2_touch_at_t1t2}. 
Let $t_1 < t < t_2$. 
By Lemma~\ref{Lem_MarkovDecomp_B_Bridge},  we obtain
\begin{align}
& E\left[F(B^{c_1\to c_2}_{[t_1,t_2]})\ ;\ B^{c_1\to c_2}_{[t_1,t_2]} \in K_{[t_1, t_2]}(f^-_{\varepsilon}, f^+_{\varepsilon})\right] 
\cdot \frac{P(W^{c_1}_{[t_1,t_2]}(t_2) \in dc)}{dc}\bigg\vert_{c=c_2} 
\nonumber \\
&\quad = 
\int_{f^-_{\varepsilon}(t)}^{f^+_{\varepsilon}(t)}  dy
\ E\left[F(B_{[t_1, t]}^{c_1,y,(f^-_{\varepsilon}, f^+_{\varepsilon})} \oplus_{t} B_{[t, t_2]}^{y, c_2,(f^-_{\varepsilon}, f^+_{\varepsilon})})\right] 
\label{WConv_BH_prodtransdensities_L1conv_step1} \\
&\qquad \qquad \quad \times 
P\left(B^{c_1\to y}_{[t_1, t]} \in K_{[t_1, t]}(f^-_{\varepsilon}, f^+_{\varepsilon})\right)
\cdot \frac{P(W^{c_1}_{[t_1, t]}(t) \in dy)}{dy}
\nonumber \\
&\qquad \qquad \quad \times 
P\left(B^{y\to c_2}_{[t,t_2]} \in K_{[t, t_2]}(f^-_{\varepsilon}, f^+_{\varepsilon})\right)
\cdot \frac{P(W^y_{[t, t_2]}(t_2) \in dc)}{dc}\bigg\vert_{c=c_2}  .
\nonumber
\end{align}
It follows from \eqref{Girsanov_pinned_formula} that
\begin{align}
&P\left(B^{c_1\to y}_{[t_1, t]} \in K_{[t_1, t]}(f^-_{\varepsilon}, f^+_{\varepsilon})\right)
\cdot \frac{P(W^{c_1}_{[t_1, t]}(t) \in dy)}{dy} \nonumber \\
&\quad = \frac{P\left(W^{c_1}_{[t_1, t]} \in K_{[t_1, t]}(f^-_{\varepsilon}, f^+_{\varepsilon}), W^{c_1}_{[t_1,t]}(t) \in dy \right)}{dy} \nonumber \\
&\quad = E\Big[ \widetilde{Z}_{[t_1,t]}^{f^-_{\varepsilon}} \big(B^{c_1 , y+f^-_{\varepsilon}(t_1)-f^-_{\varepsilon}(t) , ( f^-_{\varepsilon}(t_1), f^+_{\varepsilon}-f^-_{\varepsilon} +f^-_{\varepsilon}(t_1) )}_{[t_1,t]}\big)^{-1} \Big]\nonumber \\
&\qquad \times P\left( B^{c_1\to y+f^-_{\varepsilon}(t_1)-f^-_{\varepsilon}(t)}_{[t_1,t]} \in K_{[t_1,t]}\big(f^-_{\varepsilon}(t_1), f^+_{\varepsilon}-f^-_{\varepsilon} +f^-_{\varepsilon}(t_1)\big) \right) \nonumber \\
&\qquad \times \frac{P\big(f^-_{\varepsilon}(t)-f^-_{\varepsilon}(t_1)+W_{[t_1,t]}^{c_1}(t) \in dy\big)}{dy} , 
\qquad (f^-_{\varepsilon}(t)<y<f^+_{\varepsilon}(t)).
\label{Appd_Lem_C_proof_1}
\end{align}
For  $f^-_{\varepsilon}(t)<y<f^+_{\varepsilon}(t)$, we have
\begin{align}
&P\left( B^{c_1\to y+f^-_{\varepsilon}(t_1)-f^-_{\varepsilon}(t)}_{[t_1,t]} 
\in K_{[t_1,t]}\big(f^-_{\varepsilon}(t_1), f^+_{\varepsilon}-f^-_{\varepsilon} +f^-_{\varepsilon}(t_1)\big) \right) \nonumber \\
&\quad = P\left( B^{c_1\to y+f^-_{\varepsilon}(t_1)-f^-_{\varepsilon}(t)}_{[t_1,t]} |_{K^+_{[t_1,t]}( f^-_{\varepsilon}(t_1) )} \in K^-_{[t_1,t]}(f^+_{\varepsilon}-f^-_{\varepsilon} +f^-_{\varepsilon}(t_1)) \right) \nonumber \\
&\qquad \times P\left( B^{c_1\to y+f^-_{\varepsilon}(t_1)-f^-_{\varepsilon}(t)}_{[t_1,t]} \in K^+_{[t_1,t]}( f^-_{\varepsilon}(t_1) ) \right) 
\label{Appd_Lem_C_proof_2}
\end{align}
and
\begin{align}
&P\left( B^{c_1\to y+f^-_{\varepsilon}(t_1)-f^-_{\varepsilon}(t)}_{[t_1,t]} \in K^+_{[t_1,t]}( f^-_{\varepsilon}(t_1) ) \right) \cdot
\frac{P\big(f^-_{\varepsilon}(t)-f^-_{\varepsilon}(t_1)+W_{[t_1,t]}^{c_1}(t) \in dy\big)}{dy} \nonumber \\
&\quad = P\left( W_{[t_1,t]} \in K^+_{[t_1,t]}( f^-_{\varepsilon}(t_1) - c_1 ) \right) \nonumber \\
&\qquad \times \frac{P\left( f^-_{\varepsilon}(t)-f^-_{\varepsilon}(t_1)+ c_1 + W_{[t_1,t]}|_{K^+_{[t_1,t]}( f^-_{\varepsilon}(t_1) - c_1 )}(t)  \in dy \right)}{dy} .
\label{Appd_Lem_C_proof_3}
\end{align}
Then, using \eqref{Appd_Lem_C_proof_1}, \eqref{Appd_Lem_C_proof_2}, \eqref{Appd_Lem_C_proof_3} 
and Lemma~\ref{Lem_density_ae_convergence},  
we can deduce that
\begin{align}
&\frac{ P(B^{c_1\to y}_{[t_1, t]} \in K_{[t_1, t]}(f^-_{\varepsilon}, f^+_{\varepsilon})) }{ c_1 - f^-_{\varepsilon}(t_1) } 
\cdot \frac{P(W^{c_1}_{[t_1, t]}(t) \in dy)}{dy} \nonumber \\
&\quad =
E\Big[ \widetilde{Z}_{[t_1,t]}^{f^-_{\varepsilon}} 
\big(B^{c_1 , y+f^-_{\varepsilon}(t_1)-f^-_{\varepsilon}(t) , ( f^-_{\varepsilon}(t_1), f^+_{\varepsilon}-f^-_{\varepsilon} +f^-_{\varepsilon}(t_1) )}_{[t_1,t]}\big)^{-1} \Big]\nonumber \\
&\qquad \times P\left( B^{0\to y+f^-_{\varepsilon}(t_1)-f^-_{\varepsilon}(t)-c_1}_{[t_1,t]} |_{K^+_{[t_1,t]}( f^-_{\varepsilon}(t_1) -c_1)} 
\in K^-_{[t_1,t]}(f^+_{\varepsilon}-f^-_{\varepsilon} +f^-_{\varepsilon}(t_1)-c_1) \right) \nonumber \\
&\qquad \times 
\frac{P\left( W_{[t_1,t]} \in K^+_{[t_1,t]}( f^-_{\varepsilon}(t_1) - c_1 ) \right)}{ c_1 - f^-_{\varepsilon}(t_1) } \cdot 
\frac{P\left( f^-_{\varepsilon}(t)-f^-_{\varepsilon}(t_1)+ c_1 + W_{[t_1,t]}|_{K^+_{[t_1,t]}( f^-_{\varepsilon}(t_1) - c_1 )}(t)  \in dy \right)}{dy}
\nonumber \\
&\quad \to
E\Big[ \widetilde{Z}_{[t_1,t]}^{f^-} \big(B^{c_1 , y+c_1-f^-(t) , ( c_1, f^+-f^- +c_1 )}_{[t_1,t]}\big)^{-1} \Big]
\cdot P\left( r^{0\to y-f^-(t)}_{[t_1,t]} \in K^-_{[t_1,t]}(f^+-f^-) \right) \nonumber \\
&\quad \qquad \times 
\sqrt{\frac{2}{\pi (t - t_1)}} \cdot 
\frac{P\left( W^+_{[t_1,t]}(t) \in dy - f^-(t) \right)}{dy}
:=\widetilde{p}^{(f^-,f^+),(\uparrow)}_{[t_1,t]}(y), \qquad \varepsilon \downarrow 0.
\label{WConv_BH_prodtransdensities_L1conv_step2}
\end{align}
On the other hand, using \eqref{Girsanov_pinned_invinv_formula}, we obtain
\begin{align}
&P(B^{y \to c_2}_{[t, t_2]} \in K_{[t, t_2]}(f^-_{\varepsilon}, f^+_{\varepsilon})) 
\cdot \frac{P(W^y_{[t, t_2]}(t_2) \in d c)}{d c}\bigg\vert_{c=c_2} \nonumber \\
&\quad = \frac{P\left(W_{[t, t_2]}^{y} \in K_{[t, t_2]}(f^-_{\varepsilon}, f^+_{\varepsilon}), W_{[t, t_2]}^{y}(t_2) \in dc\right)}{dc}\bigg\vert_{c=c_2} \nonumber \\
&\quad = 
E\left[
\widetilde{Z}_{[t,t_2]}^{-\overleftarrow{f^+_{\varepsilon}}} 
\big(B^{0, c_2-y+f^+_{\varepsilon}(t)-f^+_{\varepsilon}(t_2), 
(c_2-f^+_{\varepsilon}(t_2), \overleftarrow{f^+_{\varepsilon}}-\overleftarrow{f^-_{\varepsilon}}+c_2-f^+_{\varepsilon}(t_2))}_{[t,t_2]}\big)^{-1} 
\right] \nonumber \\
&\qquad \times
P\left( B^{0\to c_2-y+f^+_{\varepsilon}(t)-f^+_{\varepsilon}(t_2)}_{[t,t_2]}
\in K_{[t,t_2]}\big(c_2-f^+_{\varepsilon}(t_2), \overleftarrow{f^+_{\varepsilon}}-\overleftarrow{f^-_{\varepsilon}}+c_2-f^+_{\varepsilon}(t_2)\big)\right)
\nonumber \\
&\qquad \times \frac{P\big(f^+_{\varepsilon}(t_2)-f^+_{\varepsilon}(t)+W^{y}_{[t,t_2]}(t_2) \in dc\big)}{dc}\bigg\vert_{c=c_2} , 
\qquad (f^-_{\varepsilon}(t)<y<f^+_{\varepsilon}(t)).
\label{Appd_Lem_C_proof_inv_1}
\end{align}
For  $f^-_{\varepsilon}(t)<y<f^+_{\varepsilon}(t)$, we have
\begin{align}
&P\left( B^{0\to c_2-y+f^+_{\varepsilon}(t)-f^+_{\varepsilon}(t_2)}_{[t,t_2]}
\in K_{[t,t_2]}\big(c_2-f^+_{\varepsilon}(t_2), \overleftarrow{f^+_{\varepsilon}}-\overleftarrow{f^-_{\varepsilon}}+c_2-f^+_{\varepsilon}(t_2)\big)\right)
\nonumber \\
&\quad = P\left( B^{0\to c_2-y+f^+_{\varepsilon}(t)-f^+_{\varepsilon}(t_2)}_{[t,t_2]} |_{K^+_{[t,t_2]}( c_2-f^+_{\varepsilon}(t_2) )} 
\in K^-_{[t,t_2]}(\overleftarrow{f^+_{\varepsilon}}-\overleftarrow{f^-_{\varepsilon}}+c_2-f^+_{\varepsilon}(t_2))\right) \nonumber \\
&\qquad \times P\left( B^{0\to c_2-y+f^+_{\varepsilon}(t)-f^+_{\varepsilon}(t_2)}_{[t,t_2]} \in K^+_{[t,t_2]}( c_2-f^+_{\varepsilon}(t_2) ) \right) 
\label{Appd_Lem_C_proof_inv_2} 
\end{align}
and
\begin{align}
&P\left( B^{0\to c_2-y+f^+_{\varepsilon}(t)-f^+_{\varepsilon}(t_2)}_{[t,t_2]} \in K^+_{[t,t_2]}( c_2-f^+_{\varepsilon}(t_2) ) \right)  \cdot
\frac{P\big(f^+_{\varepsilon}(t_2)-f^+_{\varepsilon}(t)+W^{y}_{[t,t_2]}(t_2) \in dc\big)}{dc}\bigg\vert_{c=c_2}\nonumber \\
&\quad = \frac{P\left( W_{[t,t_2]} \in K^+_{[t,t_2]}( c_2-f^+_{\varepsilon}(t_2) ) , 
y-f^+_{\varepsilon}(t)+f^+_{\varepsilon}(t_2)+W_{[t,t_2]}(t_2) \in dc \right)}{dc}\bigg\vert_{c=c_2}  \nonumber \\
&\quad = \frac{d}{dc}
P\left( W_{[t,t_2]} \in K^+_{[t,t_2]}( c_2-f^+_{\varepsilon}(t_2) ) , 
y-f^+_{\varepsilon}(t)+f^+_{\varepsilon}(t_2)+W_{[t,t_2]}(t_2) \leq c \right) \bigg\vert_{c=c_2}  \nonumber \\
&\quad = P\left( W_{[t,t_2]} \in K^+_{[t,t_2]}( c_2-f^+_{\varepsilon}(t_2) ) \right)  \nonumber \\
&\qquad \times 
\frac{P\left(y-f^+_{\varepsilon}(t)+f^+_{\varepsilon}(t_2)+W_{[t,t_2]}\vert_{K^+_{[t,t_2]}( c_2-f^+_{\varepsilon}(t_2) )}(t_2) \in dc \right)}{dc} \bigg\vert_{c=c_2}. 
\label{Appd_Lem_C_proof_inv_3}
\end{align}
Then, using \eqref{Appd_Lem_C_proof_inv_1}, \eqref{Appd_Lem_C_proof_inv_2}, \eqref{Appd_Lem_C_proof_inv_3}, 
and Lemma~\ref{Lem_density_ae_convergence},  
we can deduce %%%%%for $f^-(t)<y<f^+(t)$ 
that
\begin{align}
&\frac{ P(B^{y \to c_2}_{[t, t_2]} \in K_{[t, t_2]}(f^-_{\varepsilon}, f^+_{\varepsilon})) }{ f^+_{\varepsilon}(t_2) - c_2 } 
\cdot \frac{P(W^y_{[t, t_2]}(t_2) \in d c)}{d c}\bigg\vert_{c=c_2} 
\nonumber \\
&\quad = 
E\left[
\widetilde{Z}_{[t,t_2]}^{-\overleftarrow{f^+_{\varepsilon}}} 
\big(B^{0, c_2-y+f^+_{\varepsilon}(t)-f^+_{\varepsilon}(t_2), 
(c_2-f^+_{\varepsilon}(t_2), \overleftarrow{f^+_{\varepsilon}}-\overleftarrow{f^-_{\varepsilon}}+c_2-f^+_{\varepsilon}(t_2))}_{[t,t_2]}\big)^{-1} 
\right] \nonumber \\
&\qquad \times
P\left( B^{0\to c_2-y+f^+_{\varepsilon}(t)-f^+_{\varepsilon}(t_2)}_{[t,t_2]} |_{K^+_{[t,t_2]}( c_2-f^+_{\varepsilon}(t_2) )} 
\in K^-_{[t,t_2]}(\overleftarrow{f^+_{\varepsilon}}-\overleftarrow{f^-_{\varepsilon}}+c_2-f^+_{\varepsilon}(t_2))\right) \nonumber \\
&\qquad \times \frac{P\left( W_{[t,t_2]} \in K^+_{[t,t_2]}( c_2-f^+_{\varepsilon}(t_2) ) \right)}{ f^+_{\varepsilon}(t_2) - c_2}  \nonumber \\
&\qquad \times 
\frac{P\left(y-f^+_{\varepsilon}(t)+f^+_{\varepsilon}(t_2)+W_{[t,t_2]}\vert_{K^+_{[t,t_2]}( c_2-f^+_{\varepsilon}(t_2) )}(t_2) \in dc \right)}{dc} \bigg\vert_{c=c_2}
\nonumber \\
&\quad \to
E\left[ \widetilde{Z}^{ (-1) \cdot \overleftarrow{f^+} }_{[t,t_2]}
\Big( r^{ 0 \to f^+(t) - y }_{[t,t_2]}\mid_{K^-_{[t,t_2]}(\overleftarrow{f^+} - \overleftarrow{f^-})} \Big)^{-1} \right] 
\cdot P\left( r^{0 \to f^+(t) - y }_{[t,t_2]} \in  K^-_{[t,t_2]}(\overleftarrow{f^+} - \overleftarrow{f^-}) \right) \nonumber \\
&\qquad \times \sqrt{\frac{2}{\pi (t_2-t)}} 
\cdot \frac{P( W^+_{[t,t_2]}(t_2) \in f^+(t) - dy )}{ dy } 
=:\widetilde{p}^{(f^-,f^+),(\downarrow)}_{[t,t_2]}(y).
\label{WConv_BH_prodtransdensities_L1conv_step3}
\end{align}
By \eqref{WConv_BH_prodtransdensities_L1conv_step2} and \eqref{WConv_BH_prodtransdensities_L1conv_step3}, we can obtain
\begin{align}
&\lim_{\varepsilon \downarrow 0} \int_{\mathbb{R}} \bigg| 
\frac{ P(B^{c_1\to y}_{[t_1, t]} \in K_{[t_1, t]}(f^-_{\varepsilon}, f^+_{\varepsilon})) }{ c_1 - f^-_{\varepsilon}(t_1) } 
\cdot \frac{P(W^{c_1}_{[t_1, t]}(t) \in dy)}{dy} \label{WConv_BH_prodtransdensities_L1conv_step4}\\
&\quad \times 
\frac{P(B^{y \to c_2}_{[t, t_2]} \in K_{[t, t_2]}(f^-_{\varepsilon}, f^+_{\varepsilon})) }{ f^+_{\varepsilon}(t_2) - c_2 } 
\cdot \frac{P(W^y_{[t, t_2]}(t_2) \in d c)}{d c}\bigg\vert_{c=c_2} - 
\widetilde{p}^{(f^-,f^+),(\uparrow)}_{[t_1,t]}(y) 
\widetilde{p}^{(f^-,f^+),(\downarrow)}_{[t,t_2]}(y) \bigg| dy = 0 \nonumber
\end{align}
in a similar manner to the proof of Lemma~\ref{scheffe_pp_intg}. 
Thus, using 
\eqref{weakconv_G_bridge_on_t1t2_touch_at_t1}, 
\eqref{weakconv_G_bridge_on_t1t2_touch_at_t2}, 
\eqref{WConv_BH_prodtransdensities_L1conv_step1} and 
\eqref{WConv_BH_prodtransdensities_L1conv_step4}, we obtain
\begin{align}
I(\varepsilon,F) &:= 
\frac{E\left[F(B^{c_1\to c_2}_{[t_1,t_2]})\ ;\ B^{c_1\to c_2}_{[t_1,t_2]} \in K_{[t_1, t_2]}(f^-_{\varepsilon}, f^+_{\varepsilon})\right] }
{( c_1 - f^-_{\varepsilon}(t_1) ) ( f^+_{\varepsilon}(t_2) - c_2 )}
\cdot \frac{P(W^{c_1}_{[t_1,t_2]}(t_2) \in d c)}{d c}\bigg\vert_{c=c_2} 
\nonumber \\
&= \int_{f^-_{\varepsilon}(t)}^{f^+_{\varepsilon}(t)}  dy
\ E[F(B_{[t_1, t]}^{c_1,y,(f^-_{\varepsilon}, f^+_{\varepsilon})} \oplus_{t} B_{[t, t_2]}^{y, c_2,(f^-_{\varepsilon}, f^+_{\varepsilon})})]
\nonumber \\
&\qquad  \quad \times 
\frac{P(B^{c_1\to y}_{[t_1, t]} \in K_{[t_1, t]}(f^-_{\varepsilon}, f^+_{\varepsilon}))}{ c_1 - f^-_{\varepsilon}(t_1) }
\cdot \frac{P(W^{c_1}_{[t_1, t]}(t) \in dy)}{dy}
\nonumber \\
&\qquad \quad \times 
\frac{P(B^{y\to c_2}_{[t,t_2]} \in K_{[t, t_2]}(f^-_{\varepsilon}, f^+_{\varepsilon}))}{ f^+_{\varepsilon}(t_2) - c_2 }
\cdot \frac{P(W^y_{[t, t_2]}(t_2) \in d c)}{d c}\bigg\vert_{c=c_2}   \nonumber \\
&\to \int_{f^-(t)}^{f^+(t)} 
E[F(B_{[t_1, t]}^{c_1,y,(f^-, f^+)} \oplus_{t} B_{[t, t_2]}^{y, c_2,(f^-, f^+)})] 
\widetilde{p}^{(f^-,f^+),(\uparrow)}_{[t_1,t]}(y) 
\widetilde{p}^{(f^-,f^+),(\downarrow)}_{[t,t_2]}(y) \, dy
\qquad (\varepsilon \downarrow 0) . \label{WConv_BH_prodtransdensities_L1conv_step5}
\end{align}
Therefore, by 
\eqref{WConv_BH_prodtransdensities_L1conv_step5} 
and Theorem~3.1 in \cite{bib_ishitani_hatakenaka_suzuki}, we have
\begin{align*}
&\lim_{\varepsilon \downarrow 0} E\left[F\big(B^{c_1, c_2, (f^-_{\varepsilon}, f^+_{\varepsilon})}_{[t_1,t_2]}\big) \right] 
= \lim_{\varepsilon \downarrow 0} \frac{I(\varepsilon, F)}{ I(\varepsilon, 1) } \nonumber \\
&\quad = \int_{f^-(t)}^{f^+(t)} dy
\ E[F(B_{[t_1, t]}^{c_1,y,(f^-, f^+)} \oplus_{t} B_{[t, t_2]}^{y, c_2,(f^-, f^+)})] 
\widetilde{p}^{(f^-,f^+),(\uparrow)}_{[t_1,t]}(y) 
\widetilde{p}^{(f^-,f^+),(\downarrow)}_{[t,t_2]}(y) \nonumber \\
&\qquad \times \left( \int_{f^-(t)}^{f^+(t)} 
\ \widetilde{p}^{(f^-,f^+),(\uparrow)}_{[t_1,t]}(y) 
\widetilde{p}^{(f^-,f^+),(\downarrow)}_{[t,t_2]}(y) \, dy\right)^{-1} \nonumber \\
&\quad 
=E\left[F\big(B^{c_1, c_2, (f^-, f^+)}_{[t_1,t_2]}\big) \right] .
\end{align*}
Hence, we thus obtain \eqref{weakconv_G_bridge_on_t1t2_touch_at_t1t2}.
\end{proof}

\begin{lem}\label{Lem_conpoundweakconv_2crvs_Meander}
For every $\mathbb{R}$-valued bounded continuous function $F$ on $C([t_1,t_2], \mathbb{R})$,
\begin{itemize}
\item[{\rm ({A})}] if $c=f^-(t_1)>f^-_{\varepsilon}(t_1)$\ \ $(\varepsilon>0)$, then
\begin{align}
& E\left[ F\big(W^{c, (f^-, f^+)}_{[t_1,t_2]}\big) \right] 
= \lim_{\varepsilon \downarrow 0}
E\left[ F\big(W^{c, (f^-_{\varepsilon}, f^+_{\varepsilon})}_{[t_1,t_2]}\big) \right] , 
\label{weakconv_G_meander_on_t1t2}
\end{align}
\item[{\rm ({B})}] if $c=f^+(t_1)<f^+_{\varepsilon}(t_1)$\ \ $(\varepsilon>0)$, then
\begin{align}
& E\left[ F\big(W^{c, (f^-, f^+)}_{[t_1,t_2]}\big) \right] 
= \lim_{\varepsilon \downarrow 0}
E\left[ F\big(W^{c, (f^-_{\varepsilon}, f^+_{\varepsilon})}_{[t_1,t_2]}\big) \right] .
\label{weakconv_G_minus_meander_on_t1t2}
\end{align}
\end{itemize}
\end{lem}
\begin{proof}
Assume that $h^-$ and $h^+$ are $\mathbb{R}$-valued $C^2$-functions defined on $[t_1, t_2]$ satisfying
\begin{align*}
h^-(t_1)<c<h^+(t_1) \quad \mbox{and} \quad
\min_{t\in [t_1,t_2]}(h^+(t)-h^-(t))>0.
\end{align*} 
It follows from Lemma~A.10 in \cite{bib_ishitani_hatakenaka_suzuki} that
\begin{align}
&E[F(W_{[t_1,t_2]}^c)\ ;\ W_{[t_1,t_2]}^c \in K_{[t_1,t_2]}(h^-, h^+)]\nonumber \\
&\quad = 
E\Big[F\big(h^--h^-(t_1)+W_{[t_1,t_2]}^c\big)\widetilde{Z}_{[t_1,t_2]}^{h^-} \big(W_{[t_1,t_2]}^c\big)^{-1} \ ;\ \nonumber \\
&\qquad \qquad 
W_{[t_1,t_2]}^c \in K_{[t_1,t_2]}\big(h^-(t_1), h^+-h^- +h^-(t_1)\big)\Big]
\label{Girsanov_free_formula}\\
&\quad = 
E\Big[F\big(h^+ -h^+(t_1)+W_{[t_1,t_2]}^c\big)\widetilde{Z}_{[t_1,t_2]}^{h^+} \big(W_{[t_1,t_2]}^c\big)^{-1} \ ;\ \nonumber \\
&\qquad \qquad 
W_{[t_1,t_2]}^c \in K_{[t_1,t_2]}\big(h^--h^+ +h^+(t_1), h^+(t_1)\big)\Big].
\label{Girsanov_free_formula_minus}
\end{align}
Then, using \eqref{Girsanov_free_formula} and \eqref{Girsanov_free_formula_minus}, we obtain
\begin{align}
&E\left[F(W_{[t_1,t_2]}^{c, (h^-, h^+)})\right]\nonumber \\
&\quad = 
\frac{E\left[F\big(h^--h^-(t_1)+W_{[t_1,t_2]}^{c, (h^-(t_1), h^+-h^- +h^-(t_1))}\big)
\widetilde{Z}_{[t_1,t_2]}^{h^-} \big(W_{[t_1,t_2]}^{c, (h^-(t_1), h^+-h^- +h^-(t_1))}\big)^{-1} \right] }
{E\left[\widetilde{Z}_{[t_1,t_2]}^{h^-} \big(W_{[t_1,t_2]}^{c, (h^-(t_1), h^+-h^- +h^-(t_1))}\big)^{-1} \right] }
\label{Girsanov_free_formula_3}\\
&\quad =
\frac{E\Big[F\big(h^+ -h^+(t_1)+W_{[t_1,t_2]}^{c, (h^--h^+ +h^+(t_1), h^+(t_1))}\big)
\widetilde{Z}_{[t_1,t_2]}^{h^+} \big(W_{[t_1,t_2]}^{c, (h^--h^+ +h^+(t_1), h^+(t_1))}\big)^{-1} \Big]}
{E\Big[\widetilde{Z}_{[t_1,t_2]}^{h^+} \big(W_{[t_1,t_2]}^{c, (h^--h^+ +h^+(t_1), h^+(t_1))}\big)^{-1} \Big]}.
\label{Girsanov_free_formula_minus_3}
\end{align}
Using \eqref{Girsanov_free_formula_3} and
Proposition~\ref{Prop_Boundary_purturb_WeakConv_for_Meander_BES3bridge} 
\eqref{Prop_Boundary_purturb_WeakConv_for_Meander}, 
we obtain
\eqref{weakconv_G_meander_on_t1t2}. 
Similarly, using \eqref{Girsanov_free_formula_minus_3} and 
Proposition~\ref{Prop_Boundary_purturb_WeakConv_for_Meander_BES3bridge} 
\eqref{Prop_Boundary_purturb_WeakConv_for_Meander}, 
we obtain
\eqref{weakconv_G_minus_meander_on_t1t2}.
\end{proof}

%%%%%%%%%%%%%%%%%%%%%%%%%%%%%%%%%%%%%%%%%%%%%
\subsection{Preparation for the proof of Theorem~\ref{Thm_abs_conti}}
\label{subsection_proof_Preparation_for_absconti}
%%%%%%%%%%%%%%%%%%%%%%%%%%%%%%%%%%%%%%%%%%%%%

\begin{lem}\label{Lem_abs_conti_B_motion_and_B_bridge}
Let $a, c \in \mathbb{R}$. For $t\in(0, 1)$, we have
\begin{align*}
\frac{d\left(P\circ (\pi_{[0,t]}\circ X_{[0,1]}^{a\to c})^{-1}\right)}
{d\left(P\circ (\pi_{[0,t]}\circ (X_{[0,1]}^a))^{-1}\right)}(w)
=\frac{q_X(1-t, w(t), c)}{q_X(1, a, c)}, \quad w\in C([0, t], \mathbb{R}).
\end{align*}
\end{lem}
\begin{proof}
In this proof, let $P^V$ denote 
the measure induced by a continuous process $V=\{V(t)\}_{t\in[0, 1]}$. 
In addition, for a continuous process $V=\{V(t)\}_{t\in[0, 1]}$, 
we write the expectation with respect to the probability $P^V$ as $E^V$.
Let $A\in \mathcal{B}(C([0, t],\mathbb{R}))$ be fixed. 
By the Markov property of $X_{[0,1]}^a$, we obtain the assertion as follows: 
\begin{align*}
P\left(\pi_{[0,t]}\circ X_{[0,1]}^{a\to c} \in A\right)
&=\frac{P^{X_{[0,1]}^a}\left(\pi_{[0, t]}^{-1}(A), w(1)\in dc\right)}{P^{X_{[0,1]}^a}\left(w(1)\in dc\right)}\\
&=\frac{E^{X_{[0,1]}^a}\left[1_{\pi_{[0, t]}^{-1}(A)}(w)\cdot P^{X_{[0,1]}^a}\left(w(1)\in dc~|~w(t)\right)\right]}{P^{X_{[0,1]}^a}\left(w(1)\in dc\right)}\\
&=\int_{\pi_{[0, t]}^{-1}(A)}\frac{P^{X_{[0,1]}^a}\left(w(1)\in dc~|~w(t)\right)}{P^{X_{[0,1]}^a}\left(w(1)\in dc\right)}P^{X_{[0,1]}^a}\left(dw\right)\\
&=\int_{A}\frac{q_X(1-t, w(t), c)}{q_X(1,a,c)}P\left( \pi_{[0,t]}\circ (X_{[0,1]}^a)\in dw\right).
\end{align*}
\end{proof}

\section*{Acknowledgments}

The authors would like to thank 
Prof.\ Ryozo Miura (Hitotsubashi University), 
Prof.\ Toshihiro Yamada (Hitotsubashi University),  
Prof.\ Masaaki Fukasawa (Osaka University), 
Prof.\ Tomonori Nakatsu (Shibaura Institute of Technology), and  
Prof.\ Dai Taguchi (Kansai University) 
for their helpful comments and discussions on the subject matter. 
%%%We also thank Editage (www.editage.com) for English language editing. 
This study was supported by a JSPS KAKENHI grant (JP25K05168) 
and the MIYAKO-MIRAI Project of Tokyo Metropolitan University.
%%%Finally, we would like to express our sincere gratitude to the anonymous reviewers 
%%%for their valuable comments and suggestions, which have significantly contributed to improving the quality of this manuscript.

%\newpage
\begin{flushleft}
\mbox{  }\\
\hspace{95mm} Kensuke Ishitani\\
\hspace{95mm} Department of Mathematical Sciences\\
\hspace{95mm} Tokyo Metropolitan University\\
\hspace{95mm} Hachioji, Tokyo 192-0397\\
\hspace{95mm} Japan\\
\hspace{95mm} e-mail: k-ishitani@tmu.ac.jp\\
\mbox{  }\\
\hspace{95mm} Soma Nishino \\
\hspace{95mm} Department of Mathematical Sciences\\
\hspace{95mm} Tokyo Metropolitan University\\
\hspace{95mm} Hachioji, Tokyo 192-0397\\
\hspace{95mm} Japan
\end{flushleft}

\end{document}